\documentclass[12pt]{amsart}

\usepackage{hyperref}
\usepackage{amsfonts,amssymb,amscd,amsmath,amsthm,enumerate,mathrsfs}
\usepackage{graphicx,type1cm,eso-pic}

\newcommand{\abs}{\vspace{6pt}}

\DeclareMathOperator{\vol}{vol}

\DeclareMathOperator{\grad}{grad}

\DeclareMathOperator{\supp}{supp}
\DeclareMathOperator{\Rec}{Rec}
\DeclareMathOperator{\id}{id}
\DeclareMathOperator{\Int}{Int}

\DeclareMathOperator{\const}{const.}

\DeclareMathOperator{\Iso}{Iso}

\DeclareMathOperator{\Per}{Per}

\DeclareMathOperator{\card}{card}
\DeclareMathOperator{\Hor}{Hor}
\DeclareMathOperator{\Gro}{Gro}
\DeclareMathOperator{\Span}{span}
\DeclareMathOperator{\Dom}{Dom}

\DeclareMathOperator{\Fix}{Fix}
\DeclareMathOperator{\dist}{dist}

\DeclareMathOperator{\la}{\langle}
\DeclareMathOperator{\ra}{\rangle}
\DeclareMathOperator{\h}{\mathit{h}_{top}}
\newcommand{\skp}{{\langle .,. \rangle}}

\DeclareMathOperator{\hc}{hc}
\DeclareMathOperator{\Aff}{Aff}

\newcommand{\R}{\mathbb{R}}
\newcommand{\Z}{\mathbb{Z}}
\newcommand{\Q}{\mathbb{Q}}
\newcommand{\N}{\mathbb{N}}
\newcommand{\C}{\mathbb{C}}

\renewcommand{\H}{\mathcal{H}}

\newcommand{\NN}{\mathcal{N}}

\newcommand{\MM}{\mathfrak{M}}

\renewcommand{\O}{\mathcal{O}}

\newcommand{\F}{\mathcal{F}}
\newcommand{\TT}{\mathcal{T}}

\newcommand{\RR}{\mathcal{R}}

\newcommand{\M}{\mathcal{M}}

\newcommand{\W}{\mathcal{W}}
\newcommand{\A}{{\mathcal{A}}}

\newcommand{\T}{{\mathbb{T}^2}}
\newcommand{\Tn}{{\mathbb{T}^n}}

\newcommand{\s}{{\mathbb{S}^2}}
\newcommand{\se}{{\mathbb{S}^1}}

\newcommand{\e}{\varepsilon}
\newcommand{\vf}{\varphi}

\newcommand{\g}{\gamma}
\newcommand{\Gam}{\Gamma}
\newcommand{\al}{\alpha}
\newcommand{\om}{\omega}
\newcommand{\lam}{\lambda}
\newcommand{\Lam}{\Lambda}
\newcommand{\sig}{\sigma}

\theoremstyle{plain}
\newtheorem{defn}{Definition}[section]
\newtheorem{lemma}[defn]{Lemma}
\newtheorem{prop}[defn]{Proposition}
\newtheorem{thm}[defn]{Theorem}
\newtheorem{cor}[defn]{Corollary}
\newtheorem{prob}[defn]{Problem}

\theoremstyle{definition}
\newtheorem{remark}[defn]{Remark}

\newcommand{\mane}{Ma\~n\'e}
\newcommand{\Mane}{Ma\~n\'e }

\newcommand{\Poincare}{Poincar\'e }

\makeatletter
\def\@tocline#1#2#3#4#5#6#7{\relax
  \ifnum #1>\c@tocdepth % then omit
  \else
    \par \addpenalty\@secpenalty\addvspace{#2}%
    \begingroup \hyphenpenalty\@M
    \@ifempty{#4}{%
      \@tempdima\csname r@tocindent\number#1\endcsname\relax
    }{%
      \@tempdima#4\relax
    }%
    \parindent\z@ \leftskip#3\relax \advance\leftskip\@tempdima\relax
    \rightskip\@pnumwidth plus4em \parfillskip-\@pnumwidth
    #5\leavevmode\hskip-\@tempdima
      \ifcase #1
       \or\or \hskip 1em \or \hskip 2em \else \hskip 3em \fi%
      #6\nobreak\relax
    \dotfill\hbox to\@pnumwidth{\@tocpagenum{#7}}\par % <---- \dotfill -> \hfill
    \nobreak
    \endgroup
  \fi}
\makeatother

\begin{document}

\hypersetup{pdftitle = {Minimal geodesics and integrable behavior in geodesic flows}, pdfauthor = {Jan Philipp Schr\"oder}}

\title[Minimal geodesics and integrability]{Minimal geodesics and integrable behavior in geodesic flows}
\author[J. P. Schr\"oder]{Jan Philipp Schr\"oder}
\address{Faculty of Mathematics \\ Ruhr University \\ 44780 Bochum \\ Germany}
\email{\url{jan.schroeder-a57@rub.de} \newline \url{http://www.ruhr-uni-bochum.de/ffm/Lehrstuehle/Lehrstuhl-X/jan.html}}
%\date{\today}
%\keywords{}
%\subjclass[2010]{Primary, Secondary}

\begin{abstract}
 In this survey article we gather classical as well as recent results on minimal geodesics of Riemannian or Finsler metrics, giving special attention to the two-dimensional case. Moreover, we present open problems together with some first ideas as to the solutions.
\end{abstract}

\maketitle

\section{Introduction and structure of this paper}

When Poincar\'e first studied the stability of the solar system in the 19th century, he studied the dynamics of what is nowadays known as a monotone twist map of an annulus. A central step was to find reasonably nice sets in phase space that are left invariant under the dynamical system in question. These sets were remnants of integrable behavior of the dynamical system, the approach leading in particular to the development of KAM theory, as well as Aubry-Mather theory.

It was found that Poincar\'e's monotone twist maps are related to geodesic flows on the 2-torus. In this setting, the geodesic flow of a Riemannian metric of constant curvature on a closed, orientable surface can be seen as an ``integrable'' prototype of a dynamical system. Morse and Hedlund observed that minimal geodesics with respect to any Riemannian metric resemble in many senses the ``integrable'' behavior of the constant curvature case.

Basic results on minimal geodesics date back to the 1920ies, while some were established recently. Studying the structure of the geodesic flow, minimal geodesics have proven to be useful and hence can be the first step towards understanding geodesic flows as a whole. In this survey, we will summarize results on the structure of minimal geodesics, with special attention towards the 2-dimensional case. Central topics include asymptotic notions, such as asympotic directions of minimal geodesics, the stable norm and Mather's average action, the Gromov and horofunction boundaries, weak KAM solutions, chaotic behavior, transitivity and topological entropy. Along the way, we present some indications for further research, some of which might be suitable for PhD projects.

We include proofs in this paper if they are simple and illustrate the ideas involved or because the proofs have not yet appeared in the literature. Please be aware that it is not the aim of the author to give a complete overview on the literature.

%In Sections \ref{section basics} and \ref{section morse lemma} we recall basic definitions and fundamental properties of minimal geodesics and define asympotic directions. Sections \ref{section weak KAM} and \ref{section mather theory} develop a general theory of invariant graphs in relation to minimal geodesics, being known as weak KAM theory and Mather theory. In Sections \ref{section rational} and \ref{section irrational} we study the dynamics of minimal geodesics under the assumption of rational and irrational asymptotic directions. Section \ref{section applications} applies the results found in Sections \ref{section rational} and \ref{section irrational} to the theories developed in Sections \ref{section morse lemma}, \ref{section weak KAM} and \ref{section mather theory}. Questions regarding topological entropy are discussed in Section \ref{section entropy}. In Section \ref{section misc} we discuss Birkhoff's theorem on regions of instability and present two results on surfaces of non-positive curvature. Finally, in Section \ref{section rel dynamical systems} we show, how Finsler geodesic flows are related to monotone twist maps, Tonelli Lagrangians and symplectic geometry.

\tableofcontents

\section{First examples, setting and basic definitions}\label{section basics}

\subsection{An example}\label{section rot torus}

We briefly discuss the dynamics of the geodesic flow of the rotational torus embedded in $\R^3$. In the example of the rotational torus, many phenomena in the study of minimal geodesics already occur.

One can show that the rotational torus is isomorphic to $\T=\R^2/\Z^2$ endowed with a Riemannian metric $g'$ of the form
\[ g'_{(x_1,x_2)}(v,w) = f(x_2) \cdot g(v,w) \]
with some smooth, positive function $f:\R/\Z\to\R$, $g$ being the Euclidean metric on $\R^2$ and the quotient $\T$. The function $f$ has precisely two extrema, one minimum corresponding to the short inner loop and one maximum corresponding to the longer outer loop. Rewriting $g'$ as a Jacobi metric $g'=(e-V)g$, i.e.\ setting
\[ f(x_2) = e-V(x_2), \qquad e := \max f, \quad V(x_2) := e - f(x_2) \geq 0 , \]
we find that the geodesics of $g'$ are reparametrized solutions $c:\R\to\T$ of the Euler-Lagrange equation
\[ (c_1'',c_2'')=\ddot c = - \grad_g V(c) = (0, -V'(c_2)) \]
with total energy $E(c,\dot c)=\frac{1}{2}|\dot c|_g^2 + V(c) = e$.

In this way, one sees that the geodesics $c$ of $g'$ have a linear $c_1$-component and behave like solutions to the pendulum-like equation $x''=-V'(x)$ in the $c_2$-component. In particular, the geodesics $c$ {\em stay within finite distance of a straight line}, when lifted to the universal cover $\R^2$ of $\T$. Drawing a transverse figure of the geodesic flow, one obtains the well-known phase portrait of the simple pendulum: There are smooth, {\em invariant graphs} over the $x$-axis and in the middle there is a fixed point (the {\em short inner geodesic loop}) connected to itself by {\em homoclinic orbits} forming a pair of continuous, invariant graphs, surrounding an elliptic fixed point (the outer long geodesic loop) and orbits oscillating around it.

\subsection{Basic definitions}

We fix some notation. Throughout, we fix a connected, simply connected, smooth manifold $H$ of arbitrary dimension and endow $H$ with a fixed, complete Riemannian metric $g$. ($H$ will usually be a Hadamard manifold, hence the letter $H$.) We write $|v|_g=\sqrt{g(v,v)}$ for the norm of the Riemannian metric $g$. We let $\Gam$ be a discrete group of $g$-isometries acting freely on $H$, which is assumed to define a smooth quotient manifold
\[ M = H/\Gam \]
of the same dimension as $H$. In particular, $H$ is the universal cover of $M$ and $\Gam$ is isomorphic to the fundamental group $\pi_1(M)$ of $M$. The boundary is automatically $\partial M=\emptyset$ by the completeness of the $g$-geodesic flow. The covering map is denoted by
\[ p: H \to M , \]
reserving the letter $\pi$ for the canonical projections of the (co)tangent bundles
\[ \pi : TH, T^*H\to H, \qquad \pi : TM, T^*M\to M . \]
Periodic curves in $M$ in some homotopy class $\tau\in \Gam$ correspond to curves $c:\R\to H$ satisfying
\[ \exists T>0 : \qquad \tau \circ c(t+T) = c(t) . \]
Such curves will be termed $\tau$-periodic.

We shall work with Finsler metrics, thus incorporating monotone twist maps of the annulus, as well as Riemannian metrics and general Tonelli Lagrangians, see Section \ref{section rel dynamical systems}. We recall the definition of a Finsler metric; for an introduction to Finsler geometry see \cite{BCS}.

\begin{defn} \label{def finsler}
 A \emph{Finsler metric} on $H$ is a function $F:TH \to \R$ with the following properties:
 \begin{enumerate}[(i)]
  \item $F$ is $C^\infty$ in $TH-0_H$, $0_H$ being the zero section,
  
  \item $F(v)\geq 0$ for all $v$,
  
  \item $F$ is positively homogeneous of degree one, i.e. $F(a v)=a F(v)$ for $a\geq 0$,
  
  \item (strong convexity) there exists a constant $c_F \geq 1$, so that if $v\in TH-0_H$ and $w\in T_{\pi v}H$, then
  \[ \frac{1}{c_F^2} \cdot |w|_g^2 \leq \frac{d^2}{dt^2}\bigg|_{t=0} \tfrac{1}{2} F^2(v+tw) \leq c_F^2 \cdot |w|_g^2. \]
 \end{enumerate}
 $F$ is called {\em reversible}, if $F(-v)=F(v)$ and {\em Riemannian}, if $F(v)=\sqrt{g'(v,v)}$ is the norm of some Riemannian metric $g'$.
\end{defn}

The reader unfamiliar with Finsler metrics may assume throughout that $F$ is Riemannian.

We fix in this paper the Finsler metric $F$ on $H$, assume that $F$ is invariant under $\Gam$, so that both $F$ and $g$ descend to metrics on the quotient $M$, denoted again by $F,g$. Observe that $F$ is uniformly equivalent to the norm of $g$. Namely, taking $w=v$ in the definition of strong convexity, we have
\[ \frac{1}{c_F} \cdot |v|_g \leq F(v) \leq c_F \cdot |v|_g . \]
The unit tangent bundles $\{F=1\}\subset TM,TH$, respectively, are denoted by $S_FH, S_FM$ and the geodesic flow of $F$ will be denoted by $\phi_F^t$. Given a vector $v\in TH$ or $TM$, the $F$-geodesic with initial condition $\dot c(0)=v$ of constant speed is denoted by $c_v$, so that $\phi_F^tv=\dot c_v(t)$. Similarly, the $g$-geodesic with initial condition $\dot \g(0)=v$ is denoted by $\g_v$. Throughout, $F$-geodesics will carry Latin letters, while $g$-geodesics will carry Greek letters. The Finsler metric $F$ induces a length structure
\[ l_F(c;[a,b]) = \int_a^b F(\dot c) dt \]
on the space of continuous, piecewise $C^1$-curves $c:[a,b]\to H$ or $M$. Using this length, we can define the $F$-distance
\[ d_F(x,y) = \inf\{ l_F(c;[0,1]) : c(0)=x,c(1)=y \} \]
on $H$ and $M$. Note that the metrics $d_F$ and $d_g$ are equivalent by the same constant $c_F$ as $F$ and $g$.

We shall be treating the geodesic flow of $g$ as given and the dynamics of the $F$-geodesics as unknown. In the example of the rotational torus in Subsection \ref{section rot torus}, $(H,g)$ is the Euclidean plane, $\Gam$ is isomorphic to $\Z^2$ and $F=|.|_{g'}=\sqrt f\cdot |.|_g$ is the norm obtained from the embedding as the rotational torus in the Euclidean space $\R^3$. Already in this example, KAM-tori occur.

\begin{defn}\label{def KAM}
 Let $\phi^t:X\to X$ be a $C^\infty$-flow on a $C^\infty$-manifold $X$. A {\em $C^k$-KAM-torus} (of dimension $n$) is a $C^k$-submanifold $\TT\subset X$, such that
 \begin{enumerate}[(i)]
  \item $\TT$ is invariant under $\phi^t$,
  
  \item there exists a $C^k$-diffeomorphism $\TT\to \mathbb{T}^n = \R^n/\Z^n$ conjugating $\phi^t|_\TT$ to a linear flow $\psi^t$ on $\mathbb{T}^n$ of the form
  \[ \psi^t(x) = x+t\rho \mod \Z^n . \]
 \end{enumerate}
 We call the flow $\phi^t$ {\em integrable}, if $X$ is filled up to a set of zero Lebesgue measure by $C^\infty$-KAM-tori.
\end{defn}

The smooth invariant graphs in the case of the rotational torus are all $C^\infty$-KAM-tori. The rotational torus yields an example of an integrable geodesic flow. To be precise, we also make the following definition.

\begin{defn}
 A {\em graph (over $M,H$)} in $TM,T^*M, TH,T^*H$, respectively, is a subset $G$, so that the canonical projection $\pi|_G$ is injective.
\end{defn}

\subsection{Minimal geodesics}\label{section min geod}

Write $\R_-=(-\infty,0]$ and $\R_+=[0,\infty)$.

\begin{defn}
 An arc-length $C^1$-curve $c:\R_-,\R\to H$ with $t=d_F(c(0),c(t))$ for all $t\in \R_-,\R$ is called a {\em ray, minimal geodesic}, respectively. The sets of initial conditions $v\in S_FH$ of rays, minimal geodesics are denoted by $\RR_-,\M\subset S_FH$, respectively. We shall also call geodesics $c:\R_-,\R\to M$ rays or minimal, if they lift to rays or minimal geodesics in the universal cover $H$. The sets of initial conditions are in this case the projections $Dp(\RR_-),Dp(\M)\subset S_FM$.
\end{defn}

We defined rays and minimal geodesics in the universal cover. Equivalently, a geodesic $c:[a,b]\to M$ lifts to a minimal segment in $H$, if for all $a\leq s<t\leq b$ and all curves $c':[0,1]\to M$ with $c'(0)=c(s),c'(1)=c(t)$, which are homotopic in $M$ to $c|_{[s,t]}$ have $l_F(c')\geq l_F(c)$. Hence, rays and minimal geodesics in $M$ are homotopically minimizing.

Everything that we say will of course have an analgon to rays defined on $\R_+$. In order to avoid confusion, rays in this paper will be defined on $\R_-$. Rays defined on $\R_+$ will be termed forward rays, defining a set $\RR_+\subset S_FH$. Note that we often drop mentioning $F$ when speaking of minimal segments. In our applications, $g$ will usually be of non-positive curvature, so any $g$-geodesic is automatically minimal in $H$.

A typical assumption will be that $M$ is orientable. Note that results on rays and minimal geodesics in the orientable case carry over directly to the non-orientable case, as they are defined in the universal cover. If $M$ is non-orientable, one can move to the orientable double cover $\tilde M$ of $M$ and write $\tilde M$ as a quotient $H/\tilde\Gam$ with $\tilde\Gam\subset\Gam$.

Two central questions govern this paper.
\begin{enumerate}[(A)]
 \item\label{central quest A} How does the restricted geodesic flow $\phi_F^t|_\M$ resemble the geodesic flow $\phi_g^t$, i.e.\ being given some structure of $\phi_g^t$, what is the structure of $\phi_F^t|_\M$? Here, the flow $\phi_g^t$ takes the place of a simple, ``integrable'' model of $\phi_F^t$.
 
 \item\label{central quest B} What are the implications for the whole geodesic flow $\phi_F^t$, and how large is the set $\M\subset S_FH$?
\end{enumerate}

A large part of this paper will be concerned with the intersection properties of rays. Let us be precise.

\begin{defn}\label{def transverse intersection}
 Let $c:I\to H, c':I'\to H$ be two curves defined on intervals $I,I'$. Then $c,c'$ have {\em successive intersections}, if
 \[ \exists s<t \text{ in } I , s'<t' \text{ in } I' : \qquad   c(s)=c'(s') \text{ and } c(t)=c(t') . \]
 Let $c,c':\R_-\to H$ be two rays. Then $c'$ {\em intersects} $c$, if $c'(s)=c(t)$ for some $s\leq 0$ and some $t<0$, but $\dot c'(s)\neq \dot c(t)$.
\end{defn}

The necessity for considering {\em successive} intersection stems from the Finsler metric $F$ not being reversible in general.

The following lemma excludes in particular successive intersections of rays and shows that asymptotic rays cannot intersect in the sense of Definition \ref{def transverse intersection}. The idea of the proof is classical, see Theorem 6 in \cite{morse}. Another proof can be found in \cite{paper1}, Lemma 2.20.

\begin{lemma}\label{crossing minimals}
Let $v_n,v,w_n, w\in \RR_-$ with $v_n\to v, w_n\to w$, $\pi w=c_v(a)$ for some $a<0$, but $w\neq \dot c_v(a)$. Then for all $\delta>0$ and sufficiently large $n$
\[ \inf \{ d_g(c_{v_n}(s),c_{w_n}(t)) : s\in (-\infty,a], t\in (-\infty,-\delta] \} > 0. \]
\end{lemma}

\subsection{Model geometries}\label{section model geom}

We discuss here some classical instances of the Riemannian manifold $(M,g)$. We will in this paper mainly be concerned with the 2-dimensional case $\dim M=2$. A typical assumption will be that $M$ is moreover orientable and compact, hence closed. Then $M$ is a handle body, that is the connected sum of the 2-sphere $\s$ and $\mathfrak{g}$ 2-tori $\T$. The integer $\mathfrak{g} \geq 0$ is the genus of $M$. There exists a Riemannian metric $g$ on $M$ of constant curvature $K_g\in \{-1,0,1\}$. In the case $\mathfrak{g}=0$, the universal cover $H=\s$ is compact and there are no rays or minimal geodesics at all. Hence, we will assume that $\mathfrak{g}\geq 1$.

\abs

If $\mathfrak{g}=1$, then $M$ is the 2-torus $M=\T=\R^2/\Z^2$ and $K_g=0$. We take $H=\R^2$ and $g$ the Euclidean metric.

The $g$-geodesics in $\R^2$ are straight lines $\g_v(t) = \pi v + t v$. Here we find invariant graphs for the geodesic flow, i.e.\ in coordinates $T\R^2=\R^2\times \R^2$ the sets
\[ \Sigma_\xi:= \{ (x,\xi) : x  \in \R^2 \} \subset S_g\R^2 \]
for $\xi\in \se=\{x\in\R^2:|x|=1\}$ are invariant under the $g$-geodesic flow. Observe that all geodesics $\g_v$ with $v\in \Sigma_\xi$ have the same (asymptotic) direction $\xi\in \se$. 

The group $\Gam=\Z^2$ acts by translation $\tau(x) = x + \tau$ for $\tau\in\Z^2$. The direction of $g$-geodesics is invariant under the action of $\tau\in\Gam$ and the sets $\Sigma_\xi$ descend to $\phi_g^t$-invariant graphs $Dp(\Sigma_\xi)\subset S_g\T$. Moreover, the projections $Dp(\Sigma_\xi)$ are KAM-tori for $\phi_g^t$ in the sense of Definition \ref{def KAM}, so that $\phi_g^t:S_g\T\to S_g\T$ is an integrable geodesic flow. If $\xi=(\xi_1,\xi_2)$ has rational or infinite slope $\xi_2/\xi_1\in \Q\cup \{\infty\}$, then all geodesics in $Dp(\Sigma_\xi)$ are closed under $\phi_g^t$. If, on the other hand $\xi$ has irrational slope $\xi_2/\xi_1\in \R-\Q$, then any geodesic $\g_v(\R)\subset \T$ with $v\in Dp(\Sigma_\xi)$ is dense. Motivated by this duality, we will consider the set
\[ \Pi_\T := \{ \xi \in \se : \xi_2/\xi_1\in \Q\cup \{\infty\} \} \]
to be the directions of periodic $g$-geodesics.

\abs

Let us now assume $\mathfrak{g}\geq 2$. Then $K_g=-1$ and we take for $(H,g)$ the \Poincare disc model of the hyperbolic plane. That is,
\[ H=\{x\in\R^2:|x|<1\} \]
is the open unit disc endowed with the \Poincare metric
\[ g_x(v,w) = 4\cdot(1-|x|^2)^{-2}\la v,w\ra ,\]
writing $\skp,|.|$ for the Euclidean metric in $\R^2\supset H$.

The $g$-geodesics are circle segments in $\R^2$ meeting the boundary $\se$ orthogonally. In particular, $g$-geodesics $\g:\R\to H$ have two endpoints $\g(-\infty), \g(\infty) \in \se$. In $S_gH$ we find $\phi_g^t$-invariant graphs of the form
\[ \Sigma_\xi^- := \{ v\in S_gH : \g_v(-\infty) = \xi \} .\]
Hence, the geodesics from $\Sigma_\xi^-$ initiate in $\xi\in \se$ and spread all over $H$. Geodesics $\g_v,\g_w$ with $v,w\in\Sigma_\xi^-$ are asymptotic in the sense that
\[ \lim_{t\to\infty} d_g(\g_v(\R_-),\g_w(-t)) = 0 . \]
In particular, the sets $\g_v(\R_-),\g_w(\R_-)$ have finite Hausdorff distance. (In the literature, $\g_v,\g_w$ are sometimes called asymptotic, if the Hausdorff distance is finite. We shall reserve the term ``asymptotic'' for the strong sense above.) 

The group $\Gam$ consists of isometries of $(H,g)$. The group $\Iso^+(H,g)$ of orientation-preserving isometries consists of M\"obius transformations $\tau:H\to H$ of the form
\[ \tau(x) = \frac{ax+b}{\bar b x+ \bar a} , \qquad |a|^2 - |b|^2 = 1 \]
(using the complex structure on $\R^2\cong\C \supset H$). Using this formula, each element of $\Iso^+(H,g)$ extends to the boundary $\se$ and each element in $\Iso^+(H,g)-\{\id\}$ has either one or two fixed points in $H\cup \se$. The group $\Gam$ acts freely on $H$, so that we exclude the case of a fixed point in $H$ (such isometries are called elliptic). Then the set $\Gam-\{\id\}$ splits into {\em parabolic} elements having one fixed point in $\se$ and {\em hyperbolic} elements having two distinct fixed points in $\se$. We write 
\[ \Gam-\{\id\} = \Gam_{par} ~ \dot\cup ~ \Gam_{hyp}. \]
The parabolic elements correspond to cusps of infinite length in the quotient $M=H/\Gam$, while a hyperbolic element $\tau$ defines (up to orien\-ta\-tion-preserving reparametrization) a unique arc-length $g$-geodesic $\g_\tau:\R\to H$ by the property
\[ \exists T>0 : \qquad \tau \circ \g_\tau(t+T) = \g_\tau(t) . \]
Hence, $p\circ \g_\tau:\R\to M$ is the unique closed geodesic in the homotopy class $\tau$. The fixed points of $\tau$ are given by the endpoints $\g_\tau(\pm\infty)$ and by our convention, $\g_\tau(-\infty)$ is stable, while $\g_\tau(\infty)$ is unstable. See Figure \ref{fig_mobius} for the action of $\Gam$ corresponding to the topology of $M$. Not assuming $M$ to be compact, but only $(M,g)$ to be an orientable, complete Riemannian surface with curvature $K_g\equiv -1$, the group $\Gam$ can contain parabolic elements. If we use the assumption that $M$ is compact, then $M$ cannot have cusps, i.e.\ $\Gam_{par}=\emptyset$.

\begin{figure}\centering%[!htb]
 \includegraphics[scale=0.5]{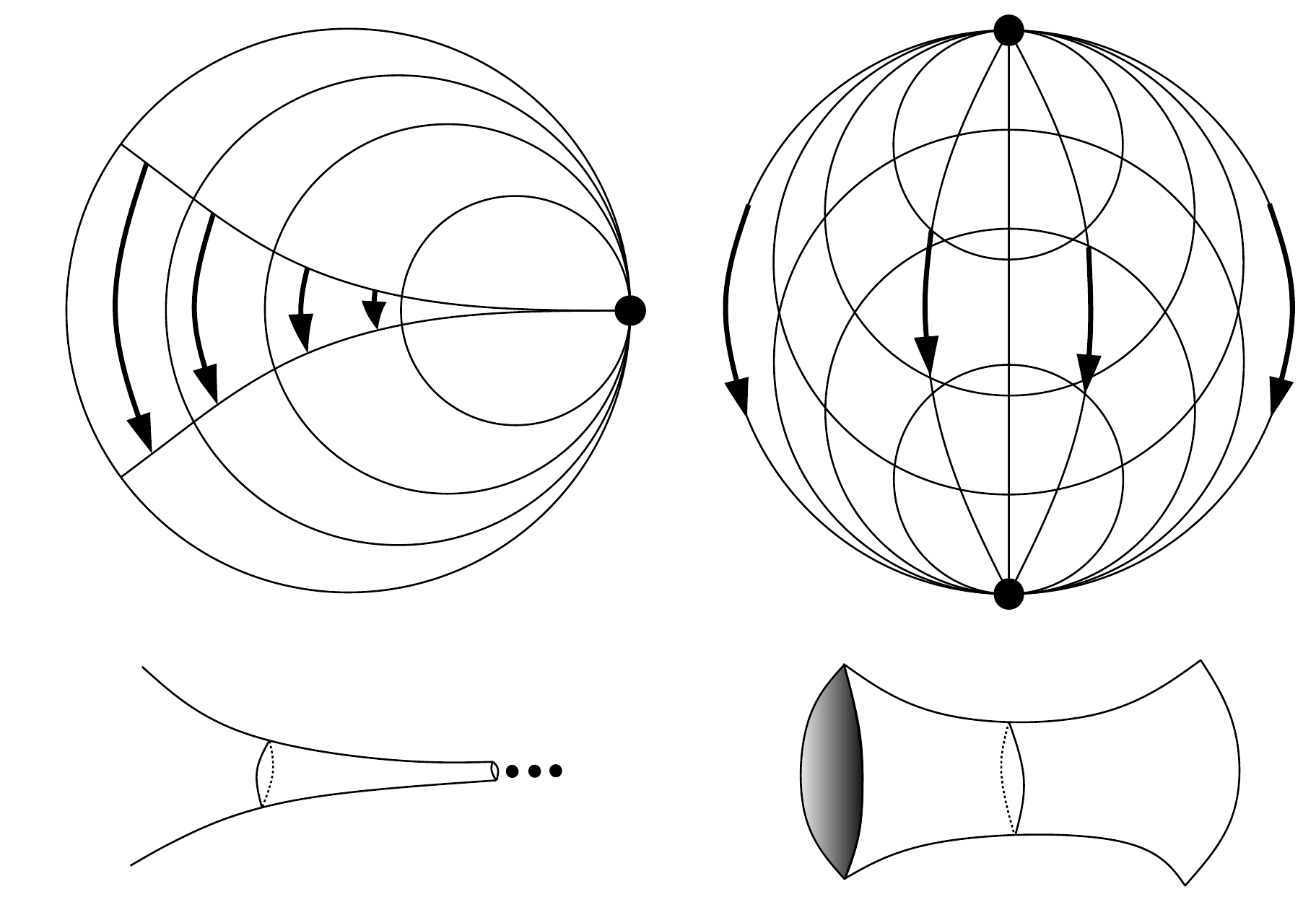}
 \caption{Top: The action of $\tau\in\Gam$ in the \Poincare disc $H$. Bottom: The situation in $M$. Left: The action of a parabolic isometry $\tau\in\Gam$ yields a cusp in $M$. $\tau$ rotates along the horocycles attached at the fixed point in $\se$. Right: The action of a hyperbolic isometry $\tau\in\Gam$ yields a closed geodesic in $M$. $\tau$ translates along the geodesic connecting the fixed points in $\se$ and expands (contracts) the horocycles attached at the unstable (stable) fixed point. \label{fig_mobius}}
\end{figure}

The foliation of $S_gH$ by the sets $\Sigma_\xi^-$ could be seen as integrability in the universal cover. However, the end points $\g(-\infty)$ and hence the graphs $\Sigma_\xi^-\subset S_gH$ are not invariant under the group $\Gam$, as was the case for $\mathfrak{g}=1$. Hence, the sets $\Sigma_\xi^-$ do not project to KAM-tori for $\phi_g^t$ in $M$ the sense of Definition \ref{def KAM}. It is well-known that the geodesic flow $\phi_g^t$ is ergodic with respect to the canonical Liouville measure in $S_gM$, if $(M,g)$ has finite volume (in particular, if $M$ is compact). In particular, almost every orbit $\dot \g:\R\to S_gM$ is dense in $S_gM$ and the geodesic $\g:\R\to M$ has many self-intersections. The set $Dp(\Sigma_\xi^-)\subset S_gM$ is never a graph over $M$.

Again we define a set $\Pi_M$ by
\[ \Pi_M := \{ \xi\in \se : \exists \tau\in \Gam_{hyp} \text{ with } \tau\xi = \xi \}. \]
Hence, $\Pi_M$ corresponds again to the endpoints of geodesics projecting to closed geodesics via $p:H\to M$. For $\xi\in\Pi_M$ the sets $\Sigma_\xi^-$ are invariant precisely under the infinite cyclic subgroup  of $\Gam$ generated by the elements fixing $\xi$ (see the Preissmann theorem in Subsection \ref{section hadamard}).

The following definition is motivated by the case of the 2-torus.

\begin{defn}
 If $M=H/\Gam$ is a closed, orientable surface of genus $\mathfrak{g}\geq 1$, then the points in $\Pi_M \subset \se$ are called {\em rational}. The points in $\se-\Pi_M$ are called {\em irrational}.
\end{defn}

Of course, if the genus of $M$ is $\mathfrak{g}\geq 2$, then the points in $\Pi_M$ will not have rational slope.

\subsection{Hadamard manifolds}\label{section hadamard}

A Hadamard manifold is a connected, simply connected manifold $H$ endowed with a complete Riemannian metric $g$ of non-positive curvature. We briefly recall four theorems in this setting. Firstly, the topology of $H$ is simple.

\begin{thm}[Hadamard-Cartan]
 $H$ is diffeomorphic to $\R^{\dim H}$.
\end{thm}

Secondly, geodesics with finite Hausdorff distance are excluded by negative curvature. More precisely, one has the following.

\begin{thm}[Flat Strip Theorem \cite{eberlein-oneill}]
 If $\g_0,\g_1:\R\to H$ are two unequal geodesics, whose images $\g_0(\R)\neq \g_1(\R)$ have finite Hausdorff distance, then $\g_0,\g_1$ bound an embedded strip $S\cong \R\times [0,1]$ in $H$ of positive width, which is totally geodesic and the curvature of $(S,g)$ vanishes.
\end{thm}

The fundamental group of a negatively curved, closed manifold $(M,g)$ cannot be abelian.

\begin{thm}[Preissmann]
 If the quotient $M=H/\Gam$ is compact and if the sectional curvatures of $g$ are all strictly negative, then any subgroup of $\Gam$ is infinitely cyclic, i.e.\ isomorphic to $\Z$.
\end{thm}

Finally, negatively curved, closed manifolds $(M,g)$ have an ergodic geodesic flow.

\begin{thm}[Hopf, Anosov]
 If the quotient $M=H/\Gam$ is compact and if the sectional curvatures of $g$ are all strictly negative, then the geodesic flow $\phi_g^t:S_gM\to S_gM$ is ergodic with respect to the canonical Liouville measure. In particular, almost every $g$-geodesic is dense in $S_gM$. Moreover, the set $\Per(\phi_g^t)$ of periodic orbits is dense in $S_gM$.
\end{thm}

One way to generalize Hadamard manifolds is to assume that the Finsler metric $F$ on $H$ does not have conjugate points. We recall the notion of conjugate points along an arc-length $F$-geodesic $c:\R\to H$ (equivalently $c:\R\to M$). For $v\in S_FH$ we write $V_v = \ker D\pi(v) \subset T_vS_FH$ for the vertical line bundle. Then two points $x_0,x_1\in c(\R)$ are {\em conjugate along $c$}, if there exist times $t_0<t_1$ with $x_i=c(t_i)$, such that for the geodesic flow
\[ D\phi_F^{t_1-t_0}(\dot c(t_0)) V_{\dot c(t_0)} = V_{\dot c(t_1)} , \]
cf. Figure \ref{fig-conjugate}. It is well-known that, if the Riemannian metric $g$ has non-positive curvature, then it has no conjugate points. If $(H,F)$ has no conjugate points, the theorem of Hadamard-Cartan still holds. Moreover, all geodesics are minimal,
\[ S_FH = \M . \]
Recall in this connection Question \eqref{central quest B} in Subsection \ref{section min geod}.

\begin{figure}%[!htb]
\centering
\includegraphics[scale=0.5]{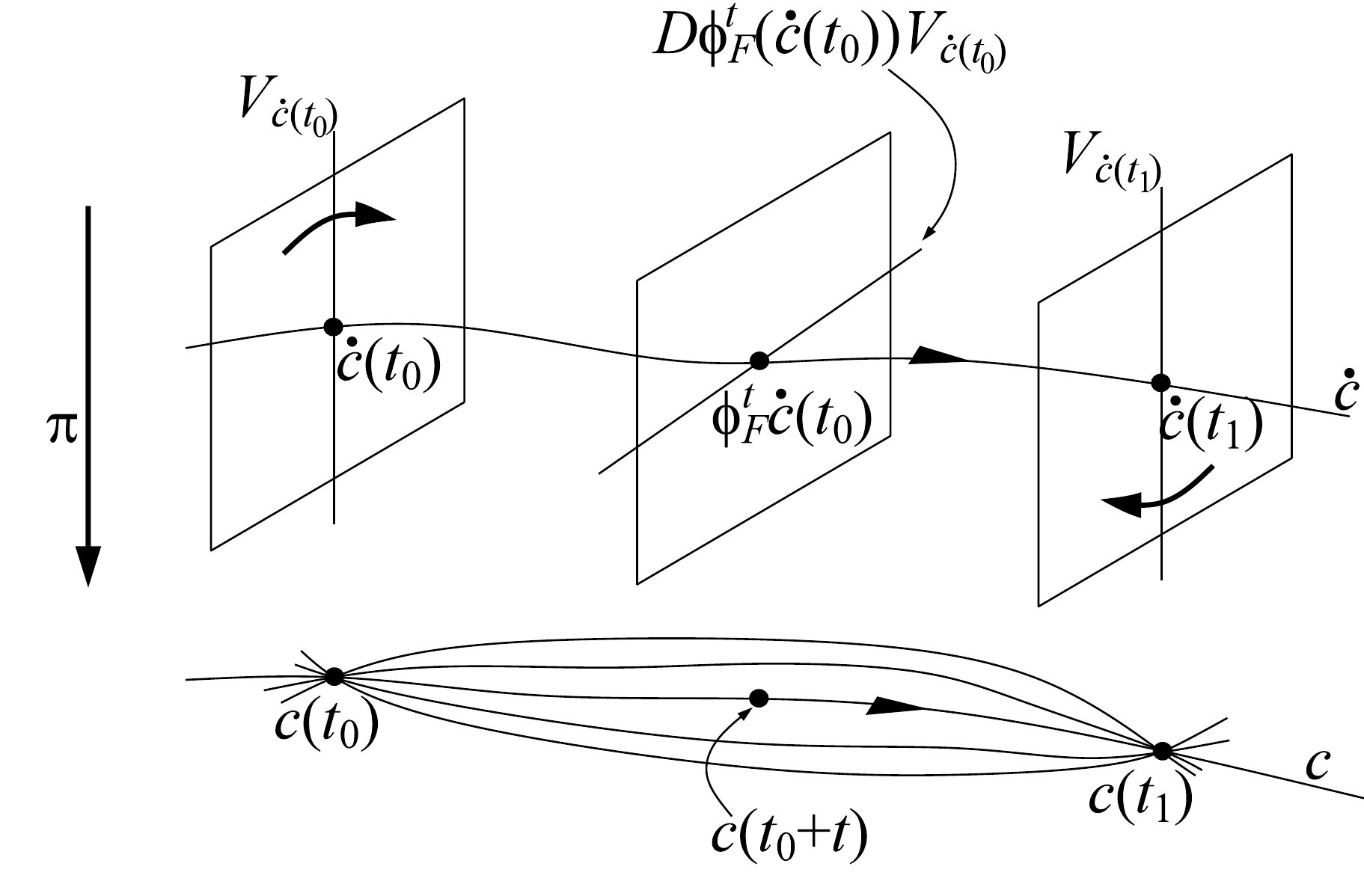}
\caption{An orbit $\dot c$ of the geodesic flow $\phi_F^t$ in the unit tangent bundle, together with the vertical line bundles $V_{\dot c}$ under the action of $D\phi_F^t$. The planes are the transverse sections of $\phi_F^t$ given by the kernel of the Liouville 1-form (i.e.\ the contact structure). Here we see the situation $D\phi_F^{t_1-t_0}(\dot c(t_0))V_{\dot c(t_0)}=V_{\dot c(t_1)}$, i.e.\ two conjugate points along $c$. Projecting to $H$ via $\pi$ yields a geodesic variation of $c$, which intuitively focuses at the conjugate endpoints $c(t_0)$ and $c(t_1)$. The arrows at $V_{\dot c(t_i)}$ indicate how the geodesic flow twists $V_{\dot c(t_i)}$. \label{fig-conjugate}}
\end{figure}

\section{Asymptotic directions for rays}\label{section morse lemma}

The main result in this section, Theorem \ref{morse lemma}, is referred to as the {\em Morse Lemma} in this paper. Note that in the literature there is also a Morse Lemma for smooth functions. The theorem to which we refer as the Morse Lemma has nothing to do with this.

\begin{thm}[Morse Lemma]\label{morse lemma}
 Let $(M,g)$ be either the flat 2-torus $(\T,euc)$ or have sectional curvatures bounded from above by some constant $-a^2<0$. Let $F$ be any Finsler metric on $M$ (recall that $F$ is uniformly equivalent to $g$). Then there exists a constant $C \geq 0$ depending only on $F$ and $g$ with the following property: If $x,y$ are any two points in the universal cover $H$ and if $c:[0,d_F(x,y)]\to H$ is a minimizing $F$-geodesic segment from $x$ to $y$, then for the (unique) $g$-geodesic segment $\g_{x,y}\subset H$ from $x$ to $y$ we have
 \[ \max \{ d_g(\g_{x,y},c(t)) : 0\leq t \leq d_F(x,y) \} \leq C. \]
\end{thm}

Hence, $F$-minimizing geodesics shadow $g$-geodesics in a uniform fashion. In this paper we will often work with the assumption that {\em the Morse Lemma holds}. In this case we mean that we are given $H,g,F$, so that the assertion in Theorem \ref{morse lemma} is satisfied.

Theorem \ref{morse lemma} appeared first as Lemma 8 in Morse's work \cite{morse} in the case of a surface of genus at least 2. Hedlund observed in \cite{hedlund} that it holds also for the 2-torus. Zaustinsky \cite{zaustinsky} observed that the metric $F$ does not have to be Riemannian or even reversible. The proof in higher dimensions in the case of manifolds with a negatively curved background metric $g$ was given by Klingenberg in \cite{klingenberg}. By an example of Hedlund \cite{hedlund}, Theorem \ref{morse lemma} is false even for $M=\Tn$ with $n\geq 3$.

We shall sketch the proof of Theorem \ref{morse lemma} in the two situations, which will be very different.

\begin{proof}[Sketch of the proof of Theorem \ref{morse lemma}, hyperbolic case]
 This is the case of an upper bound for the sectional curvatures. The proof of Klingenberg in \cite{klingenberg} is basically the same as in \cite{morse}. First observe that by using Rauch's comparison theorems, one can in essence assume that $(H,g)$ has constant sectional curvatures $-a^2$. Intuitively, if the curvature is below that, all estimates become more extreme. Let us assume $a=1$, the general case being analogous. We follow \cite{morse}.
 
 {\it Step 1}. There exists a constant $r=r(F,g)>0$ with the following property. Suppose $\g \subset H$ is a $g$-geodesic and $c:[t_0,t_1]\to H$ is a $C^1$-curve with
 \[ d_g(\g,c(t_0)) = d_g(\g,c(t_1)) = r \leq d_g(\g, c(t)) \qquad \forall t\in [t_0,t_1] , \]
 and moreover, projecting $c(t_0),c(t_1)$ to points $x_0,x_1\in\g$ via the unique $g$-geodesics orthogonal to $\g$, we assume
 \[ d_g(x_0,x_1) > 1 . \]
 Then $l_F(c;[0,1]) > d_F(c(0),c(1))$, i.e.\ $c$ is not minimal.
 
 \begin{proof}
  Let $c_r:[t_0,t_1]\to D$ be a curve from $c(t_0)$ to $c(t_1)$ in constant $g$-distance $r$ from $\g$. Hyperbolic geometry tells us that
  \[ l_g(c) \geq l_g(c_r) = d_g(x_0,x_1) \cdot \cosh(r) . \]
  The uniform equivalence of $F$ and $g$ in the sense $\frac{1}{c_F}|.|_g\leq F \leq c_F |.|_g$, $d_g(x_0,x_1)> 1$ and the triangle inequality then show
  \begin{align}
   \nonumber d_F(c(t_0),c(t_1)) & \leq c_F\bigg( d_g(c(t_0),x_0) + d_g(x_0,x_1) + d_g(x_1,c(t_1)) \bigg) \\
   \label{eqn morse lemma}& < c_F ( 2r +1) \cdot d_g(x_0,x_1) \\
   \nonumber & \leq \frac{c_F^2( 2r +1)}{\cosh(r)} \cdot l_F(c) .
  \end{align}
  Choosing $r$ so that $\frac{c_F^2( 2r +1)}{\cosh(r)} \leq 1$ proves the claim.
 \end{proof}
 
 {\it Step 2}. The theorem holds with
 \[ C = r + c_F^2 (2r+1) . \]
 
 \begin{proof}
  Let $c:[0,d_F(x,y)]\to D$ be a minimal segment from $x$ to $y$ and let $\g$ be the $g$-geodesic joining $x,y$. Suppose that $d_g(\g,c(t_*)) > r$ for some $t_*$ and let $t_0 \leq t_*$ be maximal with $d_g(\g,c(t_0))=r$ and $t_1\geq t_*$ be minimal with $d_g(\g,c(t_1))=r$. From Step 1 and the minimality of $c$ we find for the orthogonal projections $x_0,x_1$ of $c(t_0), c(t_1)$ to $\g$, that $d_g(x_0,x_1)\leq 1$. Again by minimality of $c$ and \eqref{eqn morse lemma},
  \begin{align*} 
   l_F(c;[t_0,t_1]) = d_F(c(t_0),c(t_1)) & \leq c_F ( 2r +1) .
  \end{align*}
  Hence, we find
  \begin{align*} 
   d_g(\g,c(t_*)) \leq d_g(x_0,c(t_*)) & \leq d_g(x_0,c(t_0)) + l_g(c;[t_0,t_1]) \\
   & \leq r + c_F^2 (2r+1).
  \end{align*}
  The claim follows.
 \end{proof}
\end{proof}

The proof in the torus case involves several steps and we will be more sketchy than in the hyperbolic case. We follow the ideas of Hedlund \cite{hedlund}, see also \cite{paper1} for the general Finsler case.

\begin{proof}[Sketch of the proof of Theorem \ref{morse lemma}, case of the 2-torus]
 \begin{enumerate}[1.]
  \item This is a to\-pological statement. Let $c:\R\to H \cong \R^2$ be a continuous, $\tau$-periodic curve for some $\tau\in\Gam\cong\Z^2$. Assume that $\tau$ is an iterate of some $\tau_0$, i.e.\ $\tau=\tau_0^k$ for some $k\geq 2$. Then $c$ can be decomposed into $k$ $\tau_0$-periodic curves.
  
  \item If $c:\R\to M$ is a shortest closed geodesic in the homotopy class $\tau=\tau_0^k$, then it is the $k$-th iterate of a shortest closed geodesic in the homotopy class $\tau_0$. For this one uses the decomposition from the first step: all curves in the decomposition have to be ``short'' and hence their translates cannot intersect transversely.
  
  \item It follows from Step 2 that shortest closed geodesics in homotopy classes lift to minimal geodesics in the universal cover $H$: Shortest closed geodesics in the class $\tau$ become shortest in $\tau^k$ after iteration, for all $k\geq 1$, such that the geodesics are minimizing on arbitrarily long subsegments.
  
  \item Prove that there exists a constant $r=r(F,g) \geq 0$, such that, given any $\tau\in\Gam$, all $\tau$-periodic minimal geodesics stay at distance at most $r$ from a $\tau$-periodic Euclidean straight line.
  
  \item Let $x,y\in H$ be arbitrary. Then there exist $\tau,\tau'\in\Gam$ and a $\tau$-periodic minimal geodesic $c:\R\to H$, such that $x,y$ lie in the strip bounded by $c(\R)$ and $\tau'\circ c(\R)$. Then any minimizing segment connecting $x$ to $y$ cannot exit the strip between the two $\tau$-periodic minimal geodesics, as this would imply successive intersections of minimizers (Lemma \ref{crossing minimals}). By Step 4, $c(\R),\tau'\circ c(\R)$ lie in $r$-distance from a Euclidean straight line. By controlling the size of $\tau'$, one infers the theorem.
 \end{enumerate}
\end{proof}

\begin{remark}\label{remark shorst closed geod are minimal}
 Observe that Steps 1 to 3 in the last proof also work in the case of a higher genus surface $M$, where $(H,g)$ is the \Poincare disc; observe that here the group generated by $\tau_0$ is isomorphic to $\Z$. This proves that if $M$ is a closed, orientable surface, the shortest closed geodesics in free homotopy classes are minimal geodesics.
\end{remark}

\subsection{The Gromov boundary of Hadamard manifolds}\label{section def gromov bdry}

We recall the definition of the Gromov boundary of a Hadamard manifold $(H,g)$, first introduced in \cite{eberlein-oneill}. Let $(H,g)$ be a Hadamard manifold, let $v,w\in S_gH$ and let $\g_v,\g_w:\R_-\to H$ be the corresponding $g$-rays. We render $v,w$ equivalent if the Hausdorff distance of the sets $\g_v(\R_-),\g_w(\R_-)$ in $(H,d_g)$ is finite. The {\em Gromov boundary} is the set of equivalence classes and will be denoted by
\[ \Gro_-(H,g) = \Gro(H,g) . \]
The equivalence class of a ray $\g:\R_-\to H$, respectively a vector $\dot\g(0)$ will be denoted by $\g(-\infty) \subset S_gH$. It can be shown that for each $x\in H$ and each $\xi\in \Gro(H,g)$ there is a unique $g$-ray $\g$ from $\xi$ to $x$, i.e.\ with $\g(-\infty)=\xi, \g(0)=x$. Hence, $\g(-\infty)\subset S_gH$ is a graph over $H$. We can make the same definitions with forward rays $\g:\R_+\to H$, defining a Gromov boundary $\Gro_+(H,g)$.

We endow $\Gro(H,g)$ with a topology. Given a sequence of points $\{\xi_n\}$ in $\Gro(H,g)$, which is a sequence of subsets of $S_gH$, we write $\xi_n\to \xi$ if for some (hence for all) $x\in H$ the arc-length $g$-rays $\g_n,\g:\R_-\to H$ with $\g_n(-\infty)=\xi_n , \g(-\infty)=\xi$ and $\g_n(0)=\g(0)=x$ satisfy $\dot \g_n(0) \to \dot \g(0)$ in the usual topology of $TH$. In this way, $\Gro(H,g)$ becomes homeomorphic to the unit sphere,
\[ \Gro(H,g)\cong \{ v\in T_xH : |v|_g=1\} \cong \mathbb{S}^{\dim H-1} . \]

The group $\Gam$ acts on $\Gro(H,g)$ by
\[ \tau (\g(-\infty)) := (\tau\circ \g)(-\infty) . \]

If $M$ is a closed, orientable surface of genus $\mathfrak{g}\geq 1$, the Gromov boundary is homeomorphic to $\se$. In Subsection \ref{section model geom} we defined for $\xi\in \se$ the graphs $\Sigma_\xi$ (if $\mathfrak{g}=1$) and $\Sigma_\xi^-$ (if $\mathfrak{g}\geq 2$). It is easy to see that these are precisely the graphs in $S_gH$ occuring as points $\g(-\infty)\in \Gro(H,g)$. We identify
\[ \Gro(H,g) \cong \se \qquad \text{via} \qquad \Sigma_\xi , \Sigma_\xi^- \cong \xi . \]
If $\mathfrak{g}=1$, then the Group $\Gam$ acts on $\Gro(H,g)$ only by the identity. If $\mathfrak{g}\geq 2$ then the Group $\Gam$ acts on $\Gro(H,g)\cong \se$ by the M\"obius transformations discussed in Subsection \ref{section model geom}.

\subsection{Asymptotic directions of rays}\label{section as directions}

Assume that $(H,g)$ is a Hada\-mard manifold and that the Morse Lem\-ma, i.e.\ Theorem \ref{morse lemma} holds. Due to the uniformity if the constant $C=C(F,g)$, we can associate to each $F$-ray $c:\R_-\to H$ a $g$-ray $\g:\R_-\to H$ having distance at most $C$ to $c$ and vice versa. We define the {\em asymptotic direction} of $c$ by
\[ c(-\infty) = \g(-\infty) \in\Gro(H,g) . \]
If $c$ is a forward ray, we define analogously $c(\infty) = \g(\infty)$, if $c(\R_+)$ and $\g(\R_+)$ have finite Hausdorff distance. By taking limits of minimizing $F$-geodesic segments from $\g(-n)$ to $x=\g(0)$ one sees that, for each $x\in H$ and each $\xi=\g(-\infty)\in \Gro(H,g)$, there exists an $F$-ray $c$ with $c(-\infty)=\xi$ and $c(0)=x$. Moreover, by letting $x$ tend to $+\infty$ along $\g$, one finds a minimal geodesic $c:\R\to H$ with $c(\pm\infty)=\g(\pm\infty)$.

For $\xi\in \Gro_\pm(H,g)$ and a $g$-geodesic $\g$, we set
\begin{align*}
 \RR_\pm(\xi ) & := \{ v\in \RR_\pm : c_v(\pm \infty)=\xi \}, \\
 \M(\g) & := \{ v\in \M : c_v(-\infty)=\g(-\infty) \text{ and } c_v(\infty)=\g(\infty) \}.
\end{align*}
All of these sets in $S_FH$ are non-empty, closed and $\phi_F^t$-invariant for $t\in \R_\pm, t\in\R$, respectively. Moreover,
\[ \pi(\RR_-(\xi)) = H \qquad \forall \xi\in \Gro(H,g) . \]
The sets $\RR_\pm,\M$ of rays and minimal geodesics now decompose:
\begin{align*}
 \RR_\pm & = \bigcup \big\{ \RR_\pm(\xi) : \xi\in \Gro_\pm(H,g) \}, \\
 \M &= \bigcup\big\{ \M(\g) : \g \text{ a $g$-geodesic} \big\} .
\end{align*}
Observe also that we have continuity:
\[ v_n\to v \text{ in } \RR_\pm \quad \implies  \quad  c_{v_n}(\pm\infty)\to c_v(\pm\infty) \text{ in } \Gro_\pm(H,g). \]
Note that $\RR_-$ with the decomposition into sets $\RR_-(\xi)$ generalizes the construction of the Gromov boundary. One way of expressing the above facts is that the Gromov boundary of $F$ and $g$ (defined in terms of rays) is independent of the metric, so long as the Morse Lemma holds.

\subsection{The width of asymptotic directions}\label{section width}

Let $(M=H/\Gam,g)$ have curvature bounded from above by a negative constant. The Morse Lemma holds. The following notion has proven to be very useful. It was introduced in \cite{min_rays}.

\begin{defn}\label{def width}
For $\xi\in \Gro(H,g)$ and a $g$-geodesic $\g:\R\to H$ we define the {\em widths of $\xi,\g$}, respectively:
\begin{align*}
w_-(\xi) &:= \sup \big\{ \liminf_{t\to -\infty} d_g(c_v(\R_-),c_w(t)) : v,w\in\RR_-(\xi) \big\} , \\
w_0(\g) &:= \sup \big\{ \inf_{t \in\R} d_g(c_v(\R),c_w(t)) : v,w\in\M(\g) \big\} .
\end{align*}
\end{defn}

The widths have the following properties (see \cite{min_rays}).

\begin{prop}\label{width semi-cont}
 \begin{enumerate}[(i)]
  \item The widths $w_0,w_-$ are bounded by the constant $C=C(F,g)$ given by the Morse Lemma. 
 
  \item Let $\g,\g':\R_-\to H$ be two $g$-rays and let $\{\tau_n\}\subset\Gam$ with
  \[ D\tau_n(\dot \g(t_n)) \to \dot \g'(0) \text{ for some sequence } t_n\to-\infty . \]
  Then $w_-(\g(-\infty)) \leq w_0(\g')$.  
  
  \item If $w_-(\xi)=0$, then no two rays in $\RR_-(\xi)$ intersect.
 \end{enumerate}
\end{prop}

The last property is a direct consequence of Lemma \ref{crossing minimals}.

Let us now assume that $M$ is compact. Using the ergodicity of the $g$-geodesic flow, for almost every $\xi\in \Gro(H,g)$ (using the Lebesgue measure in $\mathbb{S}^{\dim M-1}\cong \Gro(H,g)$), the $g$-rays $\g$ with $\g(-\infty)=\xi$ are backwards dense in $S_gM$ under the $g$-geodesic flow. Hence, Proposition \ref{width semi-cont} proves the following.

\begin{cor}
 Let $(M,g)$ be compact and have negative curvature. If $F$ is a Finsler metric on $M$, so that there exists a $g$-geodesic $\g:\R\to H$ with $w_0(\g)=0$, then $w_-(\xi)=0$ for almost all $\xi\in\Gro(H,g)$.
\end{cor}

We will see that $\dim M=2$ implies the existence of $\g$ with $w_0(\g)=0$. In Theorem \ref{thm rays-paper} we will even conclude, that for all irrational $\xi\in \se$ we have $w_-(\xi)=0$. In particular, if $\dim M=2$, then most sets $\RR_-(\xi)$ have a simple structure, as there are no intersecting rays. This motivates the following question.

\begin{prob}\label{problem thin neck}
 Let $(M,g)$ be compact and have negative curvature, $F$ a Finsler metric on $M$. Do there always exist $g$-geodesics $\g:\R\to H$ with $w_0(\g)=0$?
\end{prob}

\section{Dominated functions and weak KAM solutions}\label{section weak KAM}

In Subsection \ref{section rot torus} we saw the example of the rotational torus admitting invariant graphs for the geodesic flow. Moreover, in Subsection \ref{section model geom} we saw invariant graphs of the form $\Sigma_\xi,\Sigma_\xi^-$. In Subsection \ref{section as directions}, we already generalized these sets in the form $\RR_-(\xi)$. In this section, we discuss another way of generalizing the notion of such invariant graphs. We follow some ideas from \cite{fathi}.

We translate to the Hamiltonian setting and work in the universal cover $H$. The dual Finsler metric associated to $F$ is given by
\[ F^*:T^*H\to \R, \qquad F^*(\eta) := \max \{ \eta(v) : \pi v = \pi \eta , ~ F(v)=1 \} . \]
Due to strict convexity of $F$, the maximum in the definition of $F^*(\eta)$ is attained in a unique point $v\in S_FH \cap T_{\pi\eta}H$, which we denote by $\theta_F(\eta)$. Hence, we obtain the {\em Fenchel inequality}
\[ \eta(v) \leq F(v)\cdot F^*(\eta) \]
for all $v\in TH,\eta\in T^*H$ with $\pi v=\pi \eta$ and
\[ \eta(v) = F(v)\cdot F^*(\eta) \iff v= F(v) \cdot \theta_F(\eta) . \]

The invariant graphs found in the above examples are in this setting graphs of 1-forms
\[ \eta:H\to S_F^*H = \{F^*=1\}\subset T^*H \]
Let us assume that $\eta$ is $C^1$ and closed, the latter being automatically the case if $\dim H=2$. This means that the graph of $\eta$ is a Lagrangian $C^1$-submanifold of $T^*H$. As $H$ is simply connected the closed 1-form $\eta$ becomes an exact 1-form. Hence, each invariant graph in the above examples lifted to the universal cover $H$ is of the form
\[ \{ (x,du(x)) : x\in H \} \subset S_F^*H \]
for some $C^2$-function $u:H\to\R$. Hence, the function $u$ is a solution of the {\em Hamilton-Jacobi equation} (sometimes called eikonal equation)
\[ F^*(du(x)) = 1 \qquad \forall x\in H. \]
We see that being a solution of the Hamilton-Jacobi equation depends only of the differential of $u$. In particular, $u$ is a solution if and only if $u+c$ is, for any $c\in\R$.

\begin{defn}
 We identify continuous functions $u,u' \in C^0(H)$, if $u-u'=\const$. The corresponding equivalence relation is denoted by $\sim$ on $C^0(H)$, denoting equivalence classes by $[u]=\{u+c:c\in\R\}$. We endow $C^0(H)$ with the $C_{loc}^0$-topology of locally uniform convergence, i.e.\ we write $u_n\to u$ if
 \[ \forall K\subset H \text{ compact}: \quad \max\{u_n(x)-u(x) : x\in K \} \to 0. \]
 The quotient $C^0(H)/_\sim$ inherits the quotient $C_{loc}^0$-topology.
 
 We apply the analogous definitions to functions $u\in C^0(M)$.
\end{defn}

In general, we will not be able to find smooth solutions to the Hamilton-Jacobi equation. A common approach (see \cite{fathi}) is to seek subsolutions in some sense, replacing equality by an inequality in the Hamilton-Jacobi equation. Let $c:[a,b]\to H$ be a $C^1$-curve and suppose that $F^*(du)\leq 1$ everywhere. Then by the Fenchel inequality
\begin{align*}
 u(c(b))-u(c(a)) & = \int_a^b du(c)[\dot c] ~ dt \leq \int_a^b F(\dot c)\cdot F^*(du(c)) ~ dt \\
 & \leq l_F(c;[a,b]) .
\end{align*}
As $c$ was arbitrary, we find that $F^*(du)\leq 1$ implies
\begin{align}\label{eq u dominated}
 u(y)-u(x) \leq d_F(x,y) \qquad \forall x,y\in H . 
\end{align}
It is not difficult to show that the converse is also true.

\begin{defn}\label{def dominated}
 A {\em dominated function} (with respect to $F$) is an element $[u]\in C^0(H)/_\sim$, whose members satisfy the inequality \eqref{eq u dominated}. The set of all dominated functions $[u]$ is denoted by $\Dom(H,F)$.
\end{defn}

Observe that the constant functions form an element of $\Dom(H,F)$. Using the triangle inequality it is easy to see that for a given point $x\in H$ the distance function $[d_F(x,.)] \in \Dom(H,F)$. As the Finsler metric $F$ is equivalent to the Riemannian metric $g$ on $H$, any dominated function is $g$-Lipschitz with the Lipschitz constant $c_F\geq 1$, in particular continuous and differentiable almost everywhere by Rademacher's theorem. The set $\Dom(H,F)$ is sequentially closed in $C^0(H)/_\sim$ with respect to the $C_{loc}^0$-topology. Using the theorem of Arzela-Ascoli and the uniform Lipschitz continuous, the set $\Dom(H,F)$ is sequentially compact.

The inequality $u(c(b))-u(c(a)) \leq l_F(c;[a,b])$ shows the following.

\begin{prop}\label{prop calibrate implies minimal}
 If $[u]\in \Dom(H,F)$ and if $c:[a,b]\to H$ is an arc-length, continuous, piecewise $C^1$-curve with $u\circ c(b)-u\circ c(a)=b-a$, then $c$ is a minimizing geodesic segment.
\end{prop}

\begin{defn}\label{def calibration}
 An arc-length $C^1$-curve $c:[a,b]\to H$ with $u\circ c(b)-u\circ c(a)=b-a$ is said to {\em calibrate $[u]\in \Dom(H,F)$ on $[a,b]$}.
\end{defn}

The group of isometries $\Gam$ acts on $\Dom(H,F)$ via
\[ \tau [u] = [u \circ \tau^{-1}] . \]
Hence, if $c$ is $[u]$-calibrated, then $\tau\circ c$ is $\tau [u]$-calibrated.

Note that a calibrating curve $c:[a,b]\to H$ is also calibrating on all $[s,t]\subset[a,b]$. In the next well-known proposition, we relate the calibration property to the differential of $u$. Recall the map $\theta_F:T^*H\to S_FH$ defined at the beginning of this section.

\begin{prop}\label{lemma weak KAM}
 Let $[u]\in\Dom(H,F)$, $a<b$ and $c:[a,b]\to H$ be an arc-length geodesic segment.
  \begin{enumerate}[(i)]
   \item\label{u diffbar on cal} If $c$ is $[u]$-calibrated, then $u$ is differentiable with $F^* \circ du=1$ in $c(a,b) \subset H$ and $c$ satisfies $\dot c =\theta_F \circ du \circ c$ in $(a,b)$. Moreover, if $u$ is differentiable in $c(t)$ for $t\in \{a,b\}$, then $\dot c(t) =\theta_F \circ du \circ c(t)$.
   
   \item\label{item graphs calibrate} If $u$ is differentiable with $F^* \circ du=1$ in $c(a,b)\subset H$ and if $c$ satisfies $\dot c=\theta_F\circ du \circ c$ in $(a,b)$, then $c$ is $[u]$-calibrated.
  \end{enumerate}
\end{prop}

\begin{remark}
 Observe that for the constant functions the are no calibrating curves at all. On the other hand, if $u$ is not differentiable in some point $x$, then there can be several $[u]$-calibrated curves emanating from $x$. If $u$ is differentiable in $x$, then there can be at most one $[u]$-calibrated curve $c$ ending in or emanating from $x$. Moreover, one can show Lipschitz regularity results on $du$, see Sections 4.11 and 4.13 in \cite{fathi}. This shows that the initial condition $v(x)= \theta_F \circ du(x)$ of a calibrated curve through $x$ has Lipschitz regularity in the domain of $du$.
\end{remark}

\begin{proof}
 \eqref{u diffbar on cal}. For any $x\in H$ we find by \eqref{eq u dominated}
 \begin{align*}
 \psi_-(x) := - d_F(x,c(b)) ~ + & ~ u(c(b)) \leq u(x) \\
 & \leq d_F(c(a),x) + u( c(a)) =: \psi_+(x) .
 \end{align*}
 As $c$ calibrates $u$, we find $\psi_- = u =\psi_+$ in $c(a,b)$, while $\psi_\pm$ are both smooth in $c(a,b)$ by minimality of the segment $c$. The differentiability of $u$ in $c(a,b)$ follows.
 
 Suppose now that $u$ is differentiable in $c(a)$. By \eqref{eq u dominated} we have
 \begin{align*}
  \textstyle du(c(a))\dot c(a) & = \lim_{t\searrow 0} \frac{1}{t} \big[u(c(a+t)- u(c(a))\big] \\
  & \leq \lim_{t\searrow 0} \frac{1}{t}  \int_a^{a+t} F(\dot c(s))ds = F(\dot c(a)) ,
 \end{align*}
 which becomes an equality, if $c$ calibrates $[u]$. In the latter case we find
 \begin{align*}
  du(c(a))\dot c(a) & = \max\{ du(c(a))[v] : \pi v = c(a) , ~ F(v)=1 \} \\
  & = F^*(du(c(a))).
 \end{align*}
 Hence, $\dot c(a) = \theta_F\circ du \circ c(a)$ by definition of $\theta_F$.

 \eqref{item graphs calibrate} We find by the Fenchel equality along $c(a,b)$, that
 \[ du(c)[\dot c] = du(c)[\theta_F\circ du(c)] = F(\theta_F\circ du(c)) \cdot F^*(du(c)) = 1 . \]
 The claim follows.
\end{proof}

The elements $[u]\in\Dom(H,F)$ representing invariant graphs in the above example have another property, which we wish to generalize. Let us assume that $F^*(du)\equiv 1$ for some $C^2$-function $u:H\to\R$. This defines an arc-length $C^1$-vector field $V(x) := \theta_F\circ du(x)$ on $H$. Using the solutions of $\dot c=V\circ c$ and Proposition \ref{lemma weak KAM} \eqref{item graphs calibrate} proves the following.

\begin{prop}\label{prop graphs calibrate}
 If $u:H\to \R$ is $C^2$ with $F^*(du)\equiv 1$, then for each point $x\in H$ there exists a $C^1$-curve $c:\R_-\to H$ with $c(0)=x$ calibrating $u$ on $\R_-$.
\end{prop}

Of course, the same holds true for $\R$ instead of $\R_-$.

\begin{defn}\label{def weak KAM}
 An element $[u]\in C^0(H)/_\sim$ is a {\em weak KAM solution} with respect to $F$, if
 \begin{enumerate}[(i)]
  \item $[u]\in \Dom(H,F)$
  
  \item\label{cond calibration} for each $x\in H$ there exists a ray $c:\R_-\to H$ with $c(0)=x$ calibrating $[u]$ on $\R_-$.
 \end{enumerate}
 The set of all weak KAM solutions in $(H,F)$ is denoted by $\W(H,F)$.
\end{defn}

\subsection{Directed weak KAM solutions}\label{section directed KAM}

As a byproduct we saw that invariant graphs of closed 1-forms in $H$ consist of rays and minimal geodesics. In Section \ref{section morse lemma} we saw that rays have in some cases a well-defined asymptotic direction. This leads to a special class of weak KAM solutions. In general, a weak KAM solution $u\in \W(H,F)$ admits calibrated rays with many different asymptotic directions.

\begin{defn}\label{def dir weak KAM}
 Assume that $(H,g)$ is a Hadamard manifold and that the Morse Lemma holds. A weak KAM solution $[u]\in\W(H,F)$ is said to be {\em directed}, if all $[u]$-calibrated rays $c:\R_-\to H$ have a common asymptotic direction $\xi = c(-\infty)$. The set of all directed weak KAM solutions will be denoted by $\W_{dir}(H,F)$. The map assigning to $[u]\in \W_{dir}(H,F)$ the asymptotic direction $\xi = c(-\infty)$ of its calibrated rays will be denoted by
 \[ \delta_F: \W_{dir}(H,F) \to \Gro(H,g) . \]
\end{defn}

One might expect that general weak KAM solutions (and hence candidates of invariant graphs for the geodesic flow in $S_FH$) are ``pieced together'' by directed ones. In this sense, the directed weak KAM solutions could be treated as simple building blocks. One main objective in this paper is to study the structure of directed weak KAM solutions. We assume for the rest of the section, that $(H,g)$ is a Hadamard manifold and that the Morse Lemma (Theorem \ref{morse lemma}) holds. In the following, we employ some ideas from Section 3 in \cite{min_rays}, referring to \cite{min_rays} for the proofs (even though we work with horofunctions there, one easily sees that the arguments work for directed weak KAM solutions).

\begin{prop}
 The set $\W_{dir}(H,F)$ is closed in $C^0(H)/_\sim$ and the map $\delta_F$ is continuous and surjective.
\end{prop}

In general, the map $\delta_F$ need not be injective. In order to study the injectivity properties, one can show the following.

\begin{prop}\label{prop unique horo if no intersection}
 The set $\RR_-(\xi)$ consists of all rays which are calibrated with respect to some $u\in \delta_F^{-1}(\xi)$. We have $\card \delta_F^{-1}(\xi) = 1$ if and only if there are no intersecting rays in $\RR_-(\xi)$.
\end{prop}

Recall that for $(M,g)$ of negative curvature we defined the width $w_-(\xi)$ in Subsection \ref{section width}. In particular,
\[ w_-(\xi)=0 \quad \implies \quad \card \delta_F^{-1}(\xi) = 1 . \]

Rays admitting no intersections in $\RR_-(\xi)$ have special importance due to Proposition \ref{prop unique horo if no intersection}. This motivates the following definition.

\begin{defn}\label{def unstable}
 An $F$-ray $c:\R_-\to H$ is {\em (backward) unstable}, if for all rays $c':\R_- \to H$ with $c'(0)\in c(-\infty,0)$ and $c'(-\infty)=c(-\infty)$ we have $c'(\R_-)\subset c(-\infty,0)$, i.e.\ if $c$ does not admit intersecting rays in $\RR_-(c(-\infty))$. The set of unstable rays in $\RR_-(\xi)$ will be denoted by
 \[ \A_-(\xi) \subset \RR_-(\xi) . \]
\end{defn}

Note that results on unstable rays are presented in \cite{klingenberg}, many of which, however, have an incomplete proof there.

Existence results on unstable rays will be given in Theorems \ref{morse periodic} \eqref{morse periodic item 3}, \ref{thm measures unstable}, \ref{thm bangert irrat} \eqref{thm bangert irrat iv} and \ref{thm irrat no intersections}. In general, however, it is not clear whether $\A_-(\xi)\neq \emptyset$. 

It is not difficult to show the following characterization of instability. It resembles a characterization of the Aubry set in Mather theory (see Proposition \ref{weak KAM charact aubry set}), hence the letter $\A$.

\begin{prop}\label{char unstable via weak KAM}
 A ray $c:\R_-\to H$ is unstable if and only if $c$ is calibrated with respect to all directed weak KAM solutions $[u]\in \delta_F^{-1}(c(-\infty))$.
\end{prop}

The next proposition says that rays in $\RR_-(\xi)$ become ``almost calibrated'' with respect to any $[u]\in \delta_F^{-1}(\xi)$ near $-\infty$. So intuitively, the rays in $\RR_-(\xi)$ are unstable near $-\infty$.

\begin{prop}\label{int dH = 1}
 Let $[u]\in \delta_F^{-1}(\xi)$ and $v\in \RR_-(\xi)$. Then the function
 \[ t\in \R_- \mapsto  t- u\circ c_v(t) \]
 is bounded and non-decreasing.
\end{prop}

The ``non-decreasing'' part follows from \eqref{eq u dominated}. Namely, for $a\leq b$
\[ u \circ c_v(b) - u \circ c_v(a) \leq b-a . \]
The ``bounded'' part is a consequence of the Morse Lemma and the fact that any $u\in \delta_F^{-1}(\xi)$ has some calibrated ray in $\RR_-(\xi)$ in finite distance of $c_v$. The proposition shows that for $a\leq b \ll -1$ we have
\[ a- u\circ c_v(a) \approx b- u\circ c_v(b) , \]
i.e.\ $c_v$ is ``almost $u$-calibrated''.

Finally, we can find ``boundary elements'' of $\delta_F^{-1}(\xi)$, if $\dim H=2$. Recall that $\Dom(H,F)$ is sequentially compact.

\begin{prop}[boundary of $\delta_F^{-1}(\xi)$] \label{bounding horofunctions}
Let $\dim H=2$. Then for $\xi\in \Gro(H,g) \cong \se$, there exist two unique elements $[u_0],[u_1]\in \delta_F^{-1}(\xi)$ with the following property: for any sequence $\xi_n\to \xi$ in $\Gro(H,g)$ with $\xi_n\neq \xi$ and any sequence $[u_n]\in \delta_F^{-1}(\xi_n)$, any limit point lies in $\{[u_0],[u_1]\}$. More precisely, assuming the counterclockwise orientation of $\Gro(H,g) \cong \se$, we have
\[ \lim_{n\to\infty} [u_n] = \begin{cases} [u_0] & : \text{ if }~ \xi_n > \xi ~ \forall n \\ [u_1] & : \text{ if }~ \xi_n < \xi ~ \forall n \end{cases}. \]
\end{prop}

Proposition \ref{bounding horofunctions} can be used to find ``recurrent'' directed weak KAM solutions with respect to the group action by $\Gam$. The recurrent elements in $\delta_F^{-1}(\xi)$ always have the recurrent rays as calibrated curves. For this, one uses Proposition \ref{int dH = 1}.

\begin{cor}\label{cor exist recurrent horofctns}
 Let $\dim H=2$. If $\xi\in \Gro(H,F)$ and if $\tau_n\in \Gam$ is a sequence with $\tau_n\xi \to \xi$, then there exists an element $[u]\in\delta_F^{-1}(\xi)$ with $\tau_n [u] \to [u]$. Moreover, suppose that $v\in \RR_-(\xi)$ is recurrent under the sequence $\tau_n$, i.e.\ there exists a sequence $t_n\to -\infty$ with $D\tau_n(c_v(t_n))[\dot c_v(t_n)]\to v$ (this means recurrence in the quotient $M=H/\Gam$). Then $c$ is calibrated with respect to any $[u]\in\delta_F^{-1}(\xi)$ recurrent under $\{\tau_n\}$ in the above sense.
\end{cor}

Hence, recurrent rays have certain instability properties, at least in dimension two. This resembles a fact in Mather theory, where the non-wandering set of the \Mane set is contained in the Aubry set (Proposition 6.33 in \cite{sorrentino}).

\subsection{Horofunctions}\label{section horofctns}

So far, we have not discussed directly examples of dominated functions, weak KAM solutions, let alone directed weak KAM solutions. Here we discuss the classical examples, given by Busemann functions and more generally by horofunctions. In particular, we treat the compactification of $(H,F)$ due to Gromov \cite{gromov}, see also \cite{ballmann} or \cite{bridson}.

We remarked earlier, that for $x\in H$ we have
\[ i_F(x) := [d_F(x,.)] \in \Dom(H,F) . \]
The calibrated curves of $i_F(x)$ are precisely the minimizing geodesic segments and forward rays emanating from $x$. The level sets of $i_F(x)$ are the spheres $\{ y\in H: d_F(x,y)=r\}$. The definition of $i_F$ above yields an embedding
\[ i_F:H \to C^0(H)/_\sim . \]
The image $i_F(H)$ lies in $\Dom(H,F)$ and hence is sequentially pre-compact with respect to the $C_{loc}^0$-topology of locally uniform convergence. In particular, the map $i_F$ defines a compactification of $H$.

\begin{defn}
 The closure $\overline{i_F(H)}\subset \Dom(H,F)$ is called the {\em horofunction compactification} of $(H,F)$. The boundary $\Hor(H,F):= \overline{i_F(H)} - i_F(H)$ is called the {\em horofunction boundary} of $(H,F)$. The elements of $\Hor(H,F)$ are called {\em horofunctions}.
\end{defn}

One can show that $i_F(H)$ is open in $\overline{i_F(H)}$ and that $[u]\in \Dom(H,F)$ lies in the boundary $\Hor(H,F)$ if and only if it is the limit of a sequence $i_F(x_n)$ with $d_F(x_0,x_n)\to\infty$ for some and hence any $x_0\in H$. Hence, all the calibrated curves of $i_F(x_n)$ become in the limit (backward!) rays and minimal geodesics, emanating from an imaginary limit point of the sequence $\{x_n\}$ at infinity and spreading all over $H$. In particular, $\Hor(H,F) \subset \W(H,F)$. More precisely, one can show the following.

\begin{prop}
 If $(H,g)$ is a Hadamard manifold and if the Morse Lemma holds, then
 \[ \Hor(H,F)\subset\W_{dir}(H,F) . \]
\end{prop}

A standard example is again given by the \Poincare disc $(H,g)$. Given a point $\xi$ in the Gromov boundary $\Gro(H,g) \cong \se$, there exists a unique limit $[u_\xi]=\lim i_g(x_n)$ for any sequence $x_n\in H$ converging to $\xi$ in the Euclidean sense in $\R^2$. The level sets of $[u_\xi]$ are the horocycles based at $\xi$, which are the limits of spheres around the points $x_n$. The calibrated curves of $[u_\xi]$ form the ``unstable manifold'' of any geodesic emanating from $\xi$. The uniqueness of $[u_\xi]$ means that the horofunction compactification $\overline{i_g(H)}$ is in this case homeomorphic to the closed unit disc $H\cup \se$ with the Euclidean topology, while $\Hor(H,g) \cong \se \cong \Gro(H,g)$. In the same direction, one shows that in the case of the \Poincare disc
\[ \delta_g : \W_{dir}(H,g) \to \Gro(H,g) \]
is a homeomorphism.

Next, we consider Busemann functions, first introduced in \cite{busemann}. Given a ray $c:\R_-\to H$, one considers the limit
\[ b_c := \lim_{t\to \infty} d_F(c(-t),.)-t . \]
Here, $t=d_F(c(-t),c(0))$ by minimality of $c$. Using the triangle inequality one easily shows that the limit $b_c$ exists and yields an element of $\Dom(H,F)$. In fact,
\[ [b_c]  = \lim_{t\to \infty} i_F(c(-t)) . \]
The minimality of $c$ forces $c(-t)$ to diverge to infinity and $b_c$ to be an element of $\Hor(H,F)$, so that Busemann functions are a special case of horofunctions. Using Busemann functions one infers that any ray is calibrated with respect to some directed weak KAM solution.

\subsection{Injectivity of the asymptotic direction}\label{section hadamard horo}

Let $(H,g)$ be a Hada\-mard manifold. We already remarked the following fact in Subsection \ref{section def gromov bdry}. Namely, for fixed $x\in H$, there exists for each $\xi\in\Gro(H,g)$ precisely one $g$-ray $\g:\R_-\to H$ with $\g(0)=x$ and $\g(-\infty)=\xi$. This shows that there are no intersecting $g$-rays with common endpoints at $-\infty$. By Proposition \ref{prop unique horo if no intersection} the map
\[ \delta_g : \W_{dir}(H,g) \to \Gro(H,g) \]
is a homeomorphism. The same holds for $\delta_g|_{\Hor(H,g)}$.

Conversely, if the Morse Lemma holds and if $\delta_F:\W_{dir}(H,F)\to\Gro(H,g)$ is a homeomorphism, then the set of $F$-rays with a common endpoint at $-\infty$ has a very simple structure, as there are no intersecting rays. Let us pose the following problem.

\begin{prob}\label{prob when delta_F homeo}
 Let $F$ be a Finsler metric on $H$, so that the Morse Lemma holds. When is
 \[ \delta_F : \W_{dir}(H,F) \to \Gro(H,g) \]
 a homeomorphism? Equivalently, when are all $F$-rays unstable?
\end{prob}

Note that it is enough to show that $\delta_F$ is injective. In the case of closed surfaces it is proved in Theorem 12.1 of \cite{morse hedlund}, that $\delta_F$ is a homeomorphism if $F$ has no conjugate points. Here, also $\M=S_FH$, so that the sets $\RR_-(\xi)$ form a continuous foliation of $S_FH$, which could be phrased as $C^0$-integrability of $\phi_F^t:S_FH\to S_FH$. For the case of $M=\T$ and $F$ Riemannian, a metric $F$ on $\T$ without conjugate points is flat \cite{hopf}. We shall later propose a solution to Problem \ref{prob when delta_F homeo} in the case of general Finsler metrics on a closed surface $M$ (see Corollaries \ref{cor horobdry torus} and \ref{cor horobdry higher gen}).

\section{Mather theory}\label{section mather theory}

In this section we move to the quotient manifold $M=H/\Gam$ and we moreover assume that $M$ is compact. The approach in this section is due to Mather \cite{mather} and allows to study minimizers of infinite length in the (compact!) quotient $M$. Already in Section \ref{section weak KAM} we considered the case where we were given a $C^1$ closed 1-form $\eta:M\to S_F^*M$; $\eta$ defines a graph in $S_F^*M$ invariant under the (dual) geodesic flow. Note that in the quotient we cannot expect $\eta$ to be exact, as we did when working in the universal cover $H$. Recall the map $\theta_F : T^*M \to S_FM$ defined in the beginning of Section \ref{section weak KAM}. The following well-known proposition motivates the approach in this section; it rests on the same argument as Proposition \ref{lemma weak KAM} \eqref{item graphs calibrate}.

\begin{prop}\label{prop motivates mather}
 If $\eta:M\to T^*M$ is a closed 1-form with $F^*\circ \eta \equiv 1$, then for any solution $c:\R\to M$ of $\dot c = \theta_F \circ \eta(c)$ and $a\leq b$ we have
 \[ \int_a^b F(\dot c)-\eta(\dot c) ~ dt = \inf \left\{ \int_0^1 F(\dot \g)-\eta(\dot \g) ~ dt \right\} , \]
 where the infimum is taken over all continuous, piecewise $C^1$-curves $\g:[0,1]\to M$ with $\g(0)=c(a), \g(1)=c(b)$. In particular, the solutions of $\dot c=\theta_F \circ \eta(c)$ are $F$-geodesics.
\end{prop}

Note that the integral $\int F(\dot c)-\eta(\dot c) ~ dt$ is invariant under orientation preserving reparametrizations of $c$.

\begin{proof}
 We have by definition of $c,\theta_F$ and $F^*\circ\eta\equiv 1$
 \begin{align*}
  \int_a^b F(\dot c)-\eta(\dot c) ~dt = \int_a^b F(\dot c)-F^*(\eta\circ c) ~dt = \int_a^b 1-1 ~dt = 0 .
 \end{align*}
 Let now $\g$ be an arbitrary connection from $c(a)$ to $c(b)$. Then by the Fenchel inequality and $F^*\circ\eta\equiv 1$, 
 \begin{align*}
  \int_a^b F(\dot \g)-\eta(\dot \g) ~dt \geq \int_a^b F(\dot \g)-F(\dot \g)\cdot F^*(\eta\circ \g) ~dt = 0 .
 \end{align*}
 The first claim follows.
 
 To see that the solutions of $\dot c = \theta_F \circ \eta(c)$ are $F$-geodesics, observe that $\eta$ is locally exact by being closed. This means that the integral $\int F-\eta dt$ along $c$ depends locally only on $\int F dt$ and the endpoints of $c$, so that the solutions of $\dot c = \theta_F \circ \eta(c)$ locally minimize the $F$-length.
\end{proof}

Proposition \ref{prop motivates mather} motivates the study of the ``distance''
\[ d_{F-\eta}(x,y)=\inf\{ l_{F-\eta}(c;[0,1]) ~|~ c:[0,1]\to M, ~ c(0)=x, ~ c(1)=y \} \]
on $M$ induced by the ``length''
\[ l_{F-\eta}(c;[a,b]) := \int_a^b F(\dot c)-\eta(\dot c) ~ dt = l_F(c;[a,b])-\int_c\eta \]
for $C^1$-curves $c:[a,b]\to M$. Here, $\eta$ is a smooth, closed 1-form on $M$. Note that if $f:M\to\R$ is a smooth function, then $l_{F-\eta}$ and $l_{F-\eta+df}$ have the same minimizers. Hence, we shall fix a given cohomology class in $H^1(M,\R)$ and take any closed 1-form $\eta$ representing this class. The closed 1-form $\eta$ will be treated as given.

As an example, let $F$ be a Finsler metric on the 2-torus $\T=\R^2/\Z^2$. Closed 1-forms are of the form
\[ \eta = \la v, . \ra + ~ df , \]
where $v\in\R^2$ is a constant vector and $f:\T\to\R$ is a smooth function. Then we have
\[ l_{F-\eta}(c;[a,b]) = l_F(c;[a,b])-\int_a^b \la v, \dot c \ra dt - f\circ c(b) + f\circ c(a) . \]
Fixing the endpoints $c(a),c(b)$, the ``length'' $l_{F-\eta}$ singles out geodesics $c:\R\to\T$ whose lifts $\tilde c:\R\to\R^2$ to the universal cover seem shorter if they travel in the direction $v$, and longer if they travel against the direction $v$. That is, if $c$ minimizers the ``length'' $l_{F-\eta}$, it will travel in the direction of $v$. If $F=|.|$ is the Euclidean metric, then the minimizers will be straight lines of direction $v$. Note that in standard coordinates $v$ is given by the cohomology class $[\eta]\in H^1(\T,\R)\cong\R^2$.

We return to the general case. One can see that the definition of $d_{F-\eta}$ is problematic. Indeed, if $\eta$ is ``too large'', then the distance defined using $l_{F-\eta}$ will be $-\infty$. Hence, we need to define a set of appropriate $\eta$'s. Consider the set $\MM(F)$ of probability measures supported in the unit tangent bundle $S_FM$ invariant under the geodesic flow $\phi_F^t$, endowed with the topology of weak convergence. We define the {\em rotation vector}
\[ \rho:\MM(F)\to H_1(M,\R)=H^1(M,\R)^*, \qquad \rho(\mu)([\eta]) := \int_{TM} \eta ~ d\mu. \]
Note that $\rho(\mu)$ is well-defined due to flow-invariance of $\mu$. Indeed, if $f:M\to\R$ is any smooth function then
\begin{align*}
 \int df~ d\mu & = \int \frac{d}{dt}\bigg|_{t=0}f(c_v(t)) ~ d\mu(v) = \frac{d}{dt}\bigg|_{t=0} \int f\circ \pi\circ \phi_F^t ~ d\mu \\
 & = \frac{d}{dt}\bigg|_{t=0} \int f\circ \pi ~ d\mu = 0 .
\end{align*}
The map $\rho$ is affine and continuous and the set $\MM(F)$ is convex and compact, such that the set
\[ B(F) := \rho(\MM(F))\subset H_1(M,\R) \]
is convex and compact. The convex polar of $B(F)$ is the convex, compact set given by
\[ B^*(F) := \big\{ [\eta]\in H^1(M,\R) : \la [\eta], h \ra \leq 1 ~ \forall h\in B(F) \big\} . \]
One can show that the relative interior of $B(F)$ and hence of $B^*(F)$ is open in $H_1(M,\R), H^1(M,\R)$, respectively.

The observations in the next lemma are crucial in Mather theory. They motivate using the elements in $\partial B^*(F)$ for the ``distance'' $d_{F-\eta}$ to obtain interesting minimizers, where the boundary
\[ \partial B^*(F) := \big\{ [\eta]\in B^*(F) ~|~ \exists h\in B(F) : ~ \la [\eta], h \ra = 1 \big\} . \]

\begin{lemma}\label{lemma critical value}
 Let $\eta$ be a closed 1-form on $M$
 \begin{enumerate}[(i)]
  \item\label{lemma critical value ii} If $[\eta]$ lies in the complement of $B^*(F)$, then $d_{F-\eta} \equiv -\infty$. In particular, there are no $l_{F-\eta}$-minimizers.
  
  \item\label{lemma critical value iv} If $[\eta]$ lies in $B^*(F)$, then $d_{F-\eta}$ is bounded.
  
  \item\label{lemma critical value i} If $[\eta]$ lies in the interior $B^*(F)-\partial B^*(F)$, then for any two points $x,y \in M$ there exists an $F$-geodesic minimizing the length functional $l_{F-\eta}$ on curves from $x$ to $y$ and there exist no $l_{F-\eta}$-minimizers of infinite length.
  
  \item\label{lemma critical value iii} If $[\eta]$ lies in the boundary $\partial B^*(F)$, then any vector $v$ in the support of some $\mu\in\MM(F)$ with $\rho(\mu)([\eta])=1$ defines an $F$-geodesic $c_v:\R\to M$ minimizing $l_{F-\eta}$ between any of its points. In particular, there exist $l_{F-\eta}$-minimizers of infinite length.
 \end{enumerate} 
\end{lemma}

\begin{proof}%[Sketch of the proof]
 \eqref{lemma critical value ii}. Let $\mu\in\MM(F)$ with $\rho(\mu)([\eta]) > 1$. By Birk\-hoff's ergodic theorem we find $v\in\supp\mu$ with
 \[ \lim_{T\to\infty} \frac{1}{T} \int_0^T \eta(\phi_F^tv) ~ dt \geq \int \eta ~ d\mu > 1. \]
 Hence, for $T\gg 1$ we find
 \[ l_{F-\eta}(c_v;[0,T])=T-\int_0^T \eta(\phi_F^tv) ~ dt \leq - \e T \]
 for some $\e>0$ and $l_{F-\eta}$ becomes unbounded from below.
 
 \eqref{lemma critical value iv}. An upper bound can be found using that $M$ is compact and $\eta$ is fixed. Assume that $d_{F-\eta}$ is not bounded from below. One can use the upper bound to find a closed $C^1$-curve in $M$ with negative $F-\eta$-length. Replacing this curve by a shortest closed geodesic in the same free homotopy class, the $F$-length decreases, while the $\eta$-integral stays unchanged by $\eta$ being closed. This shows that there is a closed $F$-geodesic $c:[0,T]\to M$ with negative $F-\eta$-length. Consider the measure $\mu\in\MM(F)$ evenly distributed on the orbit $\dot c[0,T]$. Then
 \[ \rho(\mu)([\eta]) = \frac{1}{T}\int_0^T \eta(\dot c) ~ dt =  \frac{1}{T}\left(T-l_{F-\eta}(c;[0,T]) \right) > 1, \]
 a contradiction.
 
 \eqref{lemma critical value i}. Given any pair $x,y\in M$ and an arc-length $C^1$-curve $c:[0,T]\to M$ from $x$ to $y$, choose a lift $\tilde c$ of $c$ to the universal cover $H$ and choose a primitive $u:H\to\R$ of the lift $\tilde \eta$ of $\eta$ to $H$. Replace $\tilde c$ by a minimizing geodesic segment $\tilde c_0$ with the same endpoints. The $F-\eta$-length of the projection $c_0$ of $\tilde c_0$ does not increase, as the integral $\int_c\eta=\int_{\tilde c} du$ depends in the universal cover only on the value $u(\tilde c(T))-u(\tilde c(0))$. Hence, we can restrict to $F$-geodesic segments from $x$ to $y$, when we look for $l_{F-\eta}$-minimizers. Let now $v\in S_FM$. Then for any limit measure $\mu$ of probabilities $\mu_{T_n}$ evenly distributed along the orbit segment $\dot c_v[0,{T_n}]$ as ${T_n}\to\infty$ we find
 \[ \lim_{n\to\infty} \frac{1}{T_n} \int_0^{T_n} \eta(\phi_F^tv) ~ dt = \lim_{n\to\infty} \int \eta ~ d\mu_{T_n} = \int \eta ~ d\mu = \rho(\mu)([\eta]) < 1. \]
 Hence, for $T\gg 1$ we find
 \[ l_{F-\eta}(c_v;[0,T])=T-\int_0^T \eta(\phi_F^tv) ~ dt \geq \e T \]
 for some $\e>0$. This shows that when looking for $l_{F-\eta}$-minimizers, we have to consider only $F$-geodesics of bounded length, which form a compact family. The claim follows.
 
 \eqref{lemma critical value iii}. Let $\mu\in\MM(F)$ with $\rho(\mu)([\eta])=1$, let $v\in \supp\mu$ and let $x\in M$ be arbitrary. The triangle inequality shows for any $a\leq b$, that
 \begin{align}\label{eqn mather semistatic}
  d_{F-\eta}(x,c_v(b)) \leq d_{F-\eta}(x,c_v(a)) + l_{F-\eta}(c_v;[a,b]) .  
 \end{align}
 Rewriting this inequality and integrating with respect to $\mu$, we find by the $\phi_F^t$-invariance of $\mu$ and $\rho(\mu)([\eta])=1$
 \begin{align*}
  0 & = \int d_{F-\eta}(x,.)\circ \pi \circ \phi_F^b ~ d\mu - \int d_{F-\eta}(x,.)\circ \pi \circ \phi_F^a ~ d\mu \\
  & \leq \int_a^b\int (F-\eta)\circ\phi_F^t ~ d\mu ~ dt = (b-a)\cdot \int (F-\eta) ~ d\mu = 0. 
 \end{align*}
 Hence, the integrated triangle inequality is an equality and for all points $v\in\supp\mu$. We obtain from \eqref{eqn mather semistatic} and again the triangle inequality, that
 \[ l_{F-\eta}(c_v;[a,b]) = d_{F-\eta}(x,c_v(b)) - d_{F-\eta}(x,c_v(a)) \leq d_{F-\eta}(c_v(a),c_v(b)) . \]
 The claim follows.
\end{proof}

By taking $x=c_v(b)$ in the proof of Lemma \ref{lemma critical value} \eqref{lemma critical value iii}, one finds the even stronger minimization property
\[ l_{F-\eta}(c_v;[a,b]) = - d_{F-\eta}(c_v(b),c_v(a)) \]
for the points $v$ in the support of $\mu$.

Lemma \ref{lemma critical value} suggests the following definitions.

\begin{defn}\label{def mather set}
 Let $[\eta]\in \partial B^*(F)\subset H^1(M,\R)$.
 \begin{enumerate}[(i)]
  \item An $F$-geodesic segment $c:[a,b]\to M$ with
  \[ l_{F-\eta}(c;[a,b])=d_{F-\eta}(c(a),c(b)) \]
  is called {\em $[\eta]$-semistatic on $[a,b]$}. The set of initial conditions of geodesics which are $[\eta]$-semistatic  on $\R$ is denoted by $\NN_{[\eta]}\subset S_FM$ and called the {\em $[\eta]$-\Mane set}.
  
  \item An $F$-geodesic segment $c:[a,b]\to M$ with
  \[ l_{F-\eta}(c;[a,b])=-d_{F-\eta}(c(b),c(a)) \]
  is called {\em $[\eta]$-static on $[a,b]$}. The set of initial conditions of geode\-sics which are $[\eta]$-static  on $\R$ is denoted by $\A_{[\eta]}\subset S_FM$ and called the {\em $[\eta]$-Aubry set}.
  
  \item A measure $\mu\in \MM(F)$ is called {\em $[\eta]$-minimizing}, if $\rho(\mu)([\eta]) = 1$. The union of the supports of $[\eta]$-minimizing measures is denoted by $\M_{[\eta]}\subset S_FM$ and called the {\em $[\eta]$-Mather set}. 
 \end{enumerate}
\end{defn}

\begin{remark}\label{remark mane inclusions}
 As $\eta$ is exact in the universal cover $H$, one can easily show that
 \[ \M_{[\eta]} \subset \A_{[\eta]} \subset \NN_{[\eta]} \subset Dp(\M) \subset S_FM . \]
 Note that $\M_{[\eta]}$ is always non-empty.
\end{remark}

\begin{remark}
 In Subsection \ref{section min geod} we remarked that the notion of minimal geodesics can also be defined in the quotient $M$, via homotopic minimization. One could define a curve $c:\R\to M$ to be homologically minimizing, if for all $a<b$ and all curves $c':[0,1]\to M$ with $c'(0)=c(a),c'(1)=c(b)$, which are homologous to $c|_{[a,b]}$, $c|_{[a,b]}$ has minimal $F$-length. The union of all homologically minimizing geodesics should correspond to the union of all curves $c$, which are $[\eta]$-semistatic with respect to some $[\eta]\in\partial B^*(F)$. One could also define homological minimization in the abelian cover. See Proposition 2, p.\ 182 in \cite{mather}.
\end{remark}

We saw in Lemma \ref{lemma critical value}, that any orbit in the support of a minimizing measure is semistatic, which can be written as $ \M_{[\eta]} \subset \NN_{[\eta]}$. The converse is also true. In this sense, the minimizing measures are precisely those measures supported on minimizing orbits.

\begin{prop}
 If $\mu$ is a $\phi_F^t$-invariant probability measure supported in $\NN_{[\eta]}$ for $[\eta]\in\partial B^*(F)$, then $\rho(\mu)([\eta])=1$, i.e.\ $\mu$ is an $[\eta]$-minimizing measure.
\end{prop}

\begin{proof}
 Let us assume that $\mu$ is ergodic, the general case following from the fact that $\rho$ is affine and arbitrary invariant probabilities are convex combinations of ergodic ones. Choose a $\mu$-typical $v\in\NN_{[\eta]}$, so that
 \begin{align*}
  \rho(\mu)([\eta]) & = \int \eta d\mu = \lim_{T\to\infty} \frac{1}{T}\int_0^T \eta \circ \phi_F^tv dt \\
  & = \lim_{T\to\infty} \frac{1}{T} \left(T-l_{F-\eta}(c_v;[0,T]) \right) \\
  & = 1 - \lim_{T\to\infty} \frac{1}{T} d_{F-\eta}(c_v(0),c_v(T)) \\
  & = 1.
 \end{align*}
 Here we used that $d_{F-\eta}$ is bounded.
\end{proof}

\begin{remark}\label{remark stretch}
 There is a relationship between the above notions and the geodesic stretch, see e.g.\ \cite{knieper strech}. Given $\mu\in \MM(F)$, we can define its {\em geodesic stretch} by
 \[ S(\mu) := \int_{S_F M} \lim_{t\to\infty} \frac{d_g(\widetilde{c_v}(0),\widetilde{c_v}(t))}{t} d\mu(v) , \]
 where $\widetilde{c_v}:\R\to H$ denotes an arbitrary lift of the $F$-geodesic $c_v:\R\to M$. The larger the stretch of an element $\mu\in\MM(F)$, the shorter do the geodesics in the support of $\mu$ become. Hence, one should be able to relate the property ``maximizing the stretch'' to ``supported in $Dp(\M)$''. Hence, we replace maximizing $\rho(.)([\eta])$ by maximizing $S$. The advantage of using the stretch is one finds a larger class of ``minimizing measures''. We shall study ``periodic measures'' in Subsection \ref{section hyperbolic cylinder}. Note in this connection also \cite{boyland} associating to measures $\mu\in\MM(F)$ with positive stretch a rotation measure $\hat\rho(\mu)\in\MM(g)$, provided that the Morse Lemma holds. The dynamics in $\supp\mu$ ``shadow'' the dynamics in $\supp\hat \rho(\mu)$ via Theorem \ref{morse lemma}.
\end{remark}

\subsection{Weak KAM solutions in Mather theory}\label{section weak KAM in mather}

In Section \ref{section weak KAM} we used the distance $d_F$ in $H$ to define dominated functions and weak KAM solutions $u:H\to\R$ (Definitions \ref{def dominated} and \ref{def weak KAM}). We take the analogous definitions using the distance $d_{F-\eta}$ for elements $[u]\in C^0(M)/_\sim$, which yields {\em $\eta$-dominated functions} and {\em $\eta$-weak KAM solutions}. Note that this definition depends on the representative $\eta\in[\eta]$. The calibration condition becomes
\[ l_{F-\eta}(c;[a,b]) = d_{F-\eta}(c(a),c(b)) . \]

An instance of an $\eta$-weak KAM solution is the ``Busemann function'' associated to a $[\eta]$-semistatic ray $c:\R_-\to M$, given by
\[ [b_c] := \lim_{t\to \infty} [d_{F-\eta}(c(-t),.)] . \]
Using such ``Busemann functions'', one can see that every $[\eta]$-semistatic ray $c$ is calibrated with respect to some $\eta$-weak KAM solution $[b_c]$.

We already remarked that $[\eta]$-semistatic curves lift to minimal geode\-sics in the universal cover $H$. Similarly, one can lift $\eta$-weak KAM solutions to weak KAM solutions in $\W(H,F)$. Note that a closed 1-form $\eta$ on $M$ has a lift $\tilde\eta$ in $H$ and $\tilde \eta$ has a primitive $h:H\to\R$ with $dh=\tilde\eta$.

\begin{prop}\label{prop lifting weak KAM}
 If $\eta$ is a closed 1-form on $M$ with $[\eta]\in \partial B^*(F)$, $h:H\to\R$ a primitive of the lift $\tilde \eta$ to $H$ and if $[u]\in C^0(M)/_\sim$ is $\eta$-dominated, then
 \[ [\tilde u] := [h + u \circ p] \in \Dom(H,F)  \]
 and the $[\tilde u]$-calibrated curves are the lifts of the $[u]$-calibrated curves. In particular, if $[u]$ is an $\eta$-weak KAM solution, then 
 \[ [\tilde u] \in \W(H,F) .\]
\end{prop}

\begin{proof}
 Observe that for $x,y\in H$ and the covering map $p:H\to M$ we have
 \begin{align*}
  d_{F-\eta}(px,py) & = \inf_{\tau\in\Gam} d_{F-\tilde\eta}(x,\tau y) = \inf_{\tau\in\Gam} d_F(x,\tau y) - h(\tau y) + h(x) \\
  & \leq d_F(x, y) - h(y) + h(x) .
 \end{align*}
 The first claim follows. For the latter claim, observe that if $c:[a,b]\to M$ is $[u]$-calibrated and if $\tilde c:[a,b]\to H$ is a lift, then % by the calculation above
 \begin{align*}
  l_F(\tilde c;[a,b]) & = l_F(c;[a,b]) =u\circ c(b)-u\circ c(a) + \int_a^b\eta(\dot c) dt \\
  & = u\circ p\circ \tilde c(b)-u\circ p\circ \tilde c(a) + \int_a^b dh( \dot{\tilde c}) dt \\
  & = u\circ p\circ \tilde c(b)-u\circ p\circ \tilde c(a) + h\circ \tilde c(b)- h\circ\tilde c(a)  .
 \end{align*}
 This proves the second claim.
\end{proof}

Using the analogous version of Proposition \ref{lemma weak KAM} for $d_{F-\eta}$ in $M$, one sees that every $[\eta]$-semistatic ray is simple in $M$, i.e.\ has no self-intersections. This is by far not true for all minimal geodesics. In particular, the family of semistatic geodesics is in general much smaller than the set of minimal geodesics. Furthermore, not every weak KAM solution $[u]\in \W(H,F)$ is the lift of some weak KAM solution $[u]$ from $M$ using some closed 1-form and in general $[\tilde u]$ is not directed, if the Morse Lemma holds. We will show, however, that $[\tilde u]$ is always directed, if $M$ is the 2-torus, see Subsection \ref{section mather torus}.

We defined the Aubry set and the \Mane set using (semi)static curves. There is another characterization using weak KAM solutions. Compare this to Proposition \ref{char unstable via weak KAM}.

\begin{prop}\label{weak KAM charact aubry set}
 An $F$-geodesic is $[\eta]$-static (semistatic), if and only if it is calibrated with respect to all (one) $\eta$-weak KAM solution.
\end{prop}

\subsection{Mather's avarage action}\label{section av action}

Let us recall another way to define the sets $B(F)$ and $B^*(F)$ due to Mather \cite{mather}. Namely, given the Finsler metric $F$, one can define the ``Tonelli'' Lagrangian
\[ L_F:= \frac{1}{2}F^2 . \]
Write $\MM(L_F)$ for the probability measures in $TM$, which are invariant under the Euler-Lagrange flow of $L_F$ (which equals the geodesic flow of $F$) and satisfy $\int L_F d\mu <\infty$. We can define the rotation vector $\rho:\MM(L_F)\to H_1(M,\R)$ as above. {\em Mather's $\al$- and $\beta$-functions} are given by
\begin{align*}
 & \al_F:H^1(M,\R)\to \R, \\
 & \al_F([\eta]) := -\inf \left\{ \int L_F - \eta ~ d\mu : \mu \in \MM(L_F) \right\} 
\end{align*}
and
\begin{align*}
 & \beta_F:H_1(M,\R)\to \R, \\
 & \beta_F(h) := \inf \left\{ \int L_F ~ d\mu : \mu \in \MM(L_F) , \rho(\mu)=h \right\} .
\end{align*}
The functions $\al_F,\beta_F$ are convex and superlinear. $\al_F$ is the Fenchel transform of $\beta_F$ and vice versa. Moreover, using the homogeneity of $L_F$ one finds $\al(s\cdot [\eta])=s^2 \al([\eta])$ and $\beta(sh)=s^2\beta(h)$ for $s\geq 0$. One shows that
\begin{align*}
 B(F) = \{ \beta_F \leq 1/2\} , \qquad  \ B^*(F) = \{ \al_F \leq 1/2\} .
\end{align*}
The Mather set $\M_{[\eta]}$ is given by the union of supports of measures realizing the infimum in the definition of $\al_F([\eta])$. Similarly, one defines a Mather set $\M^h$ as the union of supports of measures realizing the infimum in the definition of $\beta_F(h)$. If $\al_F([\eta])=1/2=\beta_F(h)$, then $\M_{[\eta]}, \M^h\subset S_FM$.

Let us write
\[ \sig_F(h) := \sqrt{2 \beta_F(h)} . \]
A classical way to obtain $\sig_F$ directly is given by
\[ \sig_F(h) = \inf \left\{ \sum |r_i| \cdot l_F(c_i) \right\} , \]
where the infimum is taken over all $r_i\in\R$ and real Lipschitz 1-cycles $c_i$ in $M$ with $\sum r_i [c_i] = h$ \cite{massart-diss}. The function
\[ \sig_F:H_1(M,\R)\to \R \]
is called the {\em stable norm of $F$} (it defines a non-reversible, convex norm on $H_1(M,\R)$). Hence, the set $B(F)$ is given by the 1-ball of the stable norm. If $F$ is Riemannian, then one can use the Riemannian $q$-dimensional volume to define $\sig_F$ on the higher homology groups $H_q(M,\R)$.

Mather's $\al$-function can be recovered using $\eta$-dominated functions, recalling the dual Finsler metric $F^*$ defined in Section \ref{section weak KAM}.

\begin{prop}[\cite{CIPP}]\label{prop char al-fctn via weak KAM}
 \[ \al_F([\eta]) = \inf_{\eta'\in [\eta]} \max_{x\in M} ~ \frac{1}{2}(F^* \circ \eta'(x))^2 . \] 
\end{prop}

Here, we can assume that the cohomology class $[\eta]$ consists of smooth 1-forms cohomologous to the smooth 1-form $\eta$.

\section{Rational directions}\label{section rational}

In this section we study rays and minimal geodesics in finite distance of periodic $g$-geodesics. Hence, these rays are expected to admit some form of periodicity.

\subsection{Structure of rational rays in 2 dimensions}\label{section rational dir 2-dim}

We start with the 2-dimensional case, where the classical results are due to Morse \cite{morse} and Hedlund \cite{hedlund}. Let $M$ be a closed, orientable surface of genus $\mathfrak{g}\geq 1$. The Morse Lemma holds for the constant curvature metrics $g$ described in Subsection \ref{section model geom}. Let us fix a non-trivial, prime element $\tau\in \Gam$. By $M$ being closed, $\tau$ represents a free homotopy class containing shortest closed $F$-geodesics. We fix a $\tau$-periodic $g$-geodesic $\g_\tau \subset H$ and set
\[ \xi_\tau:= \g_\tau(-\infty) . \]
In particular, $\xi_\tau \in \Pi_M\subset\se\cong\Gro(H,g)$ and in the case of genus $\mathfrak{g}\geq 2$ (i.e.\ $(H,g)$ is the \Poincare disc), $\xi_\tau$ is the stable fixed point of the hyperbolic transformation $\tau$; if $\mathfrak{g}=1$, then $\xi_\tau$ has rational or infinite slope.

We consider the set of minimal geodesics $\M(\g_\tau)\subset S_FH$ in finite distance of $\g_\tau$ and the set of rays $\RR_-(\xi_\tau)$. We already saw in Remark \ref{remark shorst closed geod are minimal}, that the shortest closed geodesics in the homotopy class $\tau$ lift to minimal geodesics in the universal cover $H$. Hence, the set of $\tau$-periodic minimal geodesics
\[ \M^{per}(\g_\tau) := \{ v\in \M(\g_\tau) : \tau \circ c_v(\R)=c_v(\R) \} \]
is non-empty and consists of lifts of shortest closed geodesics in homotopy classes $\tau^k, k\geq 1$. Moreover, the set $\M^{per}(\g_\tau)$ is closed and $\tau$-invariant and defines a lamination of $H$ by minimal geodesics ``pa\-ral\-lel'' to $\g_\tau$. In particular, we can speak of pairs $c_0,c_1$ of neighboring $\tau$-periodic minimal geodesics (i.e.\ in the strip between $c_0,c_1$ there are no further $\tau$-periodic minimal geodesics). If $M=\T$, the set $\M^{per}(\g_\tau)$ is invariant under the whole group $\Gam\cong\Z^2$. The following theorem summarizes further structure properties of $\M(\g_\tau)$ and $\RR_-(\xi_\tau)$.

\begin{thm}[Morse, Hedlund]\label{morse periodic}
 \begin{enumerate}[(i)]
  \item \label{morse periodic item 2} If $v\in\RR_-(\xi_\tau)$, then $c_v(-t)$ is asymptotic to some $\tau$-periodic minimal geodesic $c_0(\R)$, as $t\to\infty$. The analogous statement holds for $v\in\RR_+(\g_\tau(\infty))$.
  
  \item \label{morse periodic item 3} If $v\in\RR_-(\xi_\tau)$ and $\pi v = \pi w$ for some $w\in \M^{per}(\g_\tau)$, then $v=w$.
  
  \item \label{morse periodic item 2a} For any $v\in \M(\g_\tau)-\M^{per}(\g_\tau)$, the geodesic $c_v$ is heteroclinic between a pair $c_0,c_1$ of neighboring $\tau$-periodic minimal geodesics. Conversely, for any pair of neighboring minimal geodesics $c_0,c_1$ there exist heteroclinic minimal geodesics of both possible asymptotic behaviors.

  \item \label{morse periodic item 4} If $S\subset H$ is a connected component of $H-\pi(\M^{per}(\g_\tau))$ and if $c_0$ is a minimal geodesic from $\M^{per}(\g_\tau)$ forming a component of the boundary $\partial S$, then for any $x\in S$ there exists a ray $c:\R_-\to H$ with $c(0)=x$, which is asymptotic to $c_0(\R)$.
 \end{enumerate}
\end{thm}

Note that in the case $M=\T$, the set $\RR_-(\xi_\tau)$ is $\Gam$-invariant. In this case, each component $S$ of $H-\pi(\M^{per}(\g_\tau))$ is a strip bounded by a pair $c_0,c_1$ of neighboring $\tau$-periodic minimal geodesics $c_0,c_1$. Assume that $c_1$ lies left (or above) $c_0$ in the usual orientation of $H=\R^2$. For $i=0,1$ let $\RR_-^i(\xi_\tau)$ be the set of rays in $\RR_-(\xi_\tau)-\M^{per}(\g_\tau)$ being asymptotic to $c_i$, in each of the strips $S$. Then
\[ \M^{per}(\g_\tau) \cup \RR_-^i(\xi_\tau) , \qquad i = 0,1 \]
form two $\Gam$-invariant sets, admitting no intersecting rays in $\T$. We can relate these sets to directed weak KAM solutions.

\begin{prop}
 If $M=\T$, then $\M^{per}(\g_\tau) \cup \RR_-^i(\xi_\tau)$ consist precisely of the calibrated rays of the directed weak KAM solutions $[u_i] \in \delta_F^{-1}(\xi_\tau)$ given by Proposition \ref{bounding horofunctions}.
\end{prop}

We refer to \cite{morse} and \cite{hedlund} for the proof of Theorem \ref{morse periodic}. See also \cite{paper1}. We shall only sketch here, how item \eqref{morse periodic item 2} of Theorem \ref{morse periodic} implies item \eqref{morse periodic item 3}, as the argument will appear again later in Subsections \ref{section hyperbolic cylinder} and \ref{section rays-paper}.

\begin{proof}[Sketch of the proof of Theorem \ref{morse periodic} \eqref{morse periodic item 3}]
 Let $v,w$ be as described in the statement of the theorem and write $c_w=c_0$. By item \eqref{morse periodic item 2} of Theorem \ref{morse periodic}, both backward rays $c_v(t),\tau^{-1}c_v(t)$ are asymptotic as $t\to -\infty$ to a single $\tau$-periodic minimal geodesic $c_1$. Let $T<0$ such that $\dot{\widehat{\tau c_i}}(0)=\dot c_i(T)$ for $i=0,1$ -- note that we have the same period $|T|$ for both $c_0$ and $c_1$ due to minimality of $c_i$. Due to the periodicity of $c_1$ we find $S\ll -1$, such that $d_F(\tau^{-1}c_v(S+T),c_v(S)) \approx 0$. Assuming $v\neq w$, we find some $\e,\delta>0$, such that
 \[ d_F(c_v(-\delta),c_0(\delta)) \leq 2\delta-\e \]
 by shortening the vertex at $c_0(0)=c_v(0)$. Hence, using the triangle inequality and minimality of $\tau^{-1}c_v$, we obtain a contradiction:
 \begin{align*}
  -T-S & = d_F(\tau^{-1}c_v(S+T),\tau^{-1}c_v(0)) \approx d_F(c_v(S), c_0(-T)) \\
  & \leq d_F(c_v(S),c_v(-\delta)) + d_F(c_v(-\delta),c_0(\delta)) + d_F(c_0(\delta), c_0(-T)) \\
  & \leq -S-\delta  +2\delta-\e - T-\delta \\
  & = -T-S-\e .
 \end{align*}
\end{proof}

Observe that the structure results in Theorem \ref{morse periodic} resemble the dynamics in the saddle connection seen in the phase portrait of the simple pendulum, which was discussed in Subsection \ref{section rot torus}. See also Figure \ref{fig_structure-rational}.

\begin{figure}\centering%[!htb]
 \includegraphics[scale=1]{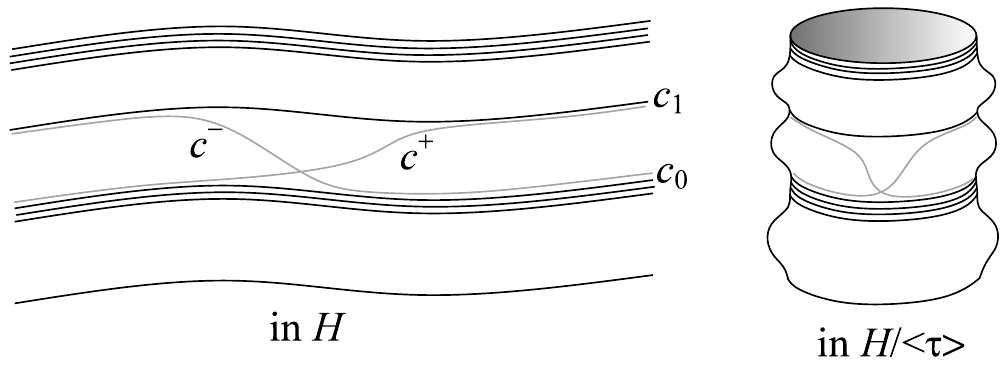}
 \caption{Some geodesics from $\M(\g_\tau)$ for $\tau\in\Gam$. There are some strips foliated by $\tau$-periodic geodesics and in between them there are gaps overstretched by heteroclinics. On the right one can see a geometric situation generating closed and heteroclinic minimals. \label{fig_structure-rational}}
\end{figure}

\subsection{Oscillating local minimizers in 2 dimensions}\label{section multibump}

In the previous subsection we saw that in the two-dimensional case, there can occur minimal geodesics $c_0,c_1$ invariant under a deck transformation $\tau$, bounding a strip with no further $\tau$-periodic minimizers; we called $c_0,c_1$ neighboring. We shall assume that $c_1$ lies left of $c_0$ (or above $c_0$). In between $c_0,c_1$, there exist heteroclinic minimizers of two classes, namely
\begin{align*}
 \M^+(c_0,c_1) & := \{ v\in \M(\g_\tau) : d_g(c_0(\R),c_v(-t)) , d_g(c_1(\R),c_v(t)) \stackrel{t\to\infty}{\to} 0 \} , \\
 \M^-(c_0,c_1) & := \{ v\in \M(\g_\tau) : d_g(c_1(\R),c_v(-t)) , d_g(c_0(\R),c_v(t)) \stackrel{t\to\infty}{\to} 0 \} .
\end{align*}
Both classes consist of asymptotic minimizers and hence, by Lemma \ref{crossing minimals}, there are no intersecting geodesics in each of the classes. Assume for the moment that one of the classes, say $\M^+(c_0,c_1)$ projects onto the open strip between $c_0,c_1$ denoted by $S(c_0,c_1)\subset H$. Then $\M^+(c_0,c_1)$ forms a continuous graph over the strip $S(c_0,c_1)$ and defines a $C^1$ weak KAM solution $[u]$ in $S(c_0,c_1)$ in the sense of Definition \ref{def weak KAM}. This would be a desirable situation. Observe that, if $c_0,c_1$ are hyperbolic closed orbits for the geodesic flow $\phi_F^t$ in $S_FM$, then $\M^\pm(c_0,c_1)$ are parts of the (un)stable manifolds of $c_0,c_1$. It is known that these intersect transversely for generic choices of $F$, prohibiting the classes $\M^\pm(c_0,c_1)$ to form continuous invariant graphs. Hence, the following condition, called the {\em gap condition} is fulfilled for many choices of $F$:
\begin{align}\label{eqn gap-cond}
 \pi(\M^+(c_0,c_1)) \neq S(c_0,c_1) \qquad \& \qquad \pi(\M^-(c_0,c_1)) \neq S(c_0,c_1) .
\end{align}
The following theorem is proved in \cite{paper2}; an earlier version is due to Bolotin and Rabinowitz \cite{bol_rab}. It shows that the gap condition implies the existence of oscillating dynamics in the geodesic flow $\phi_F^t$, which is contrary to integrable behavior. See Figure \ref{fig_multibump}.

\begin{thm}\label{thm multibump}
 Let $M$ be a closed, orientable surface of genus $\mathfrak{g}\geq 1$ and $F$ any Finsler metric on $M=H/\Gam$. Suppose there exists $\tau\in\Gam-\{\id\}$ and in the universal cover $H$ a pair of neighboring $\tau$-periodic minimal geodesics $c_0,c_1:\R\to H$ satisfying the gap condition \eqref{eqn gap-cond}. Then there exists an open disc $D\subset S(c_0,c_1)$, an integer $N\geq 1$ and for each given finite sequence $\kappa = \{k_1<...<k_n\}$ of integers and each pair $i=(i_-,i_+)\in \{0,1\}^2$ a locally minimizing geodesic $c_{\kappa,i} : \R\to S(c_0,c_1)$ with the following properties:
 \begin{itemize}
  \item $c_{\kappa,i}(\pm t)$ enters the ends of $S(c_0,c_1)$ and is asymptotic to $c_{i_\pm}(\R)$ as $t\to \infty$,
  
  \item $c_{\kappa,i}$ is disjoint from the translated discs $\tau^{k_i N} D$ for $i=1,...,n$ and passes two successive discs on opposite sides.
 \end{itemize}
\end{thm}

\begin{figure}\centering%[!htb]
 \includegraphics[scale=0.4]{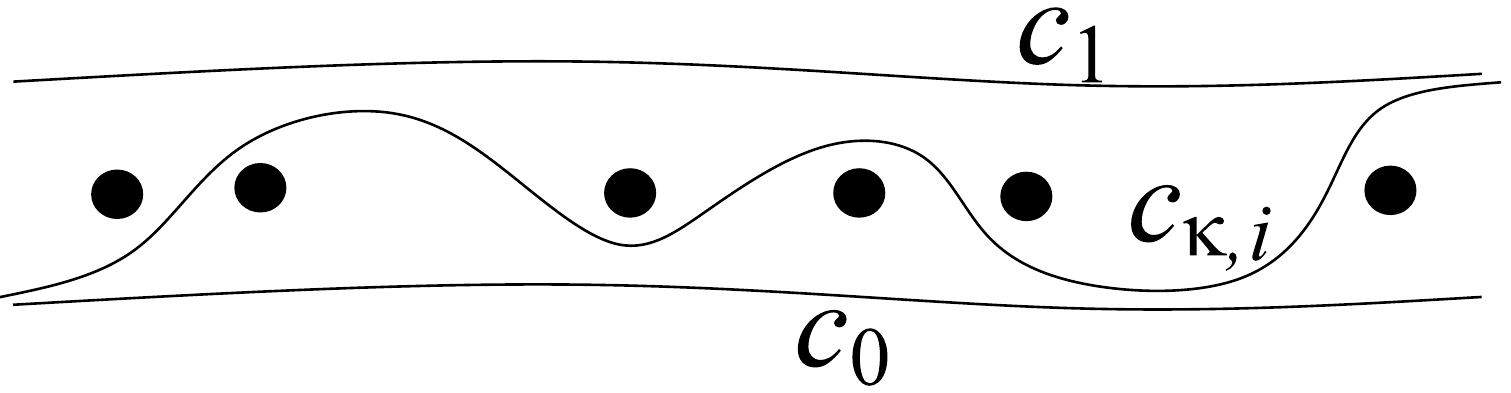}
 \caption{The geodesic $c_{\kappa,i}$ in Theorem \ref{thm multibump}. The black dots are the successive translates of the disc $D$. \label{fig_multibump}}
\end{figure}

We will see (Subsection \ref{section GKOS}) that the construction of minimal geode\-sics with the behavior of the $c_{\kappa,i}$ in Theorem \ref{thm multibump} is impossible. Locally minimizing means that the geodesic $c_{\kappa,i}(\R)$ has an open neighborhood $U \subset S(c_0,c_1)$, such that $c_{\kappa,i}$ is minimizing between any of its points under all curves lying entirely in $U$. In particular, $c_{\kappa,i}$ has no conjugate points.

Applied to $M=\T$, Theorem \ref{thm multibump} shows the following. Either the set $Dp(\M(\g_\tau))$ contains a $\phi_F^t$-invariant $C^0$-graph, producing a $C^1$ weak KAM solution $[u]$ in $(\T,F)$, or the geodesic flow $\phi_F^t$ is chaotic in the sense that it depends sensitively on initial conditions. Indeed, for fixed $i=(i_-,i_+)$ the initial conditions $Dp(\dot c_{\kappa,i}(-t))$ with $t\gg -1$ all lie close to each other in $S_F\T$, but for varying $\kappa$ have distinguishable long-time behavior unter $\phi_F^t$.

\subsection{2-dimensional cylinders}

So far, we have in this section assumed that the surface $M$ is compact. Here we treat the non-compact case where
\[ M = \R/\Z \times \R \]
is the 2-dimensional cylinder. Recall that, by assumption, the metrics $F$ and $g$ on $M$ are equivalent via some constant $c_F\geq 1$, i.e.\
\[ \frac{1}{c_F}|.|_g \leq F \leq c_F |.|_g . \]
If both ends of the cylinder become wide enough with respect to $g$ one can infer that there exist shortest closed geodesics in the prime homotopy class of $M$, lifting to minimal geodesics in the universal cover $H=\R^2$ (see Remark \ref{remark shorst closed geod are minimal}). Such minimal geodesics wind asymptotically around the cylinder in the $\R/\Z$-direction. Using the Morse Lemma, one can moreover study many minimal geodesics in the universal cover, if $g$ has strictly negative curvature. If the ends of the cylinder become thin, one could suspect that almost all geodesics are recurrent.

In this subsection, we shall assume that $g$ is the standard Euclidean metric. Of course, one can find minimal geodesics in $M$ going from one end of $M$ to the other (i.e.\ along the $\R$-direction), but it is not clear whether there are any minimal geodesics winding around the $\R/\Z$-direction.

We can prove the following theorem, writing $e_1=(1,0)\in \R^2$.

\begin{thm}\label{thm exists winding} 
 If the constant $c_F$ in the equivalence of $F$ and $g$ is sufficiently close to $1$, then there exists $\e>0$ depending only on $c_F$ and a minimal geodesic $c:\R\to M$, such that
 \[ \liminf_{T\to \infty} \frac{1}{T}\int_{-T}^0 \la \dot c(t), e_1 \ra dt \geq \e , \qquad \liminf_{T\to\infty} \frac{1}{T}\int_0^T \la \dot c(t), e_1 \ra dt \geq \e. \]
\end{thm}

In Section \ref{section twist map} we observe the connection between monotone twist maps of the cylinder $C=\R/\Z\times \R$ and geodesic flows of Finsler metrics on the torus $\T$. For geodesic flows on the cylinder $M$ there is the analogous connection to monotone twist maps of the plane $\R\times \R$. Loss in the periodicity in the second component of $M$ corresponds to losing periodicity in the first component of $C$. In this setting, Kunze and Ortega have established the existence of certain orbits, for certain types of monotone twist maps of the plane, see \cite{kunze} and the references therein. We will prove Theorem \ref{thm exists winding} by arguments similar to the ones leading to Theorem 14 in \cite{kunze}. A particular open question building on Theorem \ref{thm exists winding} and \cite{kunze} would be what assumptions on $F,g$ lead to recurrent geodesics in $(M,F)$.

\begin{proof}[Proof of Theorem \ref{thm exists winding}]
 In the following, we denote by
 \[ e_F(c[a,b]) = \int_a^b F(\dot c(t))^2dt \]
 the $F$-energy of $C^1$ curves $c:[a,b]\to \R^2 = H$.
 
 Choose some $X=(X_1,X_2)\in \R^2$ with $X_i >0$ and let $c=(c_1,c_2):[0,L]\to \R^2$ be an $l_F$-minimal segment with $c(0)=-X, c(L)=X$, such that $F(\dot c)\equiv 1$. It follows that each subsegment of $c$ is minimal with respect to $e_F$ under all curve segments with the same end points and the same connection time.
 
 Let $t_0, t_1\in [0,L]$, such that the intervals $[t_0,t_0+2],[t_1,t_1+1]$ lie in $[0,L]$ and are disjoint. Consider the straight Euclidean segments with constant $|.|$-speed, writing $e_1=(1,0)$
 \begin{align*}
  & \sig_0:[0,1]\to\R^2, && \sig_0(0)=c(t_0), && \sig_0(1)=c(t_0+2)-e_1, \\
  & \sig_1:[0,2]\to\R^2, && \sig_1(0)=c(t_1)-e_1, && \sig_1(2)=c(t_1+1).
 \end{align*}
 We assume e.g. $t_0+2\leq t_1$, the other case being analogous. The new curve
 \[ \hat c:[t_0,t_1+1]\to\R^2, \qquad \hat c = \sig_0 * (c|_{[t_0+2,t_1]}-e_1) * \sig_1 \]
 connects $c(t_0)$ to $c(t_1+1)$ in time $t_1+1-t_0$. By $c$ being minimizing, $F(\dot c)=1$, the uniform equivalence of $F$ and $|.|$ and invariance of $F$ under $x\mapsto x+e_1$ we obtain
 \begin{align*}
  &~ 2 +  e_F(c[t_0+2,t_1]) + 1 \\
  = &~ e_F(c[t_0,t_1+1]) \\
  \leq &~ e_F(\hat c[t_0,t_1+1]) \\
  = &~ e_F( \sig_0[0,1]) +  e_F(c[t_0+2,t_1]) +  e_F( \sig_1[0,2]) \\
  \leq &~ c_F^2e_{|.|}( \sig_0[0,1]) +  e_F(c[t_0+2,t_1]) +  c_F^2e_{|.|}( \sig_1[0,2]).
 \end{align*}
 If $\sig:[a,b]\to\R^2$ is a straight Euclidean segment, then we find
 \[ e_{|.|}(\sig[a,b]) = \frac{|\sig(b)-\sig(a)|^2}{b-a} \]
 and hence by definition of $\sig_0, \sig_1$
 \begin{align*}
  &~ 3 \\
  \leq &~ c_F^2 \left( e_{|.|}( \sig_0[0,1]) + e_{|.|}( \sig_1[0,2]) \right) \\
  = &~ c_F^2 \left( | c(t_0+2) - c(t_0)-e_1|^2 + \frac{1}{2} |c(t_1+1)-c(t_1)+e_1|^2 \right) \\ 
  = &~ c_F^2 \bigg( | c(t_0+2) - c(t_0) |^2 -2[c_1(t_0+2)-c_1(t_0)] + 1 \\
  &\qquad + \frac{1}{2} |c(t_1+1)-c(t_1)|^2+[c_1(t_1+1)-c_1(t_1)]+\frac{1}{2} \bigg) ,
 \end{align*}
 which is equivalent to
 \begin{align*}
  &~ 6/c_F^2-3 + 4 [c_1(t_0+2)-c_1(t_0)] \\
  \leq &~  2 | c(t_0+2) - c(t_0) |^2  +  |c(t_1+1)-c(t_1)|^2\\
  &\qquad +2 [c_1(t_1+1)-c_1(t_1)] ,
 \end{align*}
 Due to uniform equivalence of $F$ and $|.|$, $F(\dot c)=1$ and $l_F$-minimality of $c$ we have for all $t\in [0,L]$ and $T\in [0,L-t]$
 \begin{align}\label{c approx euc-arc length}
   T/c_F & = d_F(c(t),c(t+T))/c_F  \nonumber \\
   & \leq |c(t+T)-c(t)| \\ \nonumber
   & \leq c_F d_F(c(t),c(t+T)) = c_F T .
 \end{align}
 This shows
 \begin{align*}
  &~ 6/c_F^2-3 + 4 [c_1(t_0+2)-c_1(t_0)] \leq 9 c_F^2 +2 [c_1(t_1+1)-c_1(t_1)] ,
 \end{align*}
 and if $c_F\in [1,1.016]$, then one finds
 \begin{align*}
  &~ 6/c_F^2-3 - 9 c_F^2 \geq -6.5.
 \end{align*}
 Hence, if $[c_1(t_0+2)-c_1(t_0)]\geq 7/4$, then
 \begin{align*}
  1/4 & \leq \frac{1}{2}\left(6/c_F^2-3-9 c_F^2 + 4 [c_1(t_0+2)-c_1(t_0)]\right) \\
  & \leq c_1(t_1+1)-c_1(t_1) .
 \end{align*}
 
 Let us see that for good choices of $X$ we can assume the existence of $t_0\in [0,L]$ with $[c_1(t_0+2)-c_1(t_0)]\geq 7/4$. We assume the contrary, i.e.
 \begin{align}\label{less 7/4} 
  [c_1(t+2)-c_1(t)]\leq 7/4 \qquad \forall t\in [0,L-2] .
 \end{align}
 Using \eqref{c approx euc-arc length} and Pythagoras' theorem we then find\footnote{Here we have to work a bit more: \eqref{less 7/4} is an assumption on $c_1(t+2)-c_1(t)$, while we need an estimate for its absolute value...}
 \begin{align*}
  |c_2(t+2)-c_2(t)|^2 & = |c(t+2)-c(t)|^2 - |c_1(t+2)-c_1(t)|^2 \\
  & \geq 4/c_F^2 - (7/4)^2 . 
 \end{align*}
 Again, if $c_F\in [1,1.016]$, then $\e:= 4/c_F^2 - (7/4)^2>0$. Hence, for $c_F$ sufficiently close to $1$ the $c_2$-part of $c$ grows linearly like $\e/2$ in time $1$. If $X$ is chosen to lie close to the $x_1$-axis in $\R^2$, then under the assumption \eqref{less 7/4}, $c$ cannot connect $-X$ to $X$. 
 
 We saw that for $X$ close to the $x_1$-axis we obtain a minimal geodesic segment $c$ with uniformly positive average displacement along the $\R/\Z$-factor of $M$, i.e. at least $1/4$ in time $1$. We find for each such $X$ with $X_2>0$ a point of intersection in $c[0,L]\cap \R\times \{0\}$ and applying deck-transformations of the covering $\R^2\to M$ yields a minimal geodesic with initial point in $[0,1]\times \{0\}$, having positive average displacement along the $\R/\Z$-factor. The theorem follows.
\end{proof}

\begin{remark}
\begin{enumerate}
 \item Several arguments in the preceding proof are a bit crude. Being more careful, one could possibly obtain more information on the constructed minimal geodesic, as well as on the parameter range for $c_F$, which was
 \[ c_F\in [1,1.016]. \]
 
 \item Note that the topology of $M=\R/\Z\times \R$ and $\dim M=2$ were only used in few places. Hence, one might try to generalize Theorem \ref{thm exists winding} to more general manifolds having a Euclidean background structure.
\end{enumerate}
\end{remark}

Let us prove the following lemma, which might render some arguments more precise. It shows that minimal segments cannot wander to far from model geodesics in $\R^n$.

\begin{lemma}
 Let $F$ be any Finsler metric on $\R^n$ uniformly equivalent to the Euclidean norm $|.|$ via a constant $c_F\geq 1$. If $c:[0,1]\to \R^n$ minimizes the $F$-length between its endpoints and if $F(\dot c)=\const$, then $c(1/2)$ lies $R$-close to the midpoint $\frac{1}{2}(c(1)+c(0))$ (with respect to the Euclidean distance), where
 \[ R = |c(1)-c(0)| \cdot \frac{\sqrt{c_F^2-1}}{2} . \]
\end{lemma}

\begin{proof}
 $c$ minimizes the energy between its endpoints in connection time $1$. If
 \begin{align*}
  & \sig(t)=(1-t)c(0)+tc(1) , && t\in [0,1] , \\
  & \sig_0(t)=(1-2t)c(0)+2tc(1/2) , && t\in [0,1/2] , \\
  & \sig_1(t)=(1-2t)c(1/2)+2tc(1) , && t\in [1/2,1] ,
 \end{align*}
 we obtain
 \begin{align*}
 &~ 2|c(1/2)-c(0)|^2+ 2|c(1)-c(1/2)|^2 = e_{|.|}(\sig_0) + e_{|.|}(\sig_1) \\
 \leq &~ e_{|.|}(c[0,1/2]) + e_{|.|}(c[1/2,1]) \leq c_F^2 e_F(c) \leq c_F^2 e_F(\sig) = c_F^2|c(1)-c(0)|^2
 \end{align*}
 This inequality can be brought via the isometry group of $(\R^n,|.|)$ into a new position, such that w.l.o.g. $c(0)=0,c(1)=\lam e_1$ with $\lam=|c(1)-c(0)|$ and $x := c(1/2)$ in the $(x_1,x_2)$-plane. We then find
 \begin{align*}
 2|x|^2 - 2\lam x_1 + \lam^2 = |x|^2+ |\lam e_1-x|^2 \leq \frac{c_F^2 \lam^2}{2} . 
 \end{align*}
 Observe that
 \[ |x-\textstyle\frac{\lam}{2}e_1|^2 = |x|^2- \lam x_1 +\lam^2/2-\lam^2/2+ \lam^2/4 = \textstyle\frac{1}{2}(|x|^2+ |\lam e_1-x|^2) - \lam^2/4, \]
 hence
 \begin{align*}
 |x-\textstyle\frac{\lam}{2}e_1|^2 \leq \frac{c_F^2 \lam^2}{4} - \lam^2/4 = \lam^2(c_F^2-1)/4 .
 \end{align*}
 The lemma follows.
\end{proof}

\subsection{Rational directions in higher dimensions}\label{section hyperbolic cylinder}

Let $(M,g)$ be a compact manifold of strictly negative curvature and of arbitrary dimension. By $M$ being compact, there is a negative upper bound for the sectional curvatures of $g$ and the Morse Lemma holds. In this subsection we discuss a theorem showing that in the periodic directions, there are always unstable rays (see Definition \ref{def unstable}). In dimension two this was shown in Theorem \ref{morse periodic} \eqref{morse periodic item 3}.

Let $\tau\in\Gam-\{\id\}$. By Preissmann's theorem, the group $\la\tau\ra \subset\Gam$ generated by $\tau$ is isomorphic to $\Z$. We consider the ``cylinder''
\[ C_\tau := H / \la \tau\ra \]
with the covering map denoted by
\[ p_\tau: H\to C_\tau . \]
$C_\tau$ is homeomorphic to $(\R/\Z)\times \R^{\dim M-1}$. Let us denote by 
\[ \g_\tau : \R\to H \]
the unique $\tau$-periodic $g$-geodesic. The set $\M(\g_\tau)$ is invariant under $D\tau$. We shall consider the $\phi_F^t$-invariant measures in the quotient 
\[ \M_\tau := \M(\g_\tau)/\la D\tau \ra = Dp_\tau(\M(\g_\tau)) \subset S_F C_\tau . \]
Recall the definition of the set $\A_-(\xi_\tau)$ of unstable rays in $\RR_-(\xi_\tau)$, where $\xi_\tau=\g_\tau(-\infty)$. We can prove the following fact, which recovers the fact from Mather theory, that the Mather set is always contained in the Aubry set (Remark \ref{remark mane inclusions}).

\begin{thm}\label{thm measures unstable}
 Let $\mu$ be a $\phi_F^t$-invariant measure supported in the quotient $\M_\tau$ and let $\widetilde{\supp \mu} = Dp_\tau^{-1}(\supp\mu)$ be the lifted support of $\mu$. Then 
 \[ \widetilde{\supp \mu} \subset \A_-(\xi_\tau) . \]
 In particular, $\A_-(\xi_\tau) \neq \emptyset$.
\end{thm}

We give the proof of Theorem \ref{thm measures unstable}. It recovers the argument in the proof of Theorem \ref{morse periodic} \eqref{morse periodic item 3} above, which will appear again in Subsection \ref{section rays-paper}. By the techniques in the proof of Theorem \ref{thm measures unstable} (see also Subsection \ref{section rays-paper}), one might be able to answer the following.

\begin{prob}\label{prob recurrent rays are unstable}
 Let $\Rec(\phi_F^t|_{Dp(\M)})\subset S_FM$ be the set of (forward and backward) recurrent minimal geodesics and let $\A \subset S_FH$ the set of (forward and backward) unstable minimal geodesics. Is it true that
\[ Dp^{-1}(\Rec(\phi_F^t|_{Dp(\M)})) \subset \A ? \]
\end{prob}

\begin{proof}[Proof of Theorem \ref{thm measures unstable}]
 Note that
 \[ H_1(C_\tau,\R) \cong \R, \qquad H^1(C_\tau,\R) \cong \R^*. \]
 As in Section \ref{section mather theory}, we can define minimizing measures in $S_FC_\tau$, which will be supported in $\M_\tau$. This singles out a unique cohomology class $[\eta]\in H^1(C_\tau,\R)$ (by $\dim H^1(C_\tau,\R) = 1$) with $\rho(\mu)([\eta])=1$ for all measures supported in $\M_\tau$. Moreover, if $\nu$ is any $\phi_F^t$-invariant measure in $S_FC_\tau$, then $\rho(\nu)([\eta])\leq 1$. (Note that $C_\tau$ is non-compact, but the Morse Lemma can be used to overcome difficulties.) We use the closed 1-form $\eta$ for the definition of the ``distance'' $d_{F-\eta}$ in $C_\tau$. Using $d_{F-\eta}$, one finds that the orbits $c_v$ with $v\in \widetilde{\supp \mu}$ are $[\eta]$-semistatic, see Lemma \ref{lemma critical value} \eqref{lemma critical value iii}.
 
 Assume now that $v\in \widetilde{\supp \mu}$ and that $w\in \RR_-(\xi_\tau)$ with $\pi v=\pi w$. We have to show $v=w$ and assume the contrary. We then find $\e,\delta>0$, such that
 \begin{align}\label{delta-kink}
  d_F(c_w(-\delta), c_v(\delta)) \leq 2\delta-\e .
 \end{align}
 The first step is to find a suitable recurrent geodesic in the $\al$-limit set of $c_w$. For this we work in the quotient $C_\tau$. Let $\nu$ be a $\phi_F^t$-invariant probability measure supported in the limit set $\al(\phi_F^t,w)$, which lies in $\M_\tau$ (by the Morse Lemma, the geodesic $p\circ c_w$ stays in a compact part of $C_\tau$). We consider the product probability space $(\M_\tau\times\M_\tau,\mu\otimes\nu)$ with the flow $\phi_F^t\times\phi_F^t$, which preserves the product measure $\mu\otimes\nu$. By the \Poincare recurrence theorem, $\mu\otimes\nu$-almost every $(v',w')$ is recurrent. Since we are considering the product measure, we find a sequence $v_k\to v$ in $\supp\mu$ and for each $k$ a $\nu$-full measure set $A_k\subset\supp\nu$, such that the points $(v_k,w')$ with $w'\in A_k$ are backwards $\phi_F^t\times\phi_F^t$-recurrent. We choose any point $w'\in \bigcap_{k\in\N}A_k$ (which has full $\nu$-measure and hence is non-empty). Fixing $k$, this means that we find a sequence $T_{n,k}\to -\infty$ for $n\to\infty$ so that in $S_FC_\tau$
 \begin{align}\label{v,w recurrent}
  \phi_F^{T_{n,k}}v_k \to v_k, \qquad \phi_F^{T_{n,k}}w' \to w'  , \qquad n\to\infty .
 \end{align}
 Since $w'\in\supp\nu \subset \al(\phi_F^t,w)$, we find moreover by fixing $n,k$ a sequence of times $S_{m,n,k} \to -\infty$ for $m\to\infty$ with
 \begin{align}\label{c,c_w close}
  & \dist_g( c_w[S_{m,n,k}+T_{n,k},S_{m,n,k}] , c_{w'}[T_{n,k},0] ) \to 0 , \qquad m \to \infty ,
 \end{align}
 $\dist_g$ denoting the $g$-Hausdorff distance in $C_\tau$.

 We lift the situation to $H$. Let $h:H\to\R$ be a primitive of the lift $\tilde\eta$ of $\eta$ to $H$. Note that, since $\eta$ is $\tau$-invariant, we have $d(h\circ\tau-h) = 0$ and hence $h\circ\tau-h=\kappa$ for some non-zero constant $\kappa\in\R$ ($\kappa\neq 0$, since otherwise $\eta$ would be exact in $C_\tau$, against $\rho(\mu)([\eta])=1$). It follows that
 \begin{align}\label{eqn m kappa}
  h\circ\tau^m-h = \textstyle\sum_{i=1}^m h\circ\tau^i-h \circ \tau^{i-1} = m \cdot \kappa.
 \end{align}
 Next, we lift $c_{v_k}$ and $c_{w'}$ to $H$ (denoting the lifts by the same letters) in such a way that $c_{v_k}(0) \to c_v(0)$ as $k\to\infty$. By \eqref{v,w recurrent}, we find sequences $m_n^{v_k},m_n^{w'}\in \Z$ with
 \begin{align}\label{v,w recurrent lift}
  d_g(\tau^{m_n^{v_k}} c_{v_k}(0) ,  c_{v_k}(T_{n,k})) \stackrel{n\to\infty}{\to} 0, \quad d_g( \tau^{m_n^{w'}} c_{w'}(0) , c_{w'}(T_{n,k})) \stackrel{n\to\infty}{\to} 0 .
 \end{align}
 We claim that $m_n :=m_n^{w'} = m_n^{v_k}$ for sufficiently large $n$ (here we use that $c_{v_k},c_{w'}$ are {\em simultaneously} recurrent). For this observe that by \eqref{eqn m kappa} and $c_{v_k}$ being $[\eta]$-semistatic
 \begin{align*}
  &~ \left| -T_{n,k} + \kappa m_n^{v_k} - d_{F-\eta}(c_{v_k}(T_{n,k}),c_{v_k}(0)) \right| \\
  = &~ \left| -T_{n,k} + \big[ h(\tau^{m_n^{v_k}} c_{v_k}(0)) - h( c_{v_k}(0)) \big] - l_{F-\eta}( c_{v_k}[T_{n,k},0]) \right| \\
  = &~ \bigg| -T_{n,k} + \big[ h(\tau^{m_n^{v_k}} c_{v_k}(0)) - h( c_{v_k}(0)) \big] + T_{n,k} \\
  & \qquad \qquad \qquad \qquad \qquad \qquad \qquad \qquad  + \big[ h( c_{v_k}(0)) - h( c_{v_k}(T_{n,k})) \big] \bigg| \\
  = &~ \left| h(\tau^{m_n^{v_k}} c_{v_k}(0)) - h( c_{v_k}(T_{n,k})) \right| .
 \end{align*}
 By \eqref{v,w recurrent lift}, the last line tends to zero as $n\to\infty$. By \eqref{v,w recurrent} and the continuity of $d_{F-\eta}$, also the $d_{F-\eta}$-term in the first line tends to zero for $n\to\infty$. The analogous calculation for $w'$ (which also minimizes $d_{F-\eta}$ by lying in the support of $\nu$ in $\M_\tau$) shows
 \[ |\kappa | \cdot \left| m_n^{v_k} - m_n^{w'}\right| = \left| \kappa m_n^{v_k} - \kappa m_n^{w'}  \right|\to \left| T_{n,k} - T_{n,k} \right| = 0, \quad n\to\infty. \]
 Since $m_n^{v_k} , m_n^{w'}\in\Z$ and $\kappa\neq 0$, the claim follows.

\begin{figure}\centering%[htb!]
 \includegraphics[scale=1]{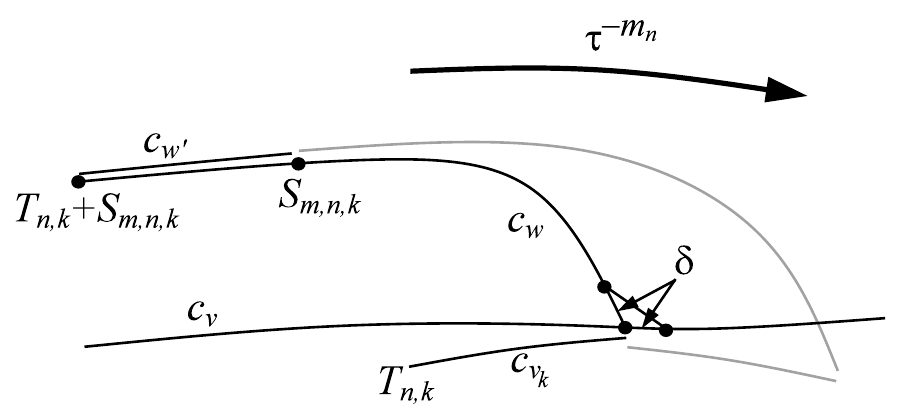}
 \caption{The argument in the proof of Theorem \ref{thm measures unstable}. Images of curves under $\tau^{-m_n}$ are depicted in gray. \label{fig morse-arg}}
\end{figure}

 By \eqref{delta-kink}, we obtain for sufficiently large $k,n$, that
 \begin{align}\label{eqn kink Phi}
  &~ d_F(c_w(-\delta), \tau^{-m_n} c_{v_k}(T_{n,k}+\delta)) \leq 2\delta-\e/2 .
 \end{align}
 We now fix the lift $c_{w'}$, so that using \eqref{c,c_w close} and \eqref{v,w recurrent lift} we have
 \begin{align}\label{eqn neu}
  d_F(\tau^{-m_n}c_w(S_{m,n,k}+T_{n,k}) , c_w(S_{m,n,k}) ) \leq \e/8 .
 \end{align}
 (Note that the choice of the lift of $c_{w'}$ depends on $m,n,k$, which we suppress in the notation.) Using the minimality of $c_w$ in $H$, the triangle inequality, $\pi v = \pi w$, $v_k\to v$, \eqref{eqn kink Phi} and \eqref{eqn neu} we find for large $m,n,k$ (cf.\ Figure \ref{fig morse-arg})
 \begin{align*}
  &~ -(S_{m,n,k}+T_{n,k}) \\
  = &~ d_F(\tau^{-m_n}c_w(S_{m,n,k}+T_{n,k}),\tau^{-m_n}c_w(0)) \\  
  \leq &~ d_F(\tau^{-m_n}c_w(S_{m,n,k}+T_{n,k}),c_w(S_{m,n,k})) + d_F(c_w(S_{m,n,k}),c_w(-\delta)) \\
  + &~ d_F(c_w(-\delta),\tau^{-m_n} c_{v_k}(T_{n,k}+\delta)) + d_F(\tau^{-m_n} c_{v_k}(T_{n,k}+\delta),\tau^{-m_n}c_{v_k}(0)) \\
  + &~ d_F(\tau^{-m_n}c_{v_k}(0),\tau^{-m_n}c_v(0)) \\
  \leq &~ \e/8 -(\delta + S_{m,n,k}) + (2\delta-\e/2) - (T_{n,k}+\delta) + \e/8 \\
  = &~ -(S_{m,n,k}+T_{n,k}) -\e/4 .
 \end{align*}
 This is a contradiction.
\end{proof}

\subsection{Multiplicity of shortest closed geodesics}\label{section generic}

In Remark \ref{remark shorst closed geod are minimal} we observed that each free homotopy class of a closed surface $M$ entails minimal geodesics. Such minimal geodesics may foliate strips in $M$, i.e.\ there may be infinitely many shortest closed geodesics in a single homotopy class. But one does not expend this to be the case for ``typical'' Finsler metrics $F$.

We let here $M$ be a closed manifold of arbitrary dimension. Let us first see what we mean by ``typical''. In \cite{mane}, R.\ \Mane perturbed a given Lagrangian $L:TM\to\R$ by a potential function $U:M\to\R$, namely considering the new Lagrangian
\[ L' = L +U \circ \pi, \]
where $\pi:TM\to M$ is the canonical projection. Properties of $L'$ for generic choices of the potential $U$ have been studied by \Mane in \cite{mane} and e.g.\ by P.\ Bernard and G.\ Contreras in \cite{bercont}. In the geometric setting we use a different kind of perturbation, which is related to the before mentioned one via Maupertuis' principle. Namely, if $L=\frac{1}{2}|.|_g^2$ is the kinetic energy associated to the Riemannian metric $g$, then $L' =L + U \circ \pi$ defines a mechanical Lagrangian and the solutions of the Euler-Lagrange equations associated to $L'$ of energy $\kappa$ are reparametrized geodesics in the so-called Jacobi metric $g' = (\kappa+U)\cdot g$. Hence, perturbing $L$ by adding a potential corresponds to perturbing the metric $g$ by a conformal function. We write
\[ E := \{\lam \in C^\infty(M,\R) : \lam(x)>0 ~ \forall x \in M\} \]
and endow $E$ with the $C^\infty$-topology. By a {\em residual subset} of $E$ we mean a countable intersection of open and dense subsets of $E$. As $E$ is an open subset of a Fr\'echet space, residual sets are dense in $E$.

\begin{defn}\label{def generic}
 A property of Finsler metrics is said to be {\em conformally generic} if for any given Finsler metric $F_0$ on $M$ there exists a residual subset $\O\subset E$, so that the property holds for all Finsler metrics $F$ of the form
 \[ F(v) = f(\pi v) \cdot F_0(v) \]
 with $f\in \O$. 
\end{defn}

By abuse of language, we will also call the Finsler metrics $fF_0, f\in\O$ conformally generic. Observe that if $F_0$ is e.g.\ Riemannian, then so are the perturbed Finsler metrics $F=f\cdot F_0$.

In \cite{generic-paper}, the author proved the following theorem.

\begin{thm}\label{thm generic loops}
 The property to admit in each free homotopy class, which is not of order 2, precisely one shortest closed geodesic is conformally generic for Finsler metrics on closed manifolds.
\end{thm}

If $M$ is a closed, orientable surface, then any homotopy class is of infinite order. Applying Theorem \ref{thm generic loops}, this means in the case of genus $\mathfrak{g}=1$, that $\M^{per}(\g_\tau)\subset S_FH$ consists generically precisely of the $\Gam\cong \Z^2$-translates of a single $\tau$-periodic geodesic, for each $\tau\in \Gam$. In particular, there will be neighboring periodic minimals and heteroclinic connections between them. In the higher genus case $\mathfrak{g}\geq 2$, then generically $\M^{per}(\g_\tau)\subset S_FH$ consists of a single $F$-geodesic. All rays in $\RR_-(\xi_\tau)$ are asymptotic to this minimal geodesic.

Theorem \ref{thm generic loops} rests on an abstract result, which was in a slightly different form proved by \Mane in \cite{mane}. We state Theorem 2.1 from \cite{generic-paper}, which can be applied to prove Theorem \ref{thm generic loops}.

\begin{thm}\label{thm mane generic}
 Let $V$ be a first-countable, Hausdorff topological $\R$-vector space and $E\subset V$ an open subset. Moreover, let $X$ be a metri\-zable topological space. We consider a function
 \[ \varphi : X \times E \to \R \]
 with the following properties:
\begin{itemize}
 \item (regularity) $\vf$ is continuous and $\vf(x,.)$ is differentiable for fixed $x$. We write $D_V\vf(x,f)[v]$ for the derivative and assume moreover, that this is continuous in all components $x,f,v$.
 
 \item (compactness) For fixed $f\in E$, the function $\varphi(.,f)$ attains its infimum in a non-empty set $M(f)\subset X$ and if $\{f_n\}\subset E$ is a convergent sequence with limit in $E$, then the union $\bigcup M(f_n)$ is contained in a compact subset of $X$.
 
 \item (non-degeneracy) For each $f\in E$, the derivative of $\vf$ separates the set of minimizers $M(f)$ in the sense that
 \[ \forall x\neq y \text{ in } X ~ \exists v\in V: \quad D_V\vf(x,f)[v] \neq D_V\vf(y,f)[v] . \]
\end{itemize}
 Then there exists a residual subset $\O \subset E$, so that
\[ f\in \O \quad\implies\quad \card M(f)=1. \]
\end{thm}

In order to prove Theorem \ref{thm generic loops} one makes the following definitions:
\begin{itemize}
 \item $V = C^\infty(M,\R)$ with the $C^\infty$-topology and $E$ the subset of positive functions.
 
 \item $X$ is the space of smooth curves $c:\R/\Z\to M$ in a fixed free homotopy class modulo parametrization, endowed with the $C^1$-topology.
 
 \item $\vf$ is given by the length, namely $\vf(c,f) = l_{f\cdot F_0}(c;[0,1])$, where $F_0$ is a fixed Finsler metric.
\end{itemize}

The approach via Theorem \ref{thm mane generic} could have more applications. In a similar procedure, one should be able to show that mass minimizing currents in homology classes $h\in H_k(M,\Z)$ become unique. Moreover, fixing a boundary integral current, solutions to the corresponding Plateau problem in $M$ should become unique. 

Let us also mention that it could be possible using Theorem \ref{thm mane generic} to show that the minimizing measures in Theorem \ref{thm measures unstable} are unique for conformally generic $F$. Here the term ``minimizing'' is understood in the sense of Section \ref{section mather theory}. See also the arguments in \cite{mane}. This would generalize the fact that generically, the set $\M(\g_\tau)$ consists of a single minimal geodesic, if $\dim M=2$. Possibly, the dynamics in the support of such a measure are special in some sense, as the support is uniquely ergodic. This could be a step towards solving Problem \ref{problem thin neck}. Possibly this also implies that for generic $F$, the directed weak KAM solution in $\delta_F^{-1}(\xi)$ is unique for $\xi$ a fixed point of the group $\Gam$.

\section{Irrational directions}\label{section irrational}

Let $M$ be a closed, orientable surface. The Morse Lemma holds for a constant curvature metric $g$ and we defined the sets $\Pi_M\subset \se$ of rational directions and $\se-\Pi_M$ of irrational directions in Subsection \ref{section model geom}. In Subsection \ref{section rational dir 2-dim}, we described the structure of $\RR_-(\xi)$ for rational $\xi$. Here, we treat irrational $\xi$. The dynamics depend on the genus $\mathfrak{g}$ being equal to 1 or $\geq 2$.

\subsection{The genus one case}\label{section irrational torus}

Let $\mathfrak{g}=1$, i.e.\ $M=\T$. The following theorem was proved by Bangert \cite{bangert} and independently by Bialy and Polterovich \cite{bialy-polt}, assuming that $F$ is reversible. For general Finsler metrics $F$, see \cite{paper1}. Given $\xi\in \se$, let $\g_\xi:\R\to H$ be a $g$-geodesic of the form $\g_\xi(t)=x+t\xi$, i.e.\ with asymptotic directions $\g(\pm\infty)=\xi$. Recall that the sets $\RR_-(\xi)$ and $\M(\g_\xi)\subset\RR_-(\xi)$ are $\Gam$-invariant. We write
\[ \M^{rec}(\g_\xi) := \{ v\in \M(\g_\xi) ~|~ \dot{\widehat{p\circ c_v}} : \R\to S_F\T \text{ is bi-recurrent} \} . \]

\begin{thm}\label{thm bangert irrat}
 Let $M=\T$ and $\xi\in \se-\Pi_M$ be irrational. Then the following hold:
 \begin{enumerate}[(i)]
  \item The set $\M^{rec}(\g_\xi) \subset \M(\g_\xi)$ is closed and minimal for $\phi_F^t$, i.e.\ it contains no non-trivial, closed, $\phi_F^t$-invariant subsets.
  
  \item The projection $\pi(\M^{rec}(\g_\xi))\subset \R^2$ is either nowhere dense in $\R^2$ or equal to $\R^2$.
  
  \item Any ray in $\RR_-(\xi)$ is asymptotic to some minimal geodesic in $\M^{rec}(\g_\xi)$. The same is true for $\RR_+(-\xi)$.
  
  \item\label{thm bangert irrat iv} No two rays in $\RR_-(\xi)$ intersect.
 \end{enumerate}
\end{thm}

Item \eqref{thm bangert irrat iv} means that all rays in $\RR_-(\xi)$ are unstable. Observe that due to $\Gam$-invariance of the set $\RR_-(\xi)$, the projected rays in $Dp(\RR_-(\xi))$ do not intersect in $\T$.

\subsection{The higher genus case}\label{section rays-paper}

In this subsection, we let $(M=H/\Gam,g)$ be a connected, orientable, complete Riemannian surface with constant curvature $-1$. The universal cover $(H,g)$ is the \Poincare disc. Recall that two rays $\g,\g':\R_-\to H$ with $\g(-\infty)=\g'(-\infty)$ are asymptotic with respect to $g$ in the strong sense that $d_g(\g(-t),\g'(\R_-))\to 0$ as $t\to\infty$. We shall need the following property of points $\xi$ in the Gromov boundary $\Gro(H,g)\cong \se$.

\begin{defn}
 A point $\xi\in \se$ is called {\em $\al$-compact} (with respect to $\Gam$ acting on $H\cup \se$), if for one and hence for all $g$-rays $\g:\R_-\to H$ the $\al$-limit set of the projected ray $\dot{\widehat{p\circ\g}}:\R_-\to S_gM$ contains a non-empty, compact subset left invariant by the $g$-geodesic flow.
\end{defn}

Note that if $M$ is compact, then every point $\xi\in \se$ is $\al$-compact.

Let $F$ be a Finsler metric on $M$ uniformly equivalent to $g$. The Morse Lemma holds. The main result in this subsection is the following, proved by the author in \cite{min_rays}.

\begin{thm}\label{thm rays-paper}
 If $\xi\in \se$ is $\al$-compact and if $\xi$ is not fixed by a hyperbolic element of $\Gam$, then for all $v,w\in \RR_-(\xi)$ we have
 \[ \liminf_{t\to\infty} d_g(c_v(\R_-),c_w(-t)) = 0 . \]
 Hence, $w_-(\xi)=0$ (see Definition \ref{def width}).
\end{thm}

In essence, Theorem \ref{thm rays-paper} recovers a fact known from the case where $F$ is Riemannian and of negative curvature: if two geodesics stay at finite distance, then by the flat strip theorem, they have to be asymptotic in the strong sense. For general Finsler metrics $F$ with no curvature assumptions, this is not at all clear. Note that Theorem \ref{thm rays-paper} is in general false for $\xi\in \se$ being fixed under a hyperbolic element of $\Gam$: imagine a surface of genus two with two thin waists forming shortest closed geodesics, bounding a thickened annulus; here the closed geodesics are minimal (Subsection \ref{section rational dir 2-dim}) and at uniformly positive distance. Theorem \ref{thm rays-paper} shows, however, that such such a thickened annulus is the only situation, in which two rays can stay in distance bounded from above and below by positive constants. Observe also that Theorem \ref{thm rays-paper} is false for $M=\T$, i.e.\ it depends on the hyperbolicity of the background metric $g$.

Combining Theorem \ref{thm rays-paper} with Lemma \ref{crossing minimals}, we obtain Theorem \ref{thm bangert irrat} \eqref{thm bangert irrat iv} for higher genus surfaces.

\begin{thm}\label{thm irrat no intersections}
 Let $M$ be a closed, orientable surface of genus at least two. Then, in the universal cover, no two (backward) rays with the same irrational asymptotic direction can intersect.
\end{thm}

We give a brief overview on the proof of Theorem \ref{thm rays-paper}, referring to \cite{min_rays} for the details. The underlying idea is similar to the proof of Theorem \ref{thm measures unstable}.

\begin{proof}[Idea of the proof of Theorem \ref{thm rays-paper}]
 For $\xi\in \Gro(H,g)$ being $\al$-com\-pact, choose a subset $G\subset S_gM$ of the $\al$-limit set of a $g$-ray $\g$ with $\g(-\infty)=\xi$, which is compact, $\phi_g^t$-invariant and minimal (contains no non-trivial, closed, invariant subsets). Assume that the width $w_-(\xi)>0$ and let $\g'$ be a $g$-geodesic lifted from $G$. As $G$ is minimal, $\g'$ is recurrent under a sequence $\{\tau_n\}\subset\Gam$ and by Proposition \ref{width semi-cont} $w_0(\g')>0$. The first step is to find two minimal geodesics $c_0,c_1$ in $\M(\g')$, which are forward recurrent under the given sequence $\{\tau_n\}$ and which stay at a uniformly positive distance (using $w_0(\g')>0$). The approximation of $\g'$ by $\g$ under the group $\Gam$ can then be used to obtain a forward ray $c$ initiating transversely from $c_0(\R)$, say, with $c(\infty)=c_0(\infty)$. Then one proves that in the $\om$-limit set of $c$, there is a minimal geodesic in $\M(\g')$, which is recurrent under the same sequence $\{\tau_n\}$ as $c_0$, and which has the same recurrence times $T_n \to \infty$, as $c_0$. For this latter fact, one uses a recurrent, directed weak KAM solution given by Corollary \ref{cor exist recurrent horofctns} (in the proof of Theorem \ref{thm measures unstable}, we used simultaneous recurrence and the fact that the occurring orbits were semistatic in $C_\tau$). Now one proceeds as in the proof of Theorem \ref{thm measures unstable}, replacing $c_v$ by $c_0$ and $c_w$ by $c$, to obtain a contradiction.
\end{proof}

We described in Theorems \ref{morse periodic} and \ref{thm irrat no intersections} the structure of the set of rays $\RR_-\subset S_FH$ and minimal geodesics $\M\subset S_FH$. Many of the results resemble the dynamics in negative curvature. In particular, one can ask whether there are dense geodesics in the set $Dp(\M)\subset S_FM$, of whether the periodic geodesics are dense in $Dp(\M)$. For this, see Section 13 in \cite{morse hedlund} and the theorem in \S 3.5 in \cite{klingenberg}. Also observe that Zaustinsky shows transitivity under an assumption on the uniqueness of rays, cf.\ Theorem 6.4 in \cite{zaustinsky}.

Another problem is the following.

\begin{prob}\label{problem rays asymptotic}
 Can $\liminf$ be replaced by $\lim$ in Theorem \ref{thm rays-paper}?
\end{prob}

Relaxing the question, one can ask whether the distance function $a(t) := d_g(c_v(\R_-),c_w(-t))$ has time-average equal to zero. This would imply that for absolutely continuous functions $f:\RR_-(\xi)\to \R$, the time average is constant. This plays a role in the ergodicity of invariant measures in $Dp(\M)\subset S_FM$. For the latter condition observe that by the techniques in \cite{min_rays}, for any $\e>0$ there exists a length $L(\e)<\infty$, so that the function $a(t)$ cannot be $\geq \e$ on intervals of length $\geq L(\e)$. Considering the sets $\{ n \in \N : a(-n)\geq \e \}$ leads to notions like piecewise syndetic sets and the density of subsets of $\N$ in number theory. One should be able to link these notions to recurrence in the dynamics of $\phi_g^t$, see \cite{furstenberg}.

Problem \ref{problem rays asymptotic} was originally motivated by the construction of measures of maximal entropy in $Dp(\M)$ proposed by Knieper, see \cite{knieper max entropy annals}. Solving the above problems can be used in proving ergodicity and uniqueness of the arising measure. Such measures in the set of minimal geodesics generalize the ergodic Lebesgue measure in the unit tangent bundle in the case of negative curvature. Note also the calculation of the topological entropy in the set $Dp(\M)$ in Subsection \ref{section GKOS}.

\section{Applications to Mather theory, weak KAM solutions and average actions}\label{section applications}

In this section we apply the structure results discussed in Sections \ref{section rational} and \ref{section irrational} to the notions introduced in Sections \ref{section weak KAM} and \ref{section mather theory}, assuming that $M$ is a closed, orientable surface of genus $\mathfrak{g}\geq 1$. (Observe that the structure of the set of minimal geodesics applies also in the non-orientable case, after moving to the orientable double cover -- minimal geodesics are defined in the universal cover --, so one can apply our previous results to the study of the non-orientable case as in \cite{massart-balacheff}.)

\subsection{The genus one case}\label{section mather torus}

Let $\mathfrak{g}=1$, i.e.\ $M=\T$. Recall that $\Pi_\T\subset \se$ was in this case the set of points with rational or infinite slope. We have $c(\infty)=c(-\infty)$ if $c$ is a minimal geodesic, as minimal geodesics shadow straight lines. Given $\xi\in \se$, we shall write $\g_\xi$ for a straight Euclidean line of the form $\g_\xi(t)=x+t\xi$, i.e.\ with $\g_\xi(\pm\infty)=\xi$.

We wish to describe the Mather, Aubry and \Mane sets defined in Section \ref{section mather theory}. The general reference for the proofs is \cite{paper1}. The first observation was made before, namely that for any $\eta$-weak KAM solution $[u]$, the Mather set $\M_{[\eta]}$ consists of $[u]$-calibrated curves. Hence, if $c:\R\to\R^2$ is a geodesic projecting into the Mather set $\M_{[\eta]}\subset S_F\T$ and if $c':\R\to\R^2$ projects to any $[\eta]$-semistatic curve in $\T$, then the two minimal geodesics $c,c'$ cannot intersect in $\R^2$. This leads to the following observation.

\begin{prop}
 All minimal geodesics $c:\R\to\R^2$ projecting to the \Mane set $\NN_{[\eta]}$ have a common asymptotic direction $c(-\infty)$.
\end{prop}

Using the above observations one can relate the asymptotic direction to Mather's $\al$-function.

\begin{thm}\label{thm al C^1 torus}
 Mather's $\al$-function $\al_F:H^1(\T,\R)-\{0\} \to \R$ is $C^1$, hence its convex conjugate $\beta_F:H_1(\T,\R)-\{0\} \to \R$ is strictly convex. The asymptotic direction of a minimal geodesic $c:\R\to\R^2$ projecting into the \Mane set $\NN_{[\eta]}$ for $[\eta]\in \partial B^*(F)=\{\al_F=1/2\}$ is given by
 \[ c(-\infty) = \frac{\nabla \al([\eta])}{|\nabla \al([\eta])|} \in \se. \]
\end{thm}

Equivalently, the asymptotic direction $c(-\infty)$ of (lifts of) $[\eta]$-semi\-static rays $c:\R_-\to \R^2$ is given by the unique vector $\xi\in \se\subset\R^2\cong H_1(\T,\R)$ with the property
\[ \la \xi, [\eta] \ra = \max\{ \la \xi, [\eta'] \ra : [\eta'] \in B^*(F) \}. \]
Given a point $\xi\in \se$, we consider the {\em face of direction $\xi$}
\[ \mathcal{F}_\xi := \left\{ [\eta]\in \partial B^*(F) : \la \xi, [\eta] \ra = \max_{[\eta'] \in B^*(F)} \la \xi, [\eta'] \ra \right\} . \]
It follows from Theorem \ref{thm al C^1 torus}, that for $[\eta]\in \mathcal{F}_\xi$, the asymptotic direction of (lifts of) $[\eta]$-semistatic rays $c$ is $c(-\infty)=\xi$.

Recall that $\eta$-weak KAM solutions $[u]\in C^0(\T)/_\sim$ can be lifted to weak KAM solutions $[\tilde u] \in \W(\R^2,F)$, see Proposition \ref{prop lifting weak KAM}.

\begin{cor}\label{cor u directed}
 If $[u]$ is an $\eta$-weak KAM solution for some $[\eta]\in\partial B^*(F)$, then the lift $[\tilde u]$ is directed and
 \[ \delta_F([\tilde u]) = \frac{\nabla \al([\eta])}{|\nabla \al([\eta])|} . \]
\end{cor}

The following result was originally proved by Bangert in \cite{bangert1}.

\begin{thm}\label{thm mather set irrat}
 If $\xi\in \se-\Pi_\T$ is irrational, then $\card \mathcal{F}_\xi = 1$ and given a closed 1-form $\eta$ on $\T$ with $\mathcal{F}_\xi=\{[\eta]\}$, the $\eta$-weak KAM solution $[u]$ is unique. Moreover, 
 \[ \M_{[\eta]}=Dp(\M^{rec}(\g_\xi)) \quad \subset \quad \A_{[\eta]} = \NN_{[\eta]} = Dp(\M(\g_\xi)) . \]
\end{thm}

We can also describe the Mather set in rational directions.

\begin{thm}\label{mather set rational torus}
 If $\xi\in \Pi_\T$ is rational, then for $[\eta]\in \mathcal{F}_\xi$ we have
 \[ \M_{[\eta]}=Dp(\M^{per}(\g_\xi)) . \]
\end{thm}

Although we have not proved this, we expect that for rational $\xi$, the Aubry set $\A_{[\eta]}$ equals the \Mane set $\NN_{[\eta]}$ for $[\eta]$ being an endpoint of the segment $\mathcal{F}_\xi \subset \partial B^*(F)$ and equals the Mather set $\M_{[\eta]}$ for $[\eta]$ in the interior of the segment $\mathcal{F}_\xi \subset \partial B^*(F)$. Moreover, if $\mathcal{F}_\xi$ has endpoints $[\eta_0],[\eta_1]$ ordered in the counterclockwise orientation of the closed curve $\partial B^*(F)$, then we expect
\[ \A_{[\eta_i]} = \NN_{[\eta_i]} = Dp\left(\M^{per}(\g_\tau) \cup (\RR_-^i(\xi_\tau)\cap\M)\right) , \qquad i = 0,1 \]
in the notation explained after Theorem \ref{morse periodic}. These sets are the curves calibrated on $\R$ with respect to the weak KAM solutions $[u_i]$ given by Proposition \ref{bounding horofunctions}.

Recall that in Subsection \ref{section av action} we also defined a Mather set $\M^h$ for $h\in \partial B(F)=\{\beta_F=1/2\}$. But using Remark 4.26 (ii) in \cite{sorrentino} and the fact that $\al_F$ is $C^1$, one sees that
\[ \M^{\nabla\al_F([\eta])} = \M_{[\eta]}, \]
while the map $\nabla\al_F : H^1(\T,\R)\to H_1(\T,\R)$ is surjective. Hence, our description of the Mather sets $\M_{[\eta]}$ suffices in order to describe also the Mather sets $\M^h$.

Recall that an $[\eta]$-semistatic curve $c:\R\to\T$ lifts to a minimal geodesic $\tilde c:\R\to\R^2$, for any $[\eta]\in\partial B^*(F)$. In general, however, the set of minimal geodesics will be much larger than the set of lifts of lifts of semistatic curves. In the case of the 2-torus, the converse holds.

\begin{prop}
 \begin{align*}
  Dp(\M) & = \bigcup\{ \NN_{[\eta]} : [\eta]\in\partial B^*(F) \} .
 \end{align*}
 The analogous result holds for (backward and forward) rays.
\end{prop}

We now discuss directed weak KAM solutions (which contain the lifts of $\eta$-weak KAM solutions for any $[\eta]\in \partial B^*(F)$ by Corollary \ref{cor u directed}) and properties of Mather's $\al$- and $\beta$-function. Recall from Subsection \ref{section directed KAM} that $\delta_F:\W_{dir}(\R^2,F)\to \Gro(\R^2,g)\cong \se$ is the map associating to $[u]$ the a\-symp\-totic direction $c(-\infty)$ of the $[u]$-calibrated rays $c$.

\begin{thm}\label{thm horobdry torus}
 Let $\xi \in \se$.
 \begin{enumerate}[(i)]
  \item\label{thm horobdry torus i} If $\xi\in \se-\Pi_\T$, then $\card\delta_F^{-1}(\xi) =1$.
  
  \item\label{thm horobdry torus ii} If $\xi \in \Pi_\T$, then $\card\delta_F^{-1}(\xi) =1$ if and only if $\RR_-(\xi)=\M^{per}(\g_\xi)$, i.e.\ if the torus $\T$ is foliated by shortest closed geodesics in the prime homotopy class in $\R_{>0}\xi \cap \Z^2$.
 \end{enumerate}
\end{thm}

The uniqueness result for irrational $\xi$ follows from Theorem \ref{thm bangert irrat}. The characterization in the rational case follows from Theorem \ref{morse periodic}. It shows in particular, that $\card\delta_F^{-1}(\xi)=1$ for rational $\xi$ implies the existence of a continuous invariant graph $Dp(\M^{per}(\g_\xi))$ in $S_F\T$, which is also a $C^0$-KAM torus in the sense of Definition \ref{def KAM}. By taking limits and recalling a celebrated result due to Hopf \cite{hopf}, one obtains the following result.

\begin{cor}\label{cor horobdry torus}
 The map $\delta_F:\W_{dir}(\R^2,F)\to \se$ is a homeomorphism if and only if $\phi_F^t$ is $C^0$-integrable, i.e.\ $S_F\T$ is foliated by continuous $\phi_F^t$-invariant graphs of the form $Dp(\M(\g_\xi))$. Equivalently, $\delta_F$ is a homeomorphism if and only if $F$ has no conjugate points. If $F$ is Riemannian, then this is the case if and only if the curvature of $F$ vanishes identically.
\end{cor}

Analogously, one can use the horofunction compactification $\overline{i_F(\R^2)}$ of $(\R^2,F)$ in Corollary \ref{cor horobdry torus}.

Theorem \ref{thm horobdry torus} is analogous to the following theorem concerning Ma\-ther's $\al$- and $\beta$-functions. The result was first proved in \cite{mather1} for monotone twist maps and by different methods in \cite{bangert1}.

\begin{thm}\label{thm mather-bangert}
 \begin{enumerate}[(i)]
  \item\label{thm mather-bangert i} $\beta_F$ is always differentiable in $h\in H_1(\T,\R)-\{0\}$ with irrational slope.
  
  \item\label{thm mather-bangert ii} $\beta_F$ is differentiable in $h\in H_1(\T,\R)-\{0\}$ with rational slope if and only if the torus $\T$ is foliated by shortest closed geodesics in the prime homotopy class in $\R_{>0}\xi \cap \Z^2$.
 \end{enumerate}
\end{thm}

Note that the differentiability of $\beta_F$ in $h\in  H_1(\T,\R)-\{0\}$ is equivalent to $\card \F_{h/|h|}=1$. For irrational $\xi\in \se$ this is always true, as was already stated in Theorem \ref{thm mather set irrat}. Theorem \ref{thm mather-bangert} shows that Corollary \ref{cor horobdry torus} can also be stated in the form that $\beta_F$ is $C^1$ or in the form that $\al_F$ is strictly convex. In this connection we refer also to \cite{massart sorrentino}.

\begin{remark}
 We saw in Theorems \ref{thm horobdry torus} and \ref{thm mather-bangert}, that $\W_{dir}(\R^2,F)$ shares many properties with $B^*(F)$. One should study such links more closely, possibly defining a map $\W_{dir}(\R^2,F)\to \partial B^*(F)$. This could also help to understand $B^*(F)$ in the case of a higher genus surface. Note that in that case, we will have a good understanding of the set $\W_{dir}(\R^2,F)$ in Subsection \ref{section applications higher genus}, while $B^*(F)$ remains a bit mysterious.
\end{remark}

We will sketch parts of Mather's proof of Theorem \ref{thm mather-bangert}, see \cite{mather1} and \cite{mather2}, as the idea is useful for recent results obtained in \cite{stablenorm_paper}. Before this, we recall another way to obtain $\beta_F$. Recall the definition of the stable norm
\[ \sig_F=\sqrt{2\beta_F} : \R^2 \cong H_1(\T,\R) \to \R \]
in Subsection \ref{section av action}. In these standard coordinates, the homotopy classes $z\in\Z^2$ of free loops lie naturally in $\R^2$.

\begin{lemma}\label{lemma stable norm is marked length}
 For $z\in\Z^2$ we have
 \[ \sig_F(z) = \inf\left\{ l_F(c) : \text{$c$ a closed curve with homotopy class $z$} \right\} . \]
\end{lemma}

The object on the right hand side is also called the {\em marked length spectrum}. Note that the homogeneity of the right hand side follows from the fact that the shortest closed geodesics in the homotopy class $kz$ are the $k$-th iterates of the minimizers in the class $z$, for $k\geq 0$. Observe also that there is an intuitive way to see the strict convexity of $\sig_F$. Namely if $z,w\in \Z^2-\{0\}$ are linearly independent, let $c_z,c_w$ be shortest closed geodesics in the respective homotopy classes. As $z,w$ are linearly independent, the lifts $\tilde c_z,\tilde c_w:\R\to\R^2$ intersect transversely in a point $x\in\R^2$. Hence,
\begin{align*}
 \sig_F(z+w) & \leq d_F(x-z,x+w) \\
 & < d_F(x-z,x) + d_F(x,x+w) = \sig_F(z)+\sig_F(w) .
\end{align*}
The strict inequality follows from the transversality of the intersection.

\begin{proof}[Proof of Lemma \ref{lemma stable norm is marked length}]
 For $c:[0,T]\to \T$ a shortest closed geodesic in the homotopy class $z$, parametrized by arc length, let $\mu_c$ be the measure given by $\int f ~ d\mu = \frac{1}{T}\int_0^T f(\dot c)~ dt$. Then $\mu_c\in\MM(L_F)$ and $\mu_c$ is minimizing. For a lift $\tilde c$ of $c$ to $\R^2$, identifying $H^1(\T,\R)\cong(\R^2)^*$ we find
 \[ \rho(\mu_c)(\la v , . \ra) = \frac{1}{T}\int_0^T \la v , \dot c \ra~ dt = \frac{1}{T} \la v , \tilde c(T)-\tilde c(0) \ra = \la v , \frac{z}{T} \ra . \]
 Hence, $\rho(\mu_c) = z/T$. On the other hand, as $\mu_c$ is minimizing and supported in $S_F\T$, we have
 \[ \beta_F(\rho(\mu_c)) = \int L_F ~d\mu_c = 1/2 , \]
 showing by homogeneity
 \[ \beta_F(z) = T^2 \cdot \beta_F(z/T) = T^2/2 . \]
 The claim follows from $T=l_F(c)$.
\end{proof}

Note that $\sig_F$ as a convex function has one-sided derivatives. We shall write
\[ D^+\sig_F(h)[v] :=  \inf_{t>0} \frac{\sig_F(h+tv)-\sig_F(h)}{t} = \lim_{t\searrow 0} \frac{\sig_F(h+tv)-\sig_F(h)}{t} . \]

\begin{proof}[Sketch of ``$\M(\g_\xi)\neq \M^{per}(\g_\xi)$ $\implies$ $\sig_F$ not differentiable in $\xi$'']
 We assume for simplicity, that $\xi=(1,0)=e_1$ and that $\M^{per}(\g_\xi)$ consists of the $\Z^2$-translates of a single minimal geodesic $c_0$. Write $c_-,c_+$ for a choice of heteroclinic connections between $c_0$ and $c_1:=c_0+e_2$. Using e.g.\ the bounding weak KAM solutions $[u_0],[u_1]$ from Proposition \ref{bounding horofunctions}, one can show after choosing appropriate parameters, that there exist $T_-,T_+\in \R$ and a function $f(t)\searrow 0$, as $t\to 0$ with
 \[ f(t) \geq \begin{cases} d_F(c_0(-t), c_+(-t)) \\ d_F(c_+(T_++t), c_1(t)) \\ d_F(c_1(-t), c_-(-t)) \\ d_F(c_-(T_-+t), c_0(t)) \end{cases} . \]
 
 Let us write $\theta=\sig_F(e_1)$ for the period of $c_0$ and for $n\in\N$ choose $F$-minimizing segments $\delta_n$ connecting $c_0(-n\theta/2)$ to $c_+(-n\theta/2)$ and $\e_n$ connecting $c_+(T_++n\theta/2)$ to $c_1(n\theta/2)$. The concatenation $c_n := \delta_n*c_+|_{[-n\theta/2,n\theta/2+T_+]}*\e_n$ has endpoints congruent modulo $ne_1+e_2$ and has length
 \begin{align}\label{eqn stable-norm paper}
  \sig_F(ne_1+e_2)\leq l_F(c_n) \leq T_+ + n\theta + 2 f(n\theta) .
 \end{align}
 This shows with $n\theta=\sig_F(ne_1)$ and $f(n\theta)\to 0$, that
 \[ D^+\sig_F(e_1)[e_2] = \lim_{n\to\infty} \sig_F(ne_1+e_2) - \sig_F(ne_1) \leq T_+ . \]
 It is not difficult to show also the reverse inequality. The analogous reasoning using $c_-$ shows
 \[ D^+\sig_F(e_1)[\pm e_2] = T_\pm . \]
 Observe now that $c_-,c_+$ have a transverse intersection in a point $x\in c_-(\R)\cap c_+(\R)$. It follows that
 \begin{align*}
  &~ d_F(c_+(-n\theta/2,c_-(T_-+n\theta/2)) \\
  \leq &~ d_F(c_+(-n\theta/2,x) + d_F(x,c_-(T_-+n\theta/2)) - \e , \\
  &~ d_F(c_-(-n\theta/2,c_+(T_++n\theta/2)) \\
  \leq &~ d_F(c_-(-n\theta/2,x) + d_F(x,c_+(T_++n\theta/2)) - \e ,
 \end{align*}
 for some $\e>0$ and $n$ sufficiently large. This shows
 \begin{align*}
  2 n\theta & = d_F(c_0(-n\theta/2),c_0(n\theta/2)) + d_F(c_1(-n\theta/2),c_1(n\theta/2)) \\
  & \leq 2 f(n\theta/2) + d_F(c_+(-n\theta/2,x) + d_F(x,c_-(T_-+n\theta/2)) - \e \\
  & \quad + 2 f(n\theta/2) + d_F(c_-(-n\theta/2,x) + d_F(x,c_+(T_++n\theta/2)) - \e \\
  & = 2 f(n\theta/2) -2\e +2n\theta + T_+ + T_-
 \end{align*}
 and hence,
 \[ D^+\sig_F(e_1)[e_2] + D^+\sig_F(e_1)[- e_2] = T_++T_-\geq 2\e >0 , \]
 showing that $\sig_F$ cannot be differentiable in $e_1$.
\end{proof}

Assuming that the closed geodesic $c_0$ in the above proof is hyperbolic, the heteroclinics $c_-,c_+$ belong to the (un)stable manifold of $c_0$. In this case, the function $f(t)$ is of the form $C\cdot \exp(-\lam t)$ for real numbers $C,\lam>0$. From inequality \eqref{eqn stable-norm paper}, we obtain an immediate corollary observed in \cite{stablenorm_paper}.

\begin{cor}\label{cor stable norm paper rational}
 If $\xi\in \Pi_\T$ and if the set $Dp(\M^{per}(\g_\xi))\subset S_F\T$ is uniformly hyperbolic for the geodesic flow $\phi_F^t$, then there exist constants $C,\lam, \e>0$, so that for all $v\in \R^2$ with Euclidean norm $|v|\leq \e$
 \begin{align*}
  & \sig_F(\xi+v) - \sig_F(\xi) - D^+\sig_F(\xi)[v] \leq |v| \cdot C \cdot \exp\left( -\lam \cdot \frac{1}{|v|} \right) .
 \end{align*}
\end{cor}

See Figure \ref{fig_stable_rot} for the level set $\{\sig_F=1\}=\partial B(F)$ for the case where $F$ is obtained from the rotational torus. Here, $\M^{per}(\pm e_1)$ consist of the inner, short closed geodesic, which is hyperbolic. The set $\{\sig_F=1\}$ looks like a straight line to both side of the vertices at $\pm e_1$, which resembles the fact that the function $t\mapsto t C\exp(-\lam \frac{1}{t})$ in Corollary \ref{cor stable norm paper rational} vanishes in $t=0$ to infinite order. Drawing the level set $\{\al_F=1/2\}=\partial B^*(F)$, one sees ``corners'' at the ends of the segments $\F_{\pm e_1}$ corresponding to the ``straight parts'' of $\{\sig_F=1\}=\{\beta_F=1/2\}$, even though $\al$ is $C^1$. See Subsection 5.2 of \cite{stablenorm_paper} for the details of this example.

\begin{figure}\centering%[!htb]
 \includegraphics[scale=0.5]{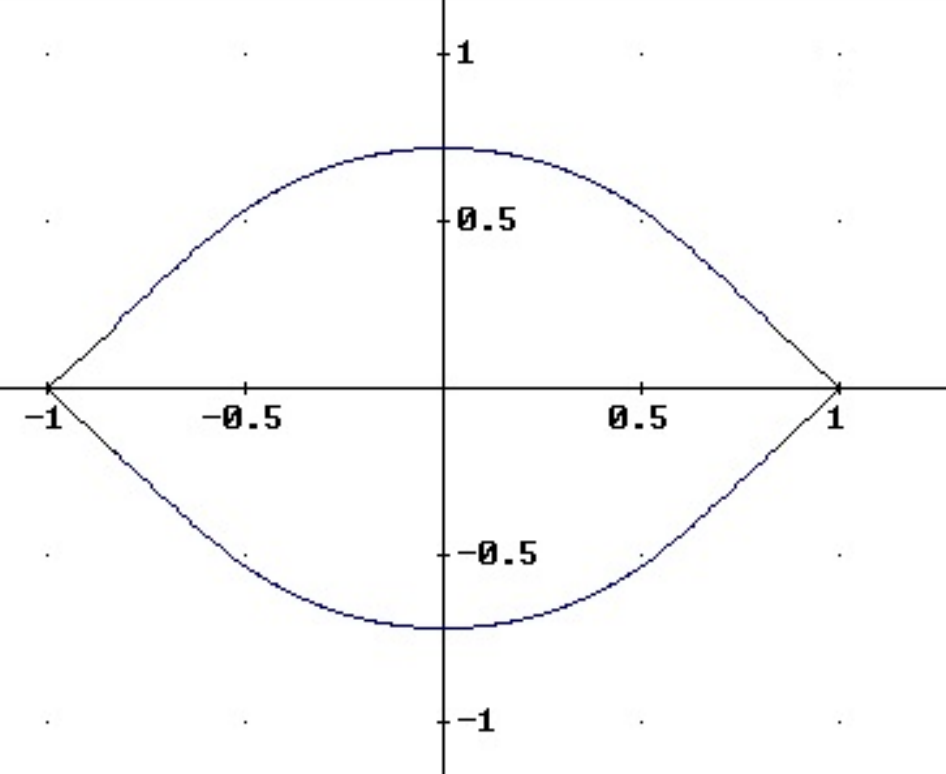}
 \caption{$\{\sig_F=1\}=\partial B(F)\subset\R^2\cong H_1(\T,\R)$ for $F$ obtained from the rotational torus. \label{fig_stable_rot}}
\end{figure} 

Using the ideas in \cite{mather2}, one can move from rational $\xi\in \Pi_\T$ to irrational $\xi\in \se-\Pi_\T$ via the convexity of $\sig_F$. This can be used to prove the differentiability of $\sig_F$ in irrational directions. Using this approach and the stronger estimate obtained in Corollary \ref{cor stable norm paper rational}, one can show the following.

\begin{thm}[\cite{stablenorm_paper}]\label{thm stable norm paper irrational}
 If the Finsler metric $F$ is chosen to be conformally generic (see Definition \ref{def generic}), if $\xi\in \se-\Pi_\T$ is irrational and if the set $Dp(\M(\g_\xi))\subset S_F\T$ is uniformly hyperbolic for the geodesic flow $\phi_F^t$, then there exist constants $C,\lam>0$, such that in all choices of rays $R \subset \R^2$ emanating from the origin there exist sequences $v_n\to 0, v_n\neq 0$, so that
 \begin{align*}
  & \sig_F(\xi+v_n) - \sig_F(\xi) - D\sig_F(\xi)[v_n] \leq |v_n|^{1/4} \cdot C \cdot \exp\left(-\lam \cdot \frac{1}{| v_n |^{1/4}} \right) .
 \end{align*}
\end{thm}

Let us note that the assumption that $Dp(\M(\g_\xi))$ is uniformly hyperbolic is satisfied in many situations, see \cite{stablenorm_paper}.

\begin{prop}\label{prop lecalvez}
 The following property of Finsler metrics on $\T$ is conformally generic: For an open and dense subset $U\subset \se$, the sets $Dp(\M(\g_\xi))$ are uniformly hyperbolic for all $\xi\in U$. The set $U$ strictly contains the set $\Pi_\T$ of rational points.
\end{prop}

Proposition \ref{prop lecalvez} shows that the sets $Dp(\M(\g_\xi))$ are often hyperbolic. On the other hand, using the KAM-theorem, the set $Dp(\M(\g_\xi))$ is a smooth KAM-torus, if the Finsler metric $F$ is close to a Finsler metric $F_0$ with an integrable geodesic flow (e.g.\ $F_0$ the Euclidean metric) and if the slope $\xi_2/\xi_1$ is a Diophantine number. In this case, we proved in \cite{stablenorm_paper} the following theorem, relying on ideas from \cite{siburg}.

\begin{thm}\label{thm stable norm paper irrational KAM}
 If $\xi \in \se-\Pi_\T$ is an irrational point and if the set $Dp(\M(\g_\xi))\subset S_F\T$ is a smooth KAM-torus for the geodesic flow $\phi_F^t$, then Mather's $\beta$-function $\beta_F$ is strongly convex near $\xi$, that is, there exists a constant $C>0$ with 
 \begin{align*}
   \beta_F(\xi+v) - \beta_F( \xi) - D\beta_F( \xi)[ v]\geq C \cdot| v|^2 
 \end{align*}
 for all $v\in\R^2$.
\end{thm}

The regularity of the KAM-torus $Dp(\M(\g_\xi))$ in Theorem \ref{thm stable norm paper irrational KAM} has to be at least $C^3$ for our argument to work.

In Theorem \ref{thm mather-bangert} \eqref{thm mather-bangert ii} we gave for rational $\xi$ a criterion on $\beta_F$ equivalent to $Dp(\M(\g_\xi))=Dp(\M^{per}(\g_\xi))$ being a KAM-torus (of regularity $C^0$). Theorems \ref{thm stable norm paper irrational} and \ref{thm stable norm paper irrational KAM} are a first step in the direction of solving the following problem.

\begin{prob}\label{quest irrat}
 Give a criterion on $\beta_F$ near a given irrational direction $\xi \in \se-\Pi_\T$, which is equivalent to the case where the set $Dp(\M(\g_\xi))$ is a $C^0$-KAM-torus for the geodesic flow $\phi_F^t$.
\end{prob}

\subsection{The higher genus case}\label{section applications higher genus}

Let us now assume that the genus $\mathfrak{g}$ of the closed, orientable surface $M$ is $\geq 2$. Recall the map
\[ \delta_F:\W_{dir}(H,F)\to \Gro(H,g)\cong \se \]
defined in Subsection \ref{section directed KAM}. We have the following result, which is analogous to Theorem \ref{thm horobdry torus}.

\begin{thm}\label{thm horobdry higher gen}
 Let $\xi \in \se$.
 \begin{enumerate}[(i)]
  \item\label{thm horobdry higher gen i} If $\xi\in \se-\Pi_M$, then $\card\delta_F^{-1}(\xi) =1$.
  
  \item\label{thm horobdry higher gen ii} If $\xi \in \Pi_M$, then $\card\delta_F^{-1}(\xi) =1$ if and only if $\pi(\M^{per}(\g_\xi))\subset H$ is connected, i.e.\ if the periodic minimal geodesics with endpoint $c(-\infty)=\xi$ foliate a strip in $H$ (possibly consisting of a single geodesic).
 \end{enumerate}
\end{thm}

The uniqueness result for irrational $\xi$ follows from Theorem \ref{thm irrat no intersections}. The characterization in the rational case follows from Theorem \ref{morse periodic} as in the case of $M=\T$. Recall that in the case $M=\T$, the map $\delta_F$ was a homeomorphism if and only if the geodesic flow $\phi_F^t$ was $C^0$-integrable (Corollary \ref{cor horobdry torus}), which is a very special situation. By Theorem \ref{thm generic loops}, this is never the case for conformally generic Finsler metrics $F$ on $\T$. In the higher genus case, quite the opposite holds, using also Theorem \ref{thm generic loops}.

\begin{cor}\label{cor horobdry higher gen}
 If $F$ is conformally generic, then the map
 \[ \delta_F:\W_{dir}(H,F)\to \Gro(H,g) \cong \W_{dir}(H,g) \cong \se \]
 is a homeomorphism.
\end{cor}

We have the natural compactification of the \Poincare disc $H$ as a subset of $\R^2$ by the unit circle $\se$. Corollary \ref{cor horobdry higher gen} can be rewritten using the horofunction compactification, see Subsection \ref{section horofctns}.

\begin{cor}\label{cor horocomp higher gen}
 If $F$ is conformally generic, then the embedding
 \[ i_F : H \to C^0(H)/_\sim \]
 extends to a homeomorphism from the closed disc $H\cup \se \subset \R^2$ with the Euclidean topology to the horofunction compactification $\overline{i_F(H)}$.
\end{cor}

In Subsection \ref{section hadamard horo} we saw that this was the case a priori only for Finsler metrics on $H$ with no conjugate points, meaning $\RR_-=\M=S_FH$.

In Subsection \ref{section mather torus} we saw that the map $\delta_F:\W_{dir}(H,F)\to \se$ shares some features with the convexity properties of the set $B^*(F)=\{\al_F\leq 1/2\}$. Reading Corollary \ref{cor horobdry higher gen} in this light proposes the following problem.

\begin{prob}\label{prob massart}
 Given a conformally generic Finsler metric $F$ on a closed, orientable surface $M$ of genus $\mathfrak{g}\geq 2$, are the convexity properties of the set $B^*(F)$ the same as those of $B^*(g)$ with respect to the metric $g$ constant curvature $-1$?
\end{prob}

In order to discuss an approach to Problem \ref{prob massart}, we consider Mather theory in the higher genus case. We shall follow ideas of Massart from \cite{massart2}. First, observe that there are no continuous $\phi_F^t$-invariant graphs in $S_FM$, so any weak KAM solution has necessarily points of non-differentiability and the Aubry sets $\A_{[\eta]}\subset S_FM$ with $[\eta]\in\partial B^*(F)$ cannot project onto $M$. Instead, the projection $\pi(\A_{[\eta]})$ yields a geodesic lamination.

\begin{defn}
 An {\em $F$-geodesic lamination of $M$} is a closed subset $L\subset M$ consisting of pairwise disjoint images of $F$-geodesics.
\end{defn}

Motivated by the ideas of Massart \cite{massart2} we can assign to a geodesic lamination $L\subset M$ its ``homological complexity''.

\begin{defn}
 Let $L\subset M$ be an $F$-geodesic lamination. A closed, piecewise $C^1$ curve $c:\R/\Z\to M$ is an $(\e,L)$-curve, if the part $c(\R/\Z)-L$ has total length at most $\e$. Then let us write
 \[ V(\e,L) := \Span \{ [c]\in H_1(M,\R) :  \text{$c$ is a closed $(\e,L)$-curve} \} \]
 and consider the vector subspace
 \[ V(L) := \bigcap_{\e>0} V(\e,L) \subset H_1(M,\R). \]
 The {\em homological complexity of the $F$-geodesic lamination $L$} is given by the integer
 \[ \hc(L) := \dim V(L) . \]
\end{defn}

Note that in the definition of $V(L)$ it makes no difference if we take the $F$- or $g$-length. As an example consider the case of the 2-torus $(\T,F)$. The $F$-geodesic lamination $L$ of $\T$ induced by $\M^{rec}(\xi)$ for $\xi$ with irrational slope satisfies $\hc(L)=2$, which follows from the results in Subsection \ref{section irrational torus}. Of course, the same holds of $\M(\xi)\supset\M^{rec}(\xi)$. If $\xi$ has rational slope, then for the lamination $L$ induced by $\M^{per}(\g_\xi)$ with $\g_\xi(t)=x+t\xi$ we have $\hc(L)=2$, if $\T$ is foliated by periodic minimal geodesics in the homotopy class of $\g_\xi$ and $\hc(L)=1$ otherwise. These facts are related to Theorem \ref{thm mather-bangert}, as we shall see. The sets $\M^{per}(\g_\xi)\cup (\M(\g_\xi)\cap \RR_-^i(\xi))$ defined after Theorem \ref{morse periodic} induce laminations $L$ with $\hc(L)=2$ as the heteroclinics can be used to span homology classes linearly independent of $[\g_\xi]$. This is related to the strict convexity of $\al_F$ in the endpoints $[\eta_i]$ of the segment $\mathcal{F}_\xi\subset\partial B^*(F)$.

We recall definitions from \cite{massart1} and \cite{massart2}.

\begin{defn}
 A {\em flat of $B^*(F)=\{\al_F\leq 1/2\}$} is the set $\mathcal{F} = A\cap \partial B^*(F)$, where $A$ is an affine subspace of $H^1(M,\R)$ meeting $B^*(F)$, but not the interior $B^*(F) - \partial B^*(F)$. The dimension $\dim \mathcal{F}$ of a flat $\mathcal{F}$ is the dimension $\dim\Aff(\mathcal{F})$ of the affine subspace $\Aff(\mathcal{F})$ spanned by $\mathcal{F}$. The interior of $\mathcal{F}$ is the interior of $\mathcal{F}$ in $\Aff(\mathcal{F})$. Given $[\eta]\in \partial B^*(F)$, the {\em maximal face $\mathcal{F}_{[\eta]}\subset \partial B^*(F)$ of $[\eta]$} is given by the maximal face (with respect to inclusion) containing $[\eta]$ in its interior.
\end{defn}

Note that, by $\beta_F$ being the convex conjugate of $\al_F$, the face $\mathcal{F}_{[\eta]}$ relates to directions of non-differentiability of $\beta_F$.

The following problem was stated as part of Theorem 1 in \cite{massart2}. While Massart pointed out some gaps in the erratum \cite{massart erratum}, observe that our Theorem \ref{thm rays-paper} could fix the gaps in \cite{massart2}. See also the arguments in the proof of Theorem \ref{thm hom complex}.

\begin{prob}\label{prob massart 1}
 Let $M$ be a closed orientable surface of genus $\mathfrak{g}\geq 1$. Is it true, that
 \[ \hc(\pi(\A_{[\eta]})) = \dim H^1(M,\R) - \dim \mathcal{F}_{[\eta]} ? \]
\end{prob}

This means that the more homology classes the projected Aubry set spans, the more strict convexity should be seen in $\al_F$ or, equivalently, the more differentiable does $\beta_F$ become. Note that by convexity of $\beta_F$ there exist many points of differentiability for $\beta_F$ and hence, many faces $\mathcal{F}_{[\eta]}$ are single points. In Theorem \ref{thm mather set irrat} we observed in this case a uniqueness result for weak KAM solutions for $M=\T$. Hence, one could expect that for many $[\eta]$, the $\eta$-weak KAM solution is unique. Observe also, that Problem \ref{prob massart 1} is related to conjectures due to \Mane in dimension two, see \cite{massart2}.

We wish to relate the homological complexity of the Aubry sets $\A_{[\eta]}$ to the background metric $g$. For $v\in \A_{[\eta]}$, consider any lift $\tilde c_v:\R\to H$ of $c_v:\R\to M$ and let $\g \subset H$ be the unique $g$-geodesic shadowing the minimal geodesic $\tilde c_v$. The totality of all such $g$-geodesics, projected back to $M$ yields a $g$-geodesic lamination
\[ \Lam(\A_{[\eta]})\subset M . \]
The fact that it is a lamination follows from the corresponding property of $\A_{[\eta]}$. Similarly, we obtain a lamination $\Lam(\M_{[\eta]})\subset M$ from the Mather set $\M_{[\eta]}\subset S_FM$ and $\Lam([u])$ from the set of $[u]$-calibrated curves with respect to an $\eta$-weak KAM solution $[u]$. More generally, let $L\subset M$ be an $F$-geodesic lamination consisting of minimal geodesics. Lifting each $F$-geodesic in $L$ to $H$, we can associate to it a unique $g$-geodesic in finite distance. The totality of such $g$-geodesics in $M$ defines a $g$-geodesic lamination $\Lam(L)$. Note that the structure of $g$-geodesic laminations is studied e.g.\ in \cite{casson-bleiler}.

The following theorem points to an affirmative solution to Problem \ref{prob massart} via the formula in Problem \ref{prob massart 1}. 

\begin{thm}\label{thm hom complex}
 If $L\subset M$ is an $F$-geodesic lamination consisting entirely of non-closed of minimal geodesics, then the homological complexity satisfies
 \[ \hc(L) = \hc(\Lam(L)) . \]
\end{thm}

\begin{proof}[Sketch of the proof of Theorem \ref{thm hom complex}]
 (see Figure \ref{fig_crown}) We abbreviate $\Lam=\Lam(L)$. Each $g$-geodesic $\g$ in $\Lam$ is the image of an $F$-minimal geodesic $c_\g$ in $L$ under the correspondence $L\mapsto \Lam(L)$ ($c_\g$ need not be unique). Imagine an isotopy $\{\psi_t\}_{t\in[0,1]}$ of $M$ taking each $\g$ in $\Lam$ into $c_\g$ in $L$, not worrying about the existence in this sketch. By Lemma 4.4 in \cite{casson-bleiler}, each component $U$ of $M-\Lam$ is either a finite-sided, ideal polygon with geodesic boundary or consists of a compact subsurface $M_0\subset M$ with finitely many boundary components consisting of closed geodesics, attaching to each boundary geodesic a ``crown''. Here we use that there are no closed geodesics in $L$ and as $L$ is closed, there are no closed geodesics in $\Lam$. In any case, each boundary component of $U$ consists of a finite, cyclic family of successively asymptotic $g$-geodesics in $\Lam$. Let $\al$ be an $(\e,\Lam)$-curve. Any subsegment $\al_0$ of $\al$ contained in the closure of a component $U\subset M-\Lam$ has to lie $\e$-close to a cyclic boundary part of $U$ and for $\e\ll 1$ it has to be disjoint from the other boundary parts. Under the isotopy $\{\psi_t\}$, the segment $\al_0$ becomes a segment $\psi_1(\al_0)$ connecting two points in the image under $\psi_1$ of the cyclic boundary of $U$, lying on geodesics of the form $c_\g=\psi_1(\g)$. By Theorem \ref{thm rays-paper}, successive geodesics in the cyclic boundary part of $\psi_1(U)$ come arbitrarily close to each other in the ends and we push $\psi_1(\al_0)$ further into the ends (using parts of $L$) obtaing a segment $a_0$, so that $a_0-L$ has length at most that of $\al_0-\Lam$. This shows that any $(\e,\Lam)$-curve $\al$ is homotopic to an $(\e,L)$-curve $a$, if $\e\ll 1$. By the same arguments one shows that any $(\e,L)$-curve is homotopic to an $(\e,\Lam)$-curve for small $\e$. The claim follows.
\end{proof}

\begin{figure}\centering%[!htb]
 \includegraphics[scale=0.5]{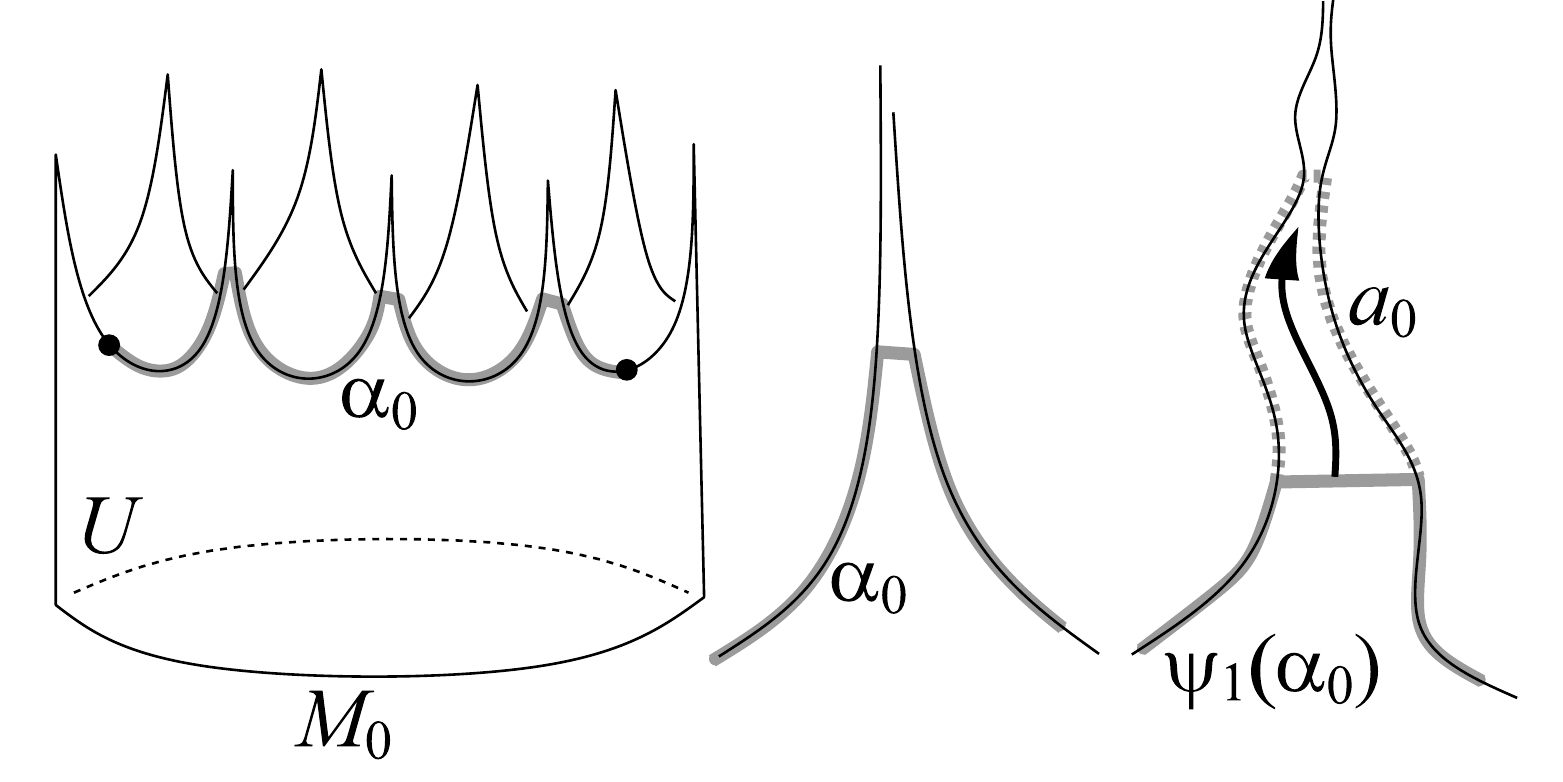}
 \caption{The ``crown'' as part of the component $U$ with the curve segment $\al_0$ and $a_0$ in $\pi_1(U)$. \label{fig_crown}}
\end{figure}

Observe that we used Theorem \ref{thm rays-paper} in order to prove Theorem \ref{thm hom complex}, applied to {\em simple} geodesics. While Theorem \ref{thm rays-paper} has a rather complicated proof in general, the case of simple geodesics is much easier, see Lemma 2.14 in \cite{min_rays} and Proposition 3.16 in \cite{paper1}. We also expect in this case, that it is not hard to prove that $\liminf$ can be replaced by $\lim$ in Theorem \ref{thm rays-paper}.

We expect (a version of) Theorem \ref{thm hom complex} to be true without the assumption that the geodesics in $L$ are not closed, but this assumption simplifies the proof. For the general case involving closed geodesics, one might want to assume that $F$ is conformally generic is in order to have $\M(\g)$ consist of a single minimal geodesic for periodic $g$-geodesics $\g$ (Theorem \ref{thm generic loops}). This shows that $F$-geodesics with rational asymptotic directions are asymptotic (Theorem \ref{morse periodic}). Without the genericity assumption, the number $\hc(L)$ might a priori decrease if there exists an embedded annulus $A\subset M$ bounded by pair of homotopic, periodic $F$-geodesics in $L$ and if in the lamination of $A$ by $L$ not every gap is bridged by a heteroclinic contained in $L$.

Theorem \ref{thm hom complex} in connection with Problems \ref{prob massart} and \ref{prob massart 1} motivates the following question.

\begin{prob}
 Given $[\eta]\in\partial B^*(F)$, does there exist $[\eta_0]\in \partial B^*(g)$ with
 \[ \Lam(\pi(\A_{[\eta]}^F)) = \pi(\A_{[\eta_0]}^g)  \]
 and vice versa, writing $\A_{[\eta]}^F$ for the Aubry set with respect to $F$?
\end{prob}

We have related the structure of laminations by $F$-minimal geodesics to the background metric $g$. Recall that in Subsection \ref{section av action} we defined the Mather set $\M^h$ of a homology class $h\in H_1(M,\R)$. Note that $H_1(M,\Z)$ lies naturally in $H_1(M,\R)$. Theorem \ref{thm hom complex} applies in particular to projected Mather sets $\pi(\M^h)\subset M$. If $h$ is completely irrational, then $\hc(\pi(\M^h))$ can be linked to $g$-geodesic laminations. Let us in this connection state the following result of Mather sets $\M^h$ with $h$ being rational, generalizing Theorem \ref{mather set rational torus}, see \cite{contreras2}. 

\begin{thm}
 If for $h\in \partial B(F)$ there exists $a>0$ with $ah\in H_1(M,\Z)$, then $\M^h$ consists of periodic minimal geodesics.
\end{thm}

\section{Topological entropy}\label{section entropy}

In this section, let $M$ be a closed, orientable surface. We saw that minimal geodesics lie at the heart of integrable behavior in geodesic flows. In this section, we study a number expressing the dynamical complexity of the geodesic flows, namely the topological entropy $\h(\phi_F^t,S_FM)$. Let us refer to \cite{KH} for the definition and more information on the topological entropy $\h(\phi^t,X)=\h(\phi^1,X) \in [0,+\infty]$ of a dynamical system $\phi^t:X\to X$ ($t\in \Z$ or $\R$) in a compact metrizable space $X$. It measures the growth of distinguishable orbits on an exponential scale. By a result of Katok (Corollary 4.3 in \cite{katok1}), $\h(\phi^t,X)>0$ is in low dimensional systems ($\dim X=2, t\in \Z$ or $\dim X=3, t\in\R$) equivalent to the existence of non-trivial, hyperbolic invariant sets (horse-shoes) and the existence of transverse intersections of stable and unstable manifolds. On the other hand, if $\h(\phi^t,X)=0$, one might expect a rather ordered behavior of $\phi^t$ -- the smaller $\h(\phi^t,X)$, the more integrable behavior there can be, in particular for the set of minimal geodesics, if $\phi^t=\phi_F^t$ and $X=S_FM$. This is indeed the case, as we shall see. Moreover, the structure in the set of minimal geodesics can be applied to study the overall geodesic flow.

\subsection{The entropy of minimal geodesics}\label{section GKOS}

We saw that minimal $F$-geodesics behave in much the same way as $g$-geodesics. Hence, one expects that the topological entropy $\h(\phi_F^t,Dp(\M))$ of the restricted geodesic flow $\phi_F^t|_{Dp(\M)}$ resembles that of $\phi_g^t$ in $S_gM$. In \cite{Ma}, Manning introduced the volume entropy (also called volume growth) $h(F)$ of $(H,F)$ defined by
\[ h(F):=\lim _{r \rightarrow \infty }\frac{1}{r}\log \vol_g B_F(x,r), \]
where $x\in H$ and $B_F(x,r)$ denotes the (forward) open $d_F$-ball with center $x$ and radius $r$. Note that, if a Riemannian metric $g'$ is equivalent to $g$, then we can measure $h(F)$ both in $g$ and $g'$ with the same resulting number. Hence, $g$ does not appear in the notation $h(F)$. We have in general the following relation.

\begin{thm}[Theorem 9.6.7 in \cite{KH}]
 \[ \h(\phi_F^t,Dp(\M)) \geq h(F). \]
\end{thm}

Note that this theorem holds in any dimension. In the case of non-positive curvature equality holds. Subsequently this was generalized by Freire and \Mane \cite{FM} to metrics without conjugate points, i.e.\ the case where $Dp(\M)=S_FM$. 

The following problem was posed by Knieper.

\begin{prob}
 If $(M,g)$ is a closed hyperbolic manifold and $F$ an arbitrary Finsler metric on $M$, does the equality $\h(\phi_F^t,Dp(\M)) = h(F)$ hold? Is it true that $\h(\phi_F^t, \M(\g))=0$ for all $g$-geodesics $\g$ in $H$?
\end{prob}

It is shown in \cite{GKOS}, that an affirmative answer to the second question entails an affirmative answer to the first one. In the case $\dim M=2$, we proved in \cite{GKOS} the following.

\begin{thm}\label{thm GKOS}
 If $M$ is a closed surface and $F$ is a Finsler metric on $M$, then
 \[ \h(\phi_F^t,Dp(\M)) = h(F). \]
\end{thm}

Using Theorems \ref{morse periodic}, \ref{thm bangert irrat} \eqref{thm bangert irrat iv} and \ref{thm irrat no intersections} and the techniques in \cite{GKOS} one can prove the following strengthening of Theorem \ref{thm GKOS}.

\begin{thm}\label{thm shaperning GKOS}
 If $M$ is a closed surface of genus $\mathfrak{g}\geq 1$ and $F$ is a Finsler metric on $M$, then for all $g$-geodesics $\g:\R\to H$ we have
 \[ \h(\phi_F^t,\M(\g)) = 0 . \]
\end{thm}

Hence, the entropy of minimal geodesics in dimension two resembles properties in constant curvature. In particular, for $M=\T$, the entropy $\h(\phi_F^t,Dp(\M))$ vanishes, as the volume of balls $B(x,r)$ grows only quadratically.

\subsection{Geodesic flows with vanishing entropy}\label{section geod flows vanishing entropy}

By a result of Dina\-burg (Corollary 4.2 in \cite{dinaburg}) we find
\[ \h(\phi_F^t,S_FM)>0, \]
if the genus of $M$ is $\mathfrak{g}\geq 2$ (this also follows from Theorem \ref{thm GKOS}). In this subsection we shall assume that $M$ is either the 2-sphere $M=\s$ or the 2-torus $M=\T$. Both surfaces admit whole families of completely integrable geodesic flows \cite{fomenko}. Due to a result of Paternain \cite{paternain}, these completely integrable Hamiltonian systems have vanishing topological entropy. Hence, on $M=\s,\T$ there exists a variety of examples of Finsler metrics $F$ with
\[ \h(\phi_F^t,S_FM) = 0. \]
A natural question is, whether the converse is true, i.e.\ does the vanishing of the entropy of $\phi_F^t$ in $S_FM$ entail some form of integrability?

For integrable geodesic flows, the closure of orbits stays small. We propose the following more general definition.

\begin{defn}\label{def weak integrability}
 A dynamical system $\phi^t:X\to X$ ($t\in T$) in a topological space $X$ is {\em weakly integrable}, if for all $x\in X$, the orbit $\O(x,\phi^t)=\{\phi^t x:t\in T\}$ is nowhere dense in $X$.
\end{defn}

One might ask, whether $\h(\phi_F^t,S_FM) = 0$ implies at least weak integrability. This is not the case. We write $\mu_F$ for the Liouville measure in $S_FM$ (the pullback of the canonical measure in the symplectic manifold $T^*M$ under the Legendre transform associated to $\frac{1}{2}F^2$, restricted to $S_FM$, which is a smooth measure invariant under $\phi_F^t$).

\begin{thm}[Katok \cite{katok example}]\label{thm sphere katok}
There exist reversible Finsler metrics $F$ on $\s$, arbitrarily $C^k$-close to the standard round metric for fixed $2\leq k<\infty$, whose geodesic flow has vanishing topological entropy $\h(\phi_F^t,S_F\s)=0$ and which admits two closed, $\phi_F^t$-invariant sets $A_0,A_1\subset S_F\s$, such that $A_1=\{-v: v\in A_0\}$, $\phi_F^t|_{A_i}$ is ergodic and such that the measure $\mu_F(S_F\s-(A_0\cup A_1))$ is positive but arbitrarily small. In particular, $\phi_F^t|_{A_i}$ is transitive.
\end{thm}

Dropping the reversibility assumption $F(-v)=F(v)$ on $F$, Katok \cite{katok example} was able to construct similar examples, where $\phi_F^t$ is ergodic in the whole unit tangent bundle $S_F\s$ (in this example $\phi_F^t$ has only two closed orbits). One question is, whether Theorem \ref{thm sphere katok} is optimal in the sense that transitivity or ergodicity can be obtained in the whole unit tangent bundle in the reversible case. We proved the following theorem in \cite{paper ergodic}.

\begin{thm}[\cite{paper ergodic}]\label{thm FrHa sphere}
 Let $F$ be a reversible Finsler metric on $\s$, such that every closed $F$-geodesic has a pair of conjugate points along itself. If the geodesic flow $\phi_F^t:S_F\s\to S_F\s$ is transitive, i.e.\ if there exists a dense orbit in $S_F\s$, then 
 \[ \h(\phi_F^t,S_F\s)>0. \]
\end{thm}

Hence, under the condition on conjugate points (which is fulfilled, e.g., if $F$ has positive flag curvatures), dense geodesics and in particular ergodicity of the geodesic flow in $S_F\s$ imply the existence of a non-trivial, hyperbolic invariant set for $\phi_F^t$, cf. Corollary 4.3 in \cite{katok1}. Assuming conversely, that the topological entropy vanishes there will be a certain amount of structure in the dynamics of $\phi_F^t$: orbits cannot become as large as $S_F\s$.

We can ask the following questions.

\begin{prob}\label{Q1 sphere}
 Does transitivity of the geodesic flow $\phi_F^t$ in $S_F\s$ of a general reversible Finsler metric $F$ on $\s$ imply $\h(\phi_F^t,S_F\s)>0$? Dropping the reversibility assumption, does transitivity together with at least three distinct closed orbits of $\phi_F^t$ imply $\h(\phi_F^t,S_F\s)>0$?
\end{prob}

\begin{prob}\label{Q2 sphere}
 Do there exist Riemannian metrics $g$ on $\s$ with everywhere strictly positive curvature, whose geodesic flow $\phi_g^t$ is ergodic in $S_g\s$?
\end{prob}

Problem \ref{Q1 sphere} was posed by G. Knieper. Our result answers the question affirmatively in the presence of conjugate points, in particular in the case of positive curvature. The general case is a topic for future research. Problem \ref{Q2 sphere} is a long-standing open problem. For instance, by the results of V. J. Donnay \cite{donnay} there exist Riemannian metrics on $\s$, whose geodesic flow is ergodic in $S_F\s$; however, for these metrics there exist large regions of negative curvature and it is the negative curvature that creates ergodicity by means of hyperbolicity. Our result shows that ergodicity and positive curvature would necessarily entail hyperbolicity in the geodesic flow by means of topological entropy.

\abs

Let us move to $M=\T$. Recall that $g$-geodesics are straight lines. We shall write $\g_\xi(t)=\xi t$. The assumption $\h(\phi_F^t,S_F\T)=0$ leads to strong integrable behavior on a large scale. Intuitively, $\h(\phi_F^t,S_F\T)=0$ prevents the transverse intersection of (un)stable manifolds by the result of Katok mentioned earlier, while the gap-condition in Subsection \ref{section multibump} generalizes such a transverse intersection. The following theorem (proved in \cite{paper2}) follows easily from Theorem \ref{thm multibump}, using the oscillating local minimizers to generate entropy. Observe that by Theorem \ref{thm shaperning GKOS}, one cannot use minimal geodesics, i.e.\ global minimizers in $\T$ to generate entropy. A version of the following theorem for reversible Finsler metrics is due to Glasmachers and Knieper \cite{glasm1}, \cite{glasm2}, for monotone twist maps it was given by Angenent \cite{angenent}.

\begin{thm}\label{thm torus htop=0}
 Let $F$ be a Finsler metric on $\T$ with $\h(\phi_F^t,S_F\T)=0$. Then for all $\xi\in \se$ the sets $Dp(\M(\g_\xi))$ project onto $\T$ and consist of one or two continuous, $\phi_F^t$-invariant graphs (always one if $\xi$ is irrational, two if $\xi$ is rational and $\M^{per}(\g_\xi)\neq \M(\g_\xi)$).
\end{thm}

Equivalently, the boundary elements $[u_0],[u_1]$ of $\delta_F^{-1}(\xi)$ are $C^1$, see Proposition \ref{bounding horofunctions}. For irrational $\xi$, the unique element $[u]\in\delta_F^{-1}(\xi)$ is $C^1$. See Figure \ref{fig_htop_0} for invariant sets in Theorem \ref{thm torus htop=0}.

\begin{figure}%[!htb]
\centering
\includegraphics[scale=1.0]{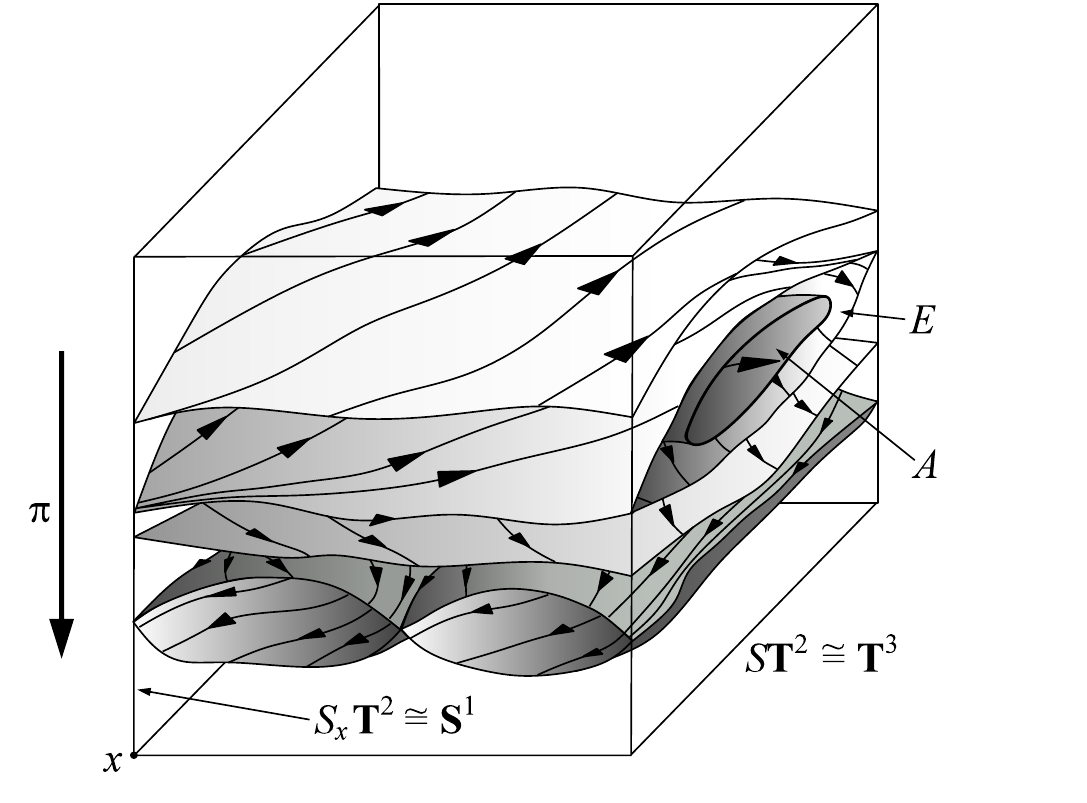}
\caption{The invariant graphs and flow lines of $\phi_F^t$ in the unit tangent bundle $S_F\T\cong \mathbb{T}^3$ occurring in Theorem \ref{thm torus htop=0} in the case of $\h(\phi_F^t)=0$. The horizontal plane can be thought of as the base $\T$. Compare this to the case of the rotational torus or the simple pendulum. One can see an elliptic tube $E$ containing a $\phi_F^t$-invariant sub-tube $A$, where $\phi_F^t|_A$ might be ergodic due to Theorem \ref{thm torus katok}. \label{fig_htop_0}}
\end{figure}

\begin{remark}[Polynomial entropy]
 Another notion of entropy is polynomial or slow entropy, measuring not the exponential growth rate of orbits, but the polynomial one. While for topological entropy transverse intersections of (un)stable manifolds trigger entropy, for the slow entropy one way to generate entropy is by heteroclinic connections. In this way, Bernard and Labrousse \cite{labrousse} were able to characterize the flat torus by minimizing the slow entropy: a metric minimizing the slow entropy cannot have heteroclinics. See also Corollary \ref{cor horobdry torus}. One could also study the links between slow entropy and the set $B(F)=\{\sig_F\leq 1\}$. Note also \cite{burago-ivanov tn} relating the stable norm to volume asymptotics, which are in turn related to entropy by Subsection \ref{section GKOS}.
\end{remark}

What remains open in Theorem \ref{thm torus htop=0} is a description of the dynamics of $\phi_F^t$ in the complement of $Dp(\M)$. The following terminology is motivated by the often found presence of elliptic closed geodesics there, see also the example in Subsection \ref{section rot torus}.

\begin{defn}\label{def elliptic tubes}
 Let $F$ be a Finsler metric on $\T$ with vanishing entropy $\h(\phi_F^t,S_F\T)=0$. An \emph{elliptic tube} is a connected component $E$ of $S_F\T - Dp(\M)$.
\end{defn}

\begin{remark}
 Note that an elliptic tube is surrounded by a single set $\M(\g_\xi)$ for some rational direction $\xi$. Notably, in the Riemannian case the only known examples of geodesic flows with $\h(\phi_F^t,S_F\T)=0$ arise from Liouville metrics of the form
 \[ g_{(x_1,x_2)}(v,w) = (f(x_1)+g(x_2))\cdot \la v,w\ra_{euc} , \]
 which have elliptic tubes only in the directions $-e_1,e_1,-e_2,e_2$. Liouville metrics are a simple generalization of rotational metrics discussed in Subsection \ref{section rot torus}. In the Finsler case, Theorem \ref{thm torus katok} below yields more examples in the Finsler setting. In a different direction, one could use monotone twist maps with many Mather sets of rational homology, that are not invariant curves; cf.\ the proof of Theorem 3.3 in \cite{neumann}. We show in Subsection \ref{section twist map} how such examples yield examples of reversible Finsler metrics on $\T$ with integrable geodesic flows, which have elliptic tubes of several directions. Hence, there exist in the Finsler setting many examples of integrable geodesic flows other than Liouville metrics.
\end{remark}

Our next theorem shows that, even if $\h(\phi_F^t,S_F\T)=0$, elliptic tubes can contain complicated dynamical behavior of $\phi_F^t$. For the proof we apply Katok's construction, which also led to the examples in Theorem \ref{thm sphere katok}, to a particular rotational metric on $\T$.

\begin{thm}[\cite{paper ergodic}]\label{thm torus katok}
 There exist reversible Finsler metrics $F$ on $\T$ with vanishing topological entropy $\h(\phi_F^t,S_F\T)=0$ and an elliptic tube $E\subset S_F\T$ containing a $\phi_F^t$-invariant, closed sub-tube $A\subset E$ with non-empty interior, such that $\phi^t_F|_A$ is ergodic, hence transitive. The measure $\mu_F(E-A)$ is positive, but can be made arbitrarily small.
\end{thm}

In the spirit of Theorem \ref{thm FrHa sphere}, our next theorem shows that complicated behavior in all of $E$ is excluded.

\begin{thm}[\cite{paper ergodic}]\label{thm FrHa torus}
 Let $F$ be a reversible Finsler metric on $\T$ with vanishing topological entropy $\h(\phi_F^t,S_F\T)=0$ and suppose that $E\subset S_F\T$ is an elliptic tube for $\phi_F^t$. Then the set
 \[ \left \{ v\in E : \overline{ \{ \phi_F^tv : t\in \R \}} \cap\partial E \neq \emptyset \right \} \]
 is nowhere dense in $E$. In particular, $\phi_F^t|_E$ cannot be transitive, hence not ergodic.
\end{thm}

Recall that a nowhere dense set is a set whose closure has empty interior. In particular, the complement of a nowhere dense set contains an open and dense set. Hence, Theorem \ref{thm FrHa torus} shows that topologically almost every $\phi_F^t$-orbit in $E$ is boun\-ded away from the boundary $\partial E$ of the elliptic tube. This indicates that there are large, closed, $\phi_F^t$-invariant sets in the interior of $E$, not touching the boundary $\partial E$. A possible picture would be a sequence $\{A_n\}$ of nested, invariant, closed tubes inside the elliptic tube $E$, such that $E=\cup A_n$. However, the sets
\[ A_n := \big\{ v\in E : \inf_{t\in \R}d( \partial E , \phi_F^tv ) \geq 1/n \big\} \]
might a priori be quite exotic and Theorem \ref{thm FrHa torus} guarantees that $\cup A_n$ equals $E$ only up to a topologically small set.

\abs

Both Theorems \ref{thm FrHa sphere} and \ref{thm FrHa torus} rely on the construction of a \Poincare section $N\subset S_FM$, which can be used to describe (more or less) the whole geodesic flow. This reduces $\phi_F^t:S_FM\to S_FM$ to a diffeomorphism $\phi$ of a surface $N$, which we study in Subsection \ref{section FrHa} below.

\subsection{Zero-entropy surface diffeomorphisms}\label{section FrHa}

We work in the setting of \cite{Fr-Ha}. Let $\mu$ be a measure on the 2-sphere $\s$ topologically conjugate to the Lebesgue measure (i.e.\ there exists a homeomorphism of $\s$ conjugating $\mu$ to the Lebesgue measure). Let $N$ be a surface diffeomorphic to $\s$ with $n\geq 0$ disjoint, smoothly bounded, open discs removed. Collapsing each boundary circle $\partial_iN$ of $N$ into a point $p_i\in \s$ defines a $C^0$-quotient map $\pi_N:N\to \s$, whose restriction $\Int N\to \s-P$ is a $C^\infty$-diffeomorphism, where $P=\{p_1,...,p_n\}$. If $\phi:N\to N$ is an orientation-preserving $C^\infty$-diffeomorphism, leaving each boundary component $\partial_iN$ invariant, we can define a homeomorphism $\psi:\s\to \s$ by $\psi\circ\pi_N=\pi_N \circ \phi$, such that $P\subset \Fix(\psi)$. We denote by
\[ \H=\H(\s,P,\mu)\subset \text{Homeo}_+(\s) \]
the set of all so obtained homeomorphisms $\psi:\s\to \s$, that in addition preserve the measure $\mu$ and have vanishing topological entropy
\[ \h(\psi,\s)=0 . \]
The above setting applies to geodesic flows on surfaces with vanishing entropy, cf.\ Subsection \ref{section geod flows vanishing entropy}.

We reformulate of parts of Theorem 1.2 from \cite{Fr-Ha}.

\begin{thm}[Franks, Handel]\label{thm 1.2 Fr/Ha}
 Let $\psi\in\H$ and consider the set $\A$ of maximal, $\psi$-invariant, open annuli $U\subset \s-\Fix(\psi)$. Then the elements of $\A$ are pairwise disjoint and the union $\bigcup_{U\in\A}U$ is (open and) dense in $\s-\Fix(\psi)$.
\end{thm}

Franks and Handel prove more on the dynamics of $\psi$ in the annuli $U\in \A$. For instance, each restriction $\psi|_U$ has a continuous integral of motion given by the rotation number. Using these results, we were in \cite{paper ergodic} able to prove Theorems \ref{thm FrHa sphere} and \ref{thm FrHa torus}.

Recall the Definition \ref{def weak integrability} of weak integrability. We already saw that $\h(\psi,\s)=0$ does not imply weak integrability, but one can ask under what conditions can the closure of an orbit contain an open set.

\begin{prob}\label{prob FrHa}
 If $\psi\in \H$ and $x\in \s$, such that the interior
 \[ \Int\overline{\O_\psi(x)}\neq \emptyset , \]
 does there exists $q\geq 1$ and a $\psi^q$-invariant, open annulus
 \[ U\subset \s-\Per(\psi) \]
 with $x\in U$?
\end{prob}

Note that if the annulus $U\subset \s-\Per(\psi)$, then in particular, $\psi^q|_U$ is an irrational pseudo-rotation. Bramham \cite{barney} proves for the closed disc, that an irrational pseudo-rotation can be approximated by a sequence of integrable maps $\psi_n$. If the answer to Problem \ref{prob FrHa} is affirmative and if one could extend the main result of \cite{barney} to $\psi^q:U\to U$ (one can use that $\psi$ comes from a diffeomorphism of a compact surface in order to control the system), then the whole system $\psi:\s\to\s$ would be weakly integrable in some set (the complement of open annuli, where $\psi$ is a pseudo-rotation) and approximated by properly integrable ones in the complement. This would be a step towards answering an old question of Katok.

\begin{prob}[Katok]
 Is every element $\psi\in\H$ a limit in some sense of integrable $\psi_n\in \H$?
\end{prob}

We give some ideas as to Problem \ref{prob FrHa}, but not having checked most of the details. Note in this connection also Theorem 3.2 and its proof in \cite{paper ergodic}.

\begin{proof}[Ideas for the proof of Problem \ref{prob FrHa}]
 For $q\in \N$ let $\A(q)$ be the set of maximal, open, $\psi^q$-invariant annuli in $\s-\Fix(\psi^q)$ as in \cite{Fr-Ha}. For $U\in \A(q)$, the rotation number $\rho_{\psi^q,U}: U \to \R/\Z$ is defined everywhere, continuous and $\psi^q$-invariant. Consider the set
 \[ X := \bigcup_{q\in\N}\bigcup_{U\in\A(q)} \bigcup_{r\in \R/\Z \text{ irrat.}} \Int \rho_{\psi^q,U}^{-1}(r) . \]
 Note that $X$ is a union of open annuli contained in $\s-\Per(\psi)$ (use Theorem 1.4 (3) in \cite{Fr-Ha} to show that each component is either an annulus or an open disc; indeed, any non-contractible loop in a component of $X$ has to be essential in the corresponding annulus; open discs are $\psi^{kq}$-invariant for some $k\geq 1$ by recurrence and $\mu$ being positive on open sets and hence contain periodic points by Brouwer's theorem, contradicting the irrationality of $r$; open annuli have to be $\psi^q$-invariant by Theorem 1.4 (3) in \cite{Fr-Ha}) and $X$ as a set is $\psi$-invariant (the components of $X$ can be permuted).

 We have to show that any point $x\in \s$ with $\Int\overline{\O_\psi(x)}\neq \emptyset$ lies in $X$. For this, suppose the contrary. By Theorem 1.2 (3) in \cite{Fr-Ha} and Baire's category theorem, one can probably show that any point $x$ with $\Int\overline{\O_\psi(x)}\neq \emptyset$ has to belong to $\cup_{U\in\A(q)}U$ for all $q$ (note that the $\om$-limit sets belong to the {\em frontier} of the fixed point sets and cannot enter the interior of the fixed point set; hence the $\om$-limit sets cannot have interior...). It then remains to show that $x$ has irrational rotation number in some $U\in \A(q)$, for some $q$. For this, one could use the techniques in the proofs of Theorems 3.2 and 1.2 in \cite{paper ergodic}: the rationality of the rotation number generates periodic points, which lead to smaller invariant annuli containing the orbit of $x$.

 Another possible approach is to directly consider the components $A$ of the set $\Int\overline{\O_\psi(x)}\neq \emptyset$. By recurrence and $\mu(A)>0$, we find $q\geq 1$ with $\psi^q(A)=A$ and $\Int\overline{\O_\psi(x)} = \cup_{i=0}^{q-1} \psi^i(A)$. If $A$ is simply connected, then $A$ is either the whole $\s$ or an open disc. In the first case, note that $\psi^q\in \text{Homeo}_+(\s)$ has at least two disjoint fixed points $x_0,x_1$ (Theorem 3.10 in \cite{bonino} plus Brouwer's fixed point theorem); here we set $U=A-\{p_0,p_1\}$. In the second case, we find a fixed point $x_0\in A$ of $\psi^q$ and set $U=A-\{x_0\}$. If $A$ is an annulus, there is nothing to prove. If $A$ has $\geq 3$ holes, then one can possibly use the invariant annuli of \cite{Fr-Ha} to further decompose the set $A$ (in each disc-like hole, one could expect to find periodic points, such that boundaries of invariant annuli have to cut through $A$, while by assumption $\psi|_A$ is transitive). More precisely: use Theorem (3.4) in \cite{franks s2} to obtain periodic points in $A$ (the rotation number of some iterate $\psi^q$ should vanish identically).
\end{proof}

\section{Miscellaneous results}\label{section misc}

\subsection{Invariant graphs and contractible geodesic loops}\label{section birkhoffs thm}

A classical theorem of Birk\-hoff states that in a region of instability of a monotone twist map of an annulus, one can connect the boundaries of that region by an orbit of the twist map. There is a well-known relationship between monotone twist maps and geodesic flows on $\T$, see Section \ref{section twist map} below. One of the main differences is (informally) that geodesics in $\T$ might be recurrent in the universal cover $\R^2$. Translated to geodesic flows on $\T$, we propose the following version of Birkhoff's theorem and will sketch a proof for reversible Finsler metrics.

\begin{prob}\label{quest birkhoffs thm}
 Let $F$ be a Finsler metric on $\T$, such that there are no continuous $\phi_F^t$-invariant graphs in $S_F\T$. Does there exist a simple closed $F$-geodesic in the universal cover $\R^2$?
\end{prob}

Conversely, if $c:\R/T\Z\to\R^2$ is a simple closed $F$-geodesic, then the velocity curve $\dot c:\R/T\Z\to S_F\R^2$ lies in one of the prime, non-trivial homotopy classes of $S_F\R^2 \cong \R^2 \times \se$ (this is known as the Umlaufsatz due to Hopf). In particular, the image $\dot c(\R/T\Z)$ intersects any $C^0$-graph in $S_F\R^2$, so that there cannot be a continuous $\phi_F^t$-invariant graph in $S_F\R^2$, in particular not in $S_F\T$.

Hence, if the answer to Problem \ref{quest birkhoffs thm} is affirmative, then the non-existence of invariant $C^0$-graphs for the geodesic flow $\phi_F^t:S_F\T\to S_F\T$ and the existence of simple closed $F$-geodesics in $\R^2$ would be equivalent. Recall that an invariant $C^0$-graph in $S_F\T$ has to be Lipschitz (a classical result due to Birkhoff, see e.g.\ the appendix in \cite{mather connecting}) and by the arguments in Section \ref{section weak KAM} one can show that it has to be contained in some set $\M(\g)$. In particular, an invariant $C^0$-graph would entail $\pi(\M(\g))=\R^2$ for some $g$-geodesic $\g$.

An intuitive situation with a contractible, simple closed geodesic $c:\R/\Z\to\T$ is given in Proposition 9.7 in \cite{bangert}: A large bump in $\T$ induces a simple closed geodesic encircling the bump.

Let us sketch a proof for an affirmative answer to Problem \ref{quest birkhoffs thm}, under the assumption that $F$ is reversible. The first step is the following lemma resembling Birkhoff's result mentioned earlier. It is proved in the setting of monotone twist maps in \cite{mather connecting} (see in particular p.\ 257), working with a variational principle related to the (symmetric) distance $d_F$ in $\R^2$ via the ideas in \cite{bangert}, see also Subsection \ref{section twist map}. Hence, it is believable that the lemma holds in the setting of $(\T,F)$.

\begin{lemma}\label{lemma mather connecting}
 Let $F$ be a reversible Finsler metric on $\T$ and let $I\subset \se$ be an interval not containing antipodal points. Assume that $\pi(\M(\g))\neq \R^2$ for all Euclidean straight lines $\g$ with direction $\g(\infty)=\xi\in I$. Then for any pair $\xi_-, \xi_+$ of irrational points in $I$ there exists an $F$-geodesic $c:\R\to \R^2$ with $c(t)$ tending to infinity as $|t|\to\infty$, with $c(\R_-)$ in finite distance of $\R_- \xi_-$ and $c(\R_+)$ in finite distance of $\R_+ \xi_+$.
\end{lemma}

The assumption that the points $\xi_-,\xi_+$ are not antipodal makes it easier to adapt the techniques from \cite{mather connecting}, while this assumption should not be necessary. Let us now assume that there are no invariant $C^0$-graphs in $S_F\T$. Lemma \ref{lemma mather connecting} applies to any non-antipodal pair $\xi_-,\xi_+$, so that we are able to construct three geodesics $c_1,c_2,c_3$ as in Figure \ref{fig_birkhoff}. As seen in Figure \ref{fig_birkhoff}, one obtains a closed domain $A\subset \R^3$, whose boundary $\partial A$ consists of parts of the $c_i$ and has three vertices of outer angles $<\pi$. Take now a smooth, simple closed curve $c:\R/\Z\to \R^2 - A$ and apply the curve shortening flow. The curve-shortening flow is discussed in \cite{grayson} and has been extended to the reversible Finsler case in \cite{angenent1}. We obtain a variation $c_t:\R/\Z\to \R^2$ of $c=c_0$ for $t\in [0,\infty)$. No curve $c_t$ can intersect the set $A$, as the complement $\R^2-A\supset c_0(\R/\Z)$ is locally convex. In particular, $c_t$ cannot shrink to a point. It is a property of the curve shortening flow that then the limit $c_\infty=\lim_{t\to\infty} c_t$ has to be a simple closed $F$-geodesic in $\R^2-A$. The claim follows.

\begin{figure}\centering%[!htb]
 \includegraphics[scale=0.3]{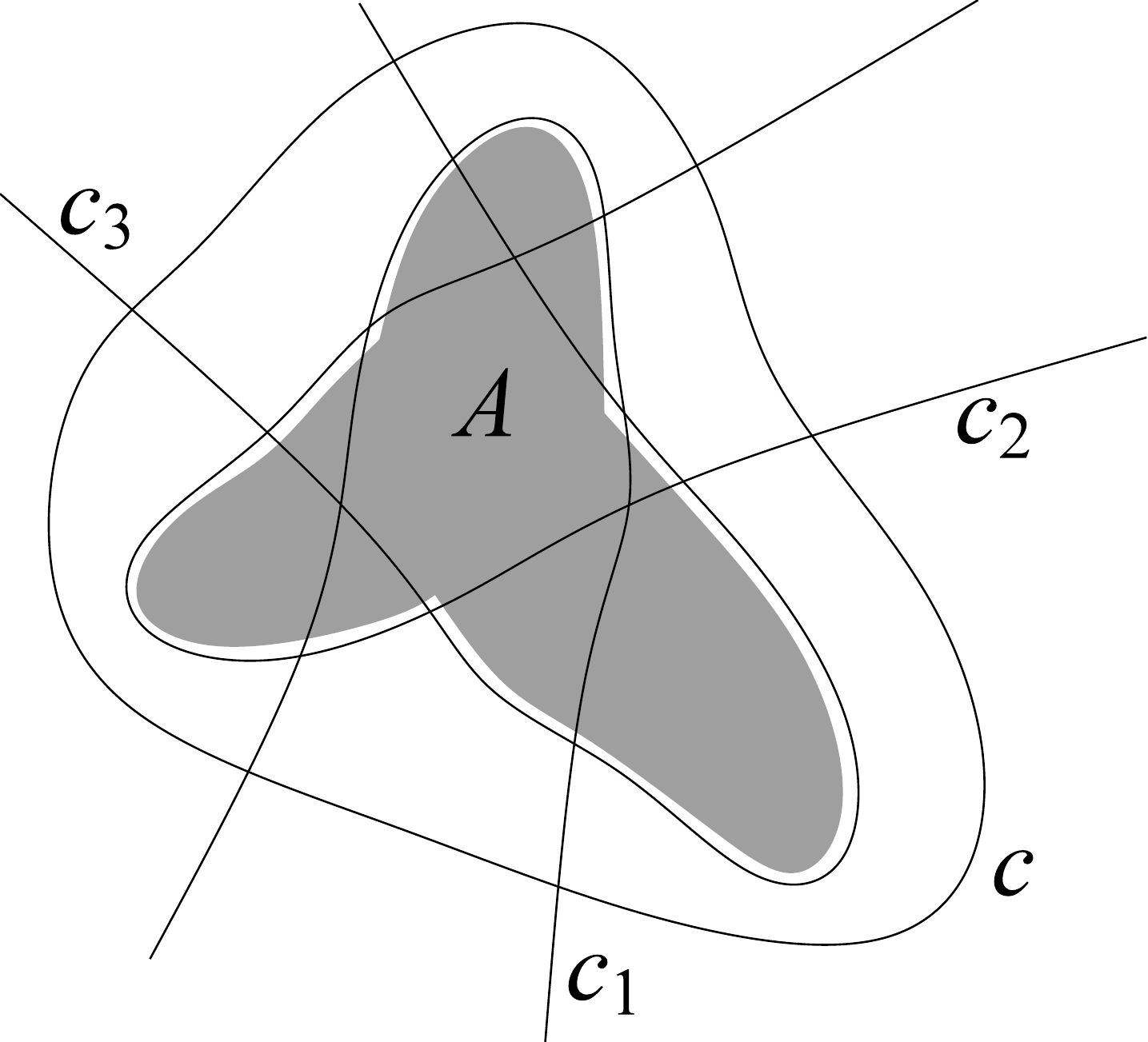}
 \caption{The curves $c_1,c_2,c_2$ in $\R^2$ given by Lemma \ref{lemma mather connecting}, the domain $A$ and the curve $c_0=c$. \label{fig_birkhoff}}
\end{figure}

Possibly, the result has an analogon for higher genus surfaces. Here, one cannot expect invariant graphs in $S_FM$, but in the universal cover we have the candidates $\RR_-(\xi)\subset S_FH$. Gaps in the projections
\[ \pi(\RR_-(\xi))\subset H \]
could be used to construct a geodesic $c:\R\to H$ with a self-intersection. The techniques in \cite{diss eva}, (Section 3 of) \cite{glasm1} and \cite{glasm2} yield a simple closed geodesic in $H$ by the same ideas as above.

\subsection{Surfaces of non-positive curvature}

Recall the flat strip theorem in Subsection \ref{section hadamard}. An application of Theorem \ref{thm rays-paper} yields the following.

\begin{cor}
 If $M=H/\Gam$ is a closed, orientable surface of genus $\mathfrak{g}\geq 2$ and $F$ is Riemannian with non-positive curvature, then flat strips $S\subset H$ are periodic, i.e.\ $S$ is invariant under some $\tau\in \Gam-\{\id\}$ and foliated by $\tau$-periodic $F$-geodesics.
\end{cor}

This result was obtained earlier by Coudene and Schapira \cite{coudene-schapira}. It is related to the following, long-standing open problem.

\begin{prob}\label{problem ergodic non-positive}
 If $M=H/\Gam$ is a closed, orientable surface of genus $\mathfrak{g}\geq 2$ and $F$ is Riemannian with non-positive curvature, is the geodesic flow $\phi_F^t:S_FM\to S_FM$ ergodic with respect to the Liouville measure?
\end{prob}

It was proposed to the author by Knieper to use the ideas from \cite{min_rays} leading to Theorem \ref{thm rays-paper} in order to show that for any orthogonal Jacobi field $J$ along a non-periodic geodesic we have
\[ \liminf_{t\to\infty}\|J(t)\|=0 . \]
This would be a key step to prove the ergodicity in Problem \ref{problem ergodic non-positive}. In this connection, one can prove the following theorem. It is a generalization (and correction) of Lemma 3.7 in \cite{wu}. The proof below was explained to the author by Knieper.

\begin{thm}\label{thm wu-knieper}
 Let $M$ be a closed orientable surface of genus $\mathfrak{g} \geq 2$ and let $F$ be Riemannian with non-positive curvature $K_F\leq 0$. Consider the set
 \[ A := \{v\in S_FM : K_F\circ c_v \equiv 0 ~ \& ~ c_v \text{ not periodic} \}. \]
 Then $A$ is closed.
\end{thm}

It would be desirable to know whether $A$ has zero Liouville measure.

\begin{proof}
 We work in $H$, writing $\tilde A=Dp^{-1}(A)$. Consider $v_n\to v$ with $v_n\in \tilde A$. It follows that $K_F\circ c_v\equiv 0$, so we have to show that $c_v$ is not periodic. Suppose $c_v$ is periodic. As $c_{v_n}\neq c_v$, we find $x_n,y_n\in c_v(\R)$, such that both $x_n,y_n$ have distance to $c_{v_n}$ at most 1, while we assume
 \[ d_g(x_n,c_{v_n}(\R)) \to 0, \qquad d_g(y_n,c_{v_n}(\R))=1 . \]
 Such points $x_n,y_n$ exist by the assumed convergence, and moreover we have $d_g(x_n,y_n)\to\infty$. W.l.o.g., $c_v$ is chosen in its homotopy class such that on the side of $c_v$ where $c_{v_n}$ passes $y_n$ in distance $1$, there are no more periodic geodesics with the same homotopy class as $c_v$ (here one uses that $c_{v_n}$ approximate every closed geodesic in the homotopy class of $c_v$, if $c_v$ is part of a flat strip). We can use the translation along the periodic $c_v$ to transport $y_n$ into a compact set and obtain from $c_{v_n}$ in the limit a geodesic $c$ asymptotic to $c_v$ with $K_F\circ c\equiv 0$.
 
 We now argue that there cannot be a pair of asympotic geodesics $c,c'$ in the universal cover of $M$, such that $K_F\circ c=K_F\circ c'=0$ and such that $c$ is invariant under some deck transformation $\tau$. Suppose w.l.o.g., that $\tau c'$ lies between $c,c'$ and let $\al$ be a geodesic segment connecting $c$ to $c'$. Then the area of the ideal triangle $\Delta$ bounded by $c',\tau c'$ and $\al$ is bounded: $\Delta$ is the disjoint union of its parts $D_k$ between $\tau^{-k-1}\al$ and $\tau^{-k}\al$ for $k\geq 0$. Under $\tau^{k}$, the $D_k$ are moved into disjoint subsets of the compact set bounded by $c,c',\al$ and $\tau^{-1}\al$. But the main result in \cite{ruggiero} states that by $K_F\circ c'=0$, the triangle $\Delta$ has infinite area. This is a contradiction.
\end{proof}

The main ingredients in the above proof are
\begin{enumerate}
 \item a certain ``simplicity'' of the limit geodesics (given by periodicity),
 
 \item the main result in \cite{ruggiero}, which holds for arbitrary ideal triangles.
\end{enumerate}
Possibly, one can replace the periodicity in the above proof by other properties of the limit geodesic and argue that the set $A$ is in fact empty. For instance, recall that a set of geodesics $A_0$ is called simple, if the geodesics in $A_0$ do not intersect in the closed surface $M$. If the set $A$ in the above theorem contains a simple closed subset, then we would again reach a contradiction, as we could project the obtained ideal triangles to $M$, where the total area is bounded.

\section{Dynamical systems related to geodesic flows}\label{section rel dynamical systems}

Geodesic flows, especially in the Riemannian case have been studied for a long time. In this section, we discuss three dynamical systems, which are related to Finsler geodesic flows.

\subsection{Monotone twist maps}\label{section twist map}

Let us write $C=\R/\Z\times \R$ for the infinite cylinder. Monotone twist maps $f:C\to C$ we studied first by \Poincare, later by Birkhoff. In the 1980ies the topic became more present again through the development of Aubry-Mather theory. A monotone twist map is a diffeomorphism that leaves invariant both ends of $C$, preserves the standard area and orientation of $C$ and a lift the the universal cover $\tilde f:\R^2\to\R^2$ maps vertical lines $\{x\}\times \R$ to graphs over the first component $\R$. An introduction to the topic is given e.g.\ in \cite{gole book}.

Given a monotone twist map $f$ of the cylinder $C$, one can find a Lagrangian $L:\R/\Z\times C \to \R$, strongly convex in the $\R$-component of $C$, so that the time one map of the Euler-Lagrange flow is $f=\phi^1_L|_{\{0\}\times C}$. See \cite{moser1}. The Lagrangian $L$ defines an action $A_L(c)=\int L(t,c,\dot c)dt$ of curves $c:[a,b]\to\R/\Z$. It is shown in Appendix A of \cite{diss} how to define a reversible Finsler metric $F$ on $\T$, so that the curves $(t,c)$ are $F$-geodesics, where $c$ is a critical curve for $A_L$. In this sense, monotone twist maps are a special case of Finsler geodesic flows on $\T$. 

Conversely, if $F$ is a Finsler metric on $\T$ and $L:\R/\Z\times C\to\R$ is defined by $L(x,r) := F(x,1,r)$, where $x\in\T,r\in\R$, then the $F$ geodesics can be reparametrized into $A_L$-critical curves. This encoding of the geodesic flow $\phi^t_F$ in the Euler-Lagrange $\phi_L^t$ of $L$ is related to constructing Poin\-car\'e surfaces of section for $\phi^t_F$ in $S_F\T$, namely consider
\[ N_t := \{ (x,v)\in S_F\T : x_1=t \} , \quad t\in \R/\Z. \]
Set $\Xi := \{ (x,v) \in S\T : v_1>0 \}$, then in $\Xi$ the geodesic flow of $F$ is transverse to all $N_t$ and $\phi^t_F$ can be reparametrized into $\phi^t_L$. The first return map 
\[ \varphi  : N_0 \to N_0 \]
is the time-1-map $\phi^1_L|_{\{0\}\times C}$ (putting questions as to where $\vf$ is defined aside). One can show that $\vf$ is a monotone twist map, if the geodesics in the domain of definition have no conjugate points before returning to $N_0$. In general, $\vf$ can be thought of as a composition of monotone twist maps.

Monotone twist maps come with a generating function $h:\R\times \R\to \R$. Letting $F$ be the reversible Finsler metric associated to $f$ as above, the function $h$ is related to the distance $d_F$ of $F$ on $\R^2$ via
\[ h(s,t) = d_F((0,s),(1,t)) . \]
This is the approach in \cite{bangert}. In this way, minimal geodesics for $F$ occur as action minimizers for $f$ and vice versa. See also Lemma 4.2.5 in \cite{siburg}: a generating function of the first return map $\vf$ to $N_0$ above is given by the first-return time to \Poincare sections, while in Finsler geodesic flows, time equals length equals distance. 

Using the above links, results on monotone twist maps correspond to results on geodesic flows on $\T$. However, there are differences. For instance, let $c:\R\to \R/\Z$ be $A_L$-critical and $(t,\tilde c(t))$ be the corresponding geodesic of the associated Finsler metric on $\T$, lifted to $\R^2$. Then $(t,\tilde c(t))$ cannot be recurrent in $\R^2$, where as Finsler metrics are allowed to have recurrent geodesics in the universal cover. This difference also appeared in Subsection \ref{section birkhoffs thm}.

\subsection{Tonelli Lagrangians}\label{section tonelli}

Given a Finsler metric $F$ on a closed manifold $M$, we can consider the kinetic energy given by $L_F=\frac{1}{2}F^2$. Apart from smoothness at the zero section, this is a good example of a Tonelli Lagrangian. Other classical examples are of the form
\[ L(v) = \frac{1}{2}|v|_g^2 + \eta(v) + U(\pi(v)) , \]
where $g$ is some Riemannian metric, $\eta$ a 1-form and $U:M\to \R$ some function on $M$.

\begin{defn}
 A {\em Tonelli Lagrangian} is a function $L:TM\to\R$, which is superlinear and strongly convex in the fibers $T_xM$ of $TM$. A {\em Tonelli Hamiltonian} is a function $H:T^*M\to\R$ with the analogous properties.
\end{defn}

Given a Tonelli Lagrangian, its Fenchel transform is a Tonelli Hamiltonian, in much the same way that we associated to a Finsler metric $F$ its dual Finsler metric $F^*$ in Section \ref{section weak KAM}. Basically all of the theory that we developed here has an analogon for Tonelli Lagrangian systems. We refer e.g.\ to \cite{sorrentino} and \cite{fathi} and the referenes therein. Just note for the moment, that $L$ defines an action $A_L(c)=\int L(\dot c)dt$ and an Euler-Lagrange flow $\phi_L^t:TM\to TM$ (for $L=L_F$ this is the geodesic flow). The Euler-Lagrange flow leaves the energy $E_L$ invariant, given by
\[ E_L := \frac{d}{dt}\bigg|_{t=0}L(v+tv) - L(v) . \]
It is thus natural to restrict the flow $\phi_L^t$ to $\{E_L=e\}\subset TM$. For $L=L_F$ and $e=1/2$, this is the unit tangent bundle $S_FM$.

In Subsection \ref{section av action} we already discussed Mather theory for $L=L_F$. The same works in the general case, i.e.\ we have Mather's average actions $\al_L,\beta_L$. A particularly interesting value is given by
\[ c_0(L) := \min \al_L , \]
called {\em \mane's strict critical value}. This minimum exists, as $\al_L$ is convex and superlinear. For energies $e>c_0(L)$, we consider {\em \mane's potential}
\[ \Phi_{L+e}(x,y) := \inf\left\{ A_L(c) ~|~ c:[0,T]\to H, c(0)=x,c(T)=y, T>0 \right\} \]
in the universal cover $H$ of $M$. The generalizes the Finsler distance $d_F$ in $H$ for $L=L_F$. The action minimizers $c:\R\to H$ with respect to $L+e$, called {\em $e$-minimizers} will have energy $e$, so that $\Phi_{L+e}$ highlights the energy level $\{E_L=e\}$. Using \mane's potential as a distance, one can develop all the theory in this paper.

There is a more direct link to Finsler metrics, which we recall now, see \cite{CIPP}.

\begin{thm}\label{thm CIPP}
 For $e>c_0(L)$, there exists a Finsler metric $F$ on $M$, so that the Euler-Lagrange flow $\phi_L^t|_{\{E_L=e\}}$ is conjugated to the geodesic flow $\phi_F^t|_{S_FM}$. \mane's potential in $H$ is then given by $\Phi_{L+e}(x,y)=d_F(x,y)+f(y)-f(x)$ for a fixed function $f:H\to\R$. In particular, $e$-minimizers of $L$ correspond to minimal $F$-geodesics.
\end{thm}

One idea for the proof is to use the associated Tonelli Hamiltonian $H:T^*M\to\R$ and the characterization of $\al_L$ given in Proposition \ref{prop char al-fctn via weak KAM} to show that the energy level $\{H=e\}$ is a sphere bundle. Then take the dual Finsler metric $F^*:T^*M\to\R$ with $F^*|_{\{H=e\}}=1$. Both Hamiltonians $H$ and $\frac{1}{2}(F^*)^2$ share a level set, which means that the flows are conjugated.

\subsection{Symplectic geometry}\label{section symplectic}

In the previous subsection, we generalized from Finslerian geodesic to Euler-Lagrange flows of Tonelli Lagrangians. Theorem \ref{thm CIPP} showed that for sufficiently large energies, the result was again a Finsler geodesic flow of some kind. A strictly more general setting is that of contact manifolds and Reeb flows in symplectic geometry. The understanding of minimal geodesics, being fundamental to the global behavior of geodesic flows, should eventually lead to results on the more general class of Reeb flows; this also concerns the question as to which phenomena are geometric and which are symplectic in nature.

Let us remark in this connection that it was shown by Bernard \cite{bernard}, that Mather, \Mane and Aubry sets (see Section \ref{section mather theory}) behave well under symplectomorphisms. Moreover, observe that Proposition \ref{prop char al-fctn via weak KAM} leads already to a Hamiltonian characterization of the set $B^*(F)=\{\al_F\leq 1/2\}$, generalized by the symplectic shape of energy sublevel sets $\{\frac{1}{2}(F^*)^2\leq 1/2\}$ in Chapter 6 of \cite{siburg}. Note that weak KAM solutions defining exact Lagrangian graphs are also known as generating functions in symplectic geometry. Much of the theory in this paper could be translated into these more general settings, and a lot is still under development.

Let us point out a more concrete question, as to what phenomena could be generalized from the setting of geodesic to Reeb flows. This concerns the results on geodesic flows on $\T$ with vanishing entropy, see Subsection \ref{section FrHa}.

\begin{prob}
 If $\phi^t:\mathbb{T}^3\to\mathbb{T}^3 = \T\times \R/\Z$ is a Reeb flow associated to a tight contact form of degree $k\geq 1$ (see the classification of tight contact structures by Yutaka in \cite{yutaka}: the contact structure is tangent to the fibers $\{x\}\times \R/\Z$ and winds around these fibers $k$ times) with $\h(\phi^t,\mathbb{T}^3)=0$, do there exist for each $\xi\in \se$ $k$ thickened invariant ``graphs'' (i.e.\ homologous to the ``base'' $B := \T\times \{0\}$) $T_{\xi,i}, i=1,...,k$, such that the orbits in $T_{\xi,i}$, when projected to $B$, wind around $B$ with asymptotic direction $\xi$ (obtaining a structure similar to Figure \ref{fig_htop_0})?
\end{prob}

\begin{proof}[Ideas for the proof]
 1. The contact structure looks in some sense like the spherization (for $k=1$), i.e. fillable, such that contact homological techniques could be used to show that for $i=1,...,k$ and each homotopy class $z\in \Z^2$ of $B$ there exists a closed orbit $\g_{z,i}$ with homotopy class $z$ (in fact: two, one minimum and one minimax of the action functional), when projected to $B$. The idea for $k=1$ is that the contact structure looks like the Liouville form in $T^*\T$ restricted to the boundary of a fiberwise starshaped set. The contact homology then ``fibers through the contact homology of two contact homologies of flat Riemannian metrics on $\T$'', cf. \cite{schlenk}. These closed orbits give a first ``frame'' for invariant sets.

 2. Generalize the results of Franks and Handel \cite{Fr-Ha} to this situation (Reeb or more general flows). A key technique might be the notion of LeCalvez' maximally unlinked sets of closed orbits corresponding to the fixed-point sets in \cite{Fr-Ha}. Namely all dynamics wind non-trivially around maximally unlinked sets by definition. Possibly one can then show as in \cite{Fr-Ha} that there exists a continuous integral of the flow given by the homological rotation vector in the base $B$.
\end{proof}

\end{document}